\documentclass[11pt]{article}
\usepackage{fullpage}

\usepackage{amsmath,amssymb,theorem, comment}
\usepackage{graphicx}
\usepackage{verbatim}

\usepackage{hyperref}
\hypersetup{
    colorlinks=true,
    linkcolor=blue,
    citecolor=red,
    urlcolor=blue,
    pdfborder={0 0 0}
}

\usepackage{color}
\usepackage{makeidx}

\theorembodyfont{\upshape}

\usepackage{makeidx}

\theorembodyfont{\upshape}

 \makeatletter
 \@addtoreset{equation}{section}
 \makeatother

 \newtheorem{ittheorem}{Theorem}
 \newtheorem{itlemma}{Lemma}
 \newtheorem{itproposition}{Proposition}
 \newtheorem{itdefinition}{Definition}

 \newtheorem{itremark}{Remark}
 \newtheorem{itclaim}{Claim}
 \newtheorem{itcorollary}{\bf Corollary}

 \newenvironment{theorem}{\addtocounter{equation}{1}
 \begin{ittheorem}}{\end{ittheorem}}

 \newenvironment{lemma}{\addtocounter{equation}{1}
 \begin{itlemma}}{\end{itlemma}}

 \newenvironment{proposition}{\addtocounter{equation}{1}
 \begin{itproposition}}{\end{itproposition}}

 \newenvironment{remark}{\addtocounter{equation}{1}
 \begin{itremark}}{\end{itremark}}

 \newenvironment{proof}{\noindent {\bf Proof.\,}
 }{\hspace*{\fill}$\qed$\medskip}

 \newenvironment{corollary}{\addtocounter{equation}{1}
 \begin{itcorollary}}{\end{itcorollary}}

\newcommand{\din}{\partial^{in}}
\newcommand{\dout}{\partial^{\, out}}

\newcommand{\doutie}{\partial^{\, out,ext}_{\infty}}
\newcommand{\dini}{\partial^{\, in}_\infty}
\newcommand{\douti}{\partial^{\, out}_\infty}
\newcommand{\doute}{\partial^{\, out,ext}}

\newcommand{\ex}{{\exists\, }}

\newcommand{\set}[2]{\big\{\, #1 :  #2\, \big\}}

\newcommand{\tatop}[2]{\genfrac{}{}{0pt}{1}{#1}{#2}}

\newcommand{\XX}{{\underline{X}}}
\newcommand{\xx}{{\underline{x}}}

\newcommand{\yy}{{\underline{y}}}

\def\T{\mathcal{T}}
\def\S{\mathcal{S}}
\def\lsim{\mathop{\sim}^{\ln}_{n\to\infty}}
\def\bD{\overline{D}}
\def\supp{\text{supp\,}}
\def\afm{\overline{|f|\,}^{\lower2pt\hbox{$\scriptstyle M$}}}
\def\fm{\overline{f\,}^{\lower2pt\hbox{$\scriptstyle M$}}}
\def\fmn{\overline{f_N\,}^{\lower2pt\hbox{$\scriptstyle M$}}}
\def\fmne{^{\displaystyle \eta}\overline{f_N\,}^{\lower2pt\hbox{$\scriptstyle M$}}}
\def\pfm{\overline{\psi f\,}^{\lower2pt\hbox{$\scriptstyle M$}}}
\def\phfm{\overline{\phi_X f_N\,}^{\lower2pt\hbox{$\scriptstyle M$}}}

\def\tD{\widetilde{D}}
\def\tD{{D}}

\def\uro{\smash{{U}^{\!\!\!\!\raise5pt\hbox{$\scriptstyle o$}}}}

 \def \ba {\begin{array}}
 \def \ea {\end{array}}

 \def \qed {{\heartsuit\hfill}}
 \def \Z {{\mathbb Z}}
 \def \Zd {{\mathbb Z^d}}
 \def \Rd {{\mathbb R^d}}
 
 \def \Znd {{\mathbb Z_n^d}}
 \def \R {{\mathbb R}}
 
 \def \N {{\mathbb N}}
 \def \P {{P}}
 \def \E {{E}}

 \def \cE {{\cal E}}
 \def \cF {{\cal F}}
 \def \cA {{\cal A}}

 \def \cR {{\cal R}}

\def\wP{\smash{\widetilde P}}
\def\wh{{\tilde h}}

\def\hX{\smash{\widehat X}}

\def\hp{{\widehat \phi}}
 
 \def \cG {{\cal G}}

 \def \cX {{\cal X}}
 \def \cB {{\cal B}}

 \def \cL {{\cal L}}
 
 \def \baf {{\overline f}}
 \def \bag {{\overline g}}

\def \qed {{\square\hfill}}

\def\ft{\widetilde f}
\def\hht{\smash{\widetilde h}}
\def\gt{\widetilde g}
 
 \def \La {{\Lambda}}

\newcommand{\Ld}{{\mathbb{L}}^d}
\newcommand{\Ldi}{\smash{{\mathbb{L}}^{d,\infty}}}
\newcommand{\Edi}{{\mathbb{E}}^{d,\infty}}

\newcommand{\Ed}{{\mathbb{E}}^d}

\newcommand{\dinf}{\, {\rm d}_{\infty}\, }

\def\cA{{\cal A}} \def\cB{{\cal B}}  
\def\cE{{\cal E}} \def\cF{{\cal F}} \def\cG{{\cal G}} 
   \def\cL{{\cal L}}
   
 \def\cR{{\cal R}}  

   \def\cX{{\cal X}}

\def \qed {{\square\hfill}}

\def\R{{\mathbb R}}
\def\Q{{Q}}
\def\N{{\mathbb N}}

\def\eqref#1{(\ref{#1})}

\DeclareMathOperator{\dist}{dist}

\def\cX{\mathcal{X}}

\def\cR{\mathcal{R}}

\def\cL{\mathcal{L}}

\def\cG{\mathcal{G}}
\def\cF{\mathcal{F}}
\def\cE{\mathcal{E}}

\def\cB{\mathcal{B}}
\def\cA{\mathcal{A}}

\def\eps{\varepsilon}

\makeindex

\begin{document}

\title{The random walk penalised by its range\\
in dimensions $d\geq 3$}

 \author{
\begin{tabular}{c c }
\text{Nathana\"el Berestycki} &
\text{\qquad Rapha\"el Cerf\footnote{
\noindent
DMA, Ecole Normale Sup\'erieure,
CNRS, PSL Research University, 75005 Paris.}
\footnote{
\noindent Laboratoire de Math\'ematiques d'Orsay, Universit\'e Paris-Sud, CNRS, Universit\'e
	Paris--Saclay, 91405 Orsay.}}\\
\text{Universit\"at Wien\footnote{On leave from the University of Cambridge}} &
\text{ {\'E}cole Normale Sup\'erieure }
\end{tabular}
}

\maketitle



\begin{abstract}
\noindent We study a self-attractive random walk such that each trajectory of length $N$ is penalised by a factor proportional to $\exp ( - |R_N|)$, where $R_N$ is the set of sites visited by the walk. We show that the range of such a walk is close to a solid Euclidean ball of radius approximately $\rho_d N^{1/(d+2)}$, for some explicit constant $\rho_d >0$. This proves a conjecture of Bolthausen \cite{BO} who obtained this result in the case $d=2$.
\end{abstract}

\tableofcontents

\section{Introduction}

\subsection{Main results}

Let $P$ be the law of the
discrete-time simple random walk $(S_n)_{n\in\N}$ on $\mathbb{Z}^d$, $d\geq 1$, starting from the origin.
Let $N$ be a positive integer. In \cite{BO}, Bolthausen proposed the following model for a self-attractive random walk: let us denote by $R_N$ the set of the points visited by the random walk until time $N$
and by $|R_N|$ its cardinality.
We define a new probability on the set of the $N$--steps trajectories
by setting
\begin{equation}
\label{defPN}\frac{d\wP_N}{dP}\,=\,\frac{1}{Z_N}\exp\big(-|R_N|\big)\,,
\end{equation}
where the normalization factor (or partition function) $Z_N$ is given by
\begin{equation}\label{partitionfunction}
Z_N\,=\,E\big(
\exp\big(-|R_N|\big)\big).
\end{equation}
Clearly, $\wP_N$ favours configurations where the trajectory is localised on a small number of points. Bolthausen asked what can be said about a typical realisation of $\wP_N$. The question is particularly natural from the point of view of large deviations theory. Indeed one of the early successes of the theory, due to Donsker and Varadhan \cite{DV1}, was a determination of the first order asymptotics of the partition function $Z_N$:
\begin{equation}\label{DV}
  Z_N = \exp \left( - (1+ o(1)) \chi_d
	N^{d/{d+2}}
	\right),
\end{equation}
for some $\chi_d>0$ depending only on the dimension.
 Bolthausen was able to show that in dimension $d=2$, under $\wP_N$ a typical trajectory localises on a Euclidean ball of radius approximately $\rho_2 N^{1/4}$ for some constant $\rho_2>0$. His analysis strongly suggests that in general dimensions $d \ge 3$, a similar result holds except that the walk now localises on a ball of radius approximately $\rho_d N^{1/ (d+2)}$, where $\rho_d>0$ is a specific constant depending only on the ambient dimension $d$.

 \medskip The main goal of this paper is to verify Bolthausen's conjecture. Bolthausen actually provided support for his conjecture by showing that two (admittedly crucial) estimates implied the conjecture in general dimension $d \ge 2$; these two estimates were in turn proved for $d = 2$. The two theorems below provide a proof of these two estimates in the general case $d \ge 3$, thereby proving as a corollary of Bolthausen's paper \cite{BO} that his conjecture is true and hence completing his programme. We now state below these two results from which Bolthausen's conjecture follows.

\medskip We suppose without loss of generality that
$$n\, =\, N^{1/(d+2)}$$ is an integer. We define the
local time $L_N$ as
\begin{equation*}
\forall x\in\Zd\qquad L_N(x)\,=\,
\sum_{k=0}^{N-1}1_{\{\,S_k=x\,\}}\,.
\end{equation*}
We define the continuous rescaled version $\ell_N$ of $L_N$ by
 \begin{equation}\label{D:rescaledLN:intro}
\forall x\in\Rd\qquad \ell_N(x)\,=\,
\frac{n^d}{N}{L_N\big(\lfloor nx\rfloor\big)}\,.
\end{equation}
For $x \in \R^d$, we let $\phi_x$ be the principal eigenfunction (normalised so that $\|\phi_x\|_2^2 := \int_{\R^d} \phi_x^2 = 1$) of $- \Delta$ in $B(x, \rho_d)$, with Dirichlet boundary conditions.
We denote by
$L^1(\R^d) $ the set of the integrable Borel functions on $\R^d$ and we use the standard norm:
$$\forall f \in L^1(\R^d)\qquad \|f\|_1\, =\, \int_{\R^d} |f(x)|dx\,.$$
The first result below is a quantitative shape theorem in the $L^1$ sense for the local time profile at time $N$.

\begin{theorem}\label{T:Prop31_intro}
Let $\cL_n$ be the set of
functions
	defined by
$$\cL_n\,=\,\big\{\,
\ell\in L^1(\Rd):
||\ell||_1=1,\, \ell\geq 0,\,
\inf_{x\in\Rd}\,||\ell-(\phi_x)^2||_1\,\geq 1/n^{1/800}\,\big\}\,.$$
For $n$ large enough, we have
\begin{equation}
\label{ouf:intro}
E\big(e^{-|R_N|};
\ell_N\in\cL_n)\,\leq\,
\exp\Big(-
n^d\,
	\chi_d -  n^{d-\frac{1}{17}}\Big)
\end{equation}
where $\chi_d$ is as in \eqref{DV}.
\end{theorem}
\noindent
The upper bound obtained in formula~\eqref{ouf:intro} is in fact negligible compared to the partition function~$Z_N$. This is a consequence
of another result of \cite{BO} which is recalled in proposition~\ref{P:lbZ}.

The second main result says that
if $x$ is such that $\ell_N$ is close to $(\phi_x)^2$ in the $L^1$ sense above, then actually almost all of the ball of radius $\rho_d n$ around $x$ has been filled by the range of the walk. More precisely, set
$$
\cG_{loc,x}\,=\,\Big\{\, \|\ell_N - (\phi_x)^2 \|_1 \le \frac{1}{n^{1/800}}\,\Big\}\,.
$$
For $\kappa>0$ and $x \in \R^d$, we define
$$
\cR_{\kappa,x} = \big\{\forall  z \in B(x, \rho_d n (1 - n^{- \kappa}) ),\quad \ell_N(z)>0 \big \}.
$$
This is the event that the ball of radius $\rho_d n (1 - n^{-\kappa})$ around $x$ is filled.

  \begin{theorem}\label{T:filledball_intro} There exists $\kappa>0$ such that, for any $a>0$,
  $$
  \frac1{Z_N}\E( e^{ - |R_N|} ; \cG_{loc,x} ; (\cR_{\kappa,x})^c) = o(N^{-a}).
  $$
\end{theorem}

\noindent
Together with Bolthausen's results \cite{BO} (see p.877, immediately below Conjecture 1.3), Theorems \ref{T:Prop31_intro} and \ref{T:filledball_intro} immediately imply the following, which is the main conclusion of our paper.

\begin{corollary}\label{T:main}
Let us denote
	by $B(x,r)$ the $d$ dimensional Euclidean ball centered at $x$ of radius $r$.
There exists a positive constant $\rho_d$, which depends only on the dimension $d$, such that,
for any $\varepsilon>0$, as $N \to \infty$,
$$
\wP_N\Big(\exists x\in\Rd\quad
	B\big(x,\rho_d(1-\varepsilon)N^{\frac{1}{d+2}}\big)
\cap\Zd
\subset R_N
\subset
B\big(x,\rho_d(1+\varepsilon)N^{\frac{1}{d+2}}\big)\Big)\longrightarrow 1\,.$$
\end{corollary}

\medskip We mention here a result obtained along the way, which we feel is interesting in its own right. This is a Donsker--Varadhan large deviation estimate which is valid for the random walk
in the full space $\mathbb{Z}^d$, and so bypasses the assumption of compactness for the state space which underlies \cite{DV1}. Let $D$ be an arbitrary finite subset of $\Zd$. For $t\in \N$, we define
$$\tau(D,t)\,=\,\inf\,\big\{\,k\geq 1:L_k(D)=t\,\big\}\,$$
and we set
$$\forall x\in D\qquad
L_t^D(x)\,=\,L_{\tau(D,t)}\,=\,\sum_{k=0}^{\tau(D,t)-1}1_{\{\,S_k=x\,\}}\,.$$
The function $L_t^D$ is a function from $D$ to $\N$. We shall
work in the functional space $\ell^1(D)$ equipped with the norm
$$\forall f\in\ell^1(D)\qquad
||f||_{1,D}\,=\,\sum_{x\in  D}|f(x)|\,.$$
For a function $f: \Z^d \to \R$, we define
\begin{equation}
\label{defE}
\cE(f,D) = \frac1{2d} \sum_{y, z \in D: | y -z| = 1} (f(y) - f(z))^2.
\end{equation}
We define $\bD$ as
$$\bD\,=\,D\cup\{\,x\in\Zd:\exists\, y\in D\quad |x-y|=1\,\}\,.$$
When $D = \Z^d$ we will simply write $\|f\|_1$ and $\cE(f)$.
\begin{theorem}
\label{GD}
Let $C$ be a closed convex subset of $\ell^1(D)$.
For any $t\geq 1$, we have
$$\inf_{x\in\bD}\,P_x\Big(\frac{1}{t}L_t^D\in C,\,
\tau(D,t)<\infty\Big)\,\leq\,
\exp\Big(
-t\,\inf_{h\in C}\,\frac{1}{2}\cE(\sqrt{h},D)
\Big)
\,.$$
\end{theorem}

\subsection{Heuristics}

We begin a discussion of the above results (such as Corollary \ref{T:main}) with a rough heuristics explaining where the limit shape comes from.
Note that a random walk will stay in a box of diameter $n$ for duration $N$ with probability approximately $\exp( - O( N/n^2))$, since it may leave this box with positive probability every $n^2$ units of time. On the other hand, the energetic contribution to \eqref{partitionfunction} of such configurations is of order $\exp ( - n^d)$. Balancing entropy and energy, we expect that the trajectories that contribute most to \eqref{partitionfunction} are such that $n \asymp N^{1/(d+2)}$, which explains the scaling in Corollary \ref{T:main}. We now discuss this in a bit more detail, but still ignoring many technical details.

If $U\subset \R^d$ is an open bounded subset, and $n = N^{1/ (d+2)}$ is as above, then the probability for the random walk to remain in $nU$ for a long time $N$ is approximately $\exp ( - \lambda_U N/(2n^2))$, where $\lambda_U$ is the principal eigenvalue of $- \Delta$ in $U$ with Dirichlet boundary conditions on $\partial U$. Hence the contribution to \eqref{partitionfunction} coming from trajectories staying in $nU$ should be well approximated by $\exp ( - n^d (\lambda_U/2 + |U| ))$, where $|U|$ is the Lebesgue measure of $U$. Using the \textbf{Faber--Krahn inequality}, it can be seen that $\inf_U \{\lambda_U/2 + |U|\}$ is attained when $U$ is a Euclidean ball of radius $r$, say. The radius $r$ can then be determined by noting that principal eigenvalues obey diffusive scaling, i.e., $\lambda_{B(r) } = \lambda / r^2$, where $\lambda$ is the principal eigenvalue in the unit ball. Hence, if $\omega_d $ is the volume of the unit ball, we deduce that $r = \rho_d$ is obtained as the minimiser of the following functional:
$$
\rho_d \,=\,\arg \min \Big\{\,\frac{\lambda}{2r^2} + \omega_d r^d:r>0\,\Big\}\,  = \,\left(\frac{\lambda}{d \omega_d} \right)^{1/ (d+2)}\,.
$$
The constant $\rho_d$ is the one which appears in the theorem.
As already mentioned, Bolthausen proved the
corresponding result in two dimensions \cite{BO}.
The main problem to extend Bolthausen's proof to dimensions~$3$ and higher was to extend
Lemma~$3.1$ and Proposition 4.1 from his paper \cite{BO}. The rest of his proof is written for arbitrary dimensions $d\ge 3$ and it is solely the statement of these two lemmas which depend on $d$ being equal to 2 in his proof.

Lemma 3.1 in \cite{BO} can be seen as a quantitative
Faber--Krahn inequality on the torus, saying that if a set $U$ is not far from minimising $\lambda_U/2 + |U|$ then $U$ itself is not far from a Euclidean ball.
Unfortunately, such an inequality is not available yet in dimensions three and higher.
Even an analogue of the quantitative isoperimetric inequality on the torus
has not been proved so far.
Therefore we cannot use the standard projection of the random walk on the torus,
as in Bolthausen's proof.
This creates a serious difficulty. Indeed, as far as the probabilistic estimates
on the local time
are concerned, it is very convenient to work on the torus: the state space of the
walk becomes
compact and one can readily use the classical large deviations estimates of Donsker
and Varadhan \cite{DV1}.
The good news is that a quantitative Faber-Krahn inequality has been proved recently
in $\Rd$ by Brasco, De Philippis and Velichkov \cite{BLGV}.
Ultimately, we are to use this inequality. Therefore we have to deal with the
random walk
in the full space and we cannot afford the luxury of projecting its trajectories
on the torus.
A key point to carry out this program
is to develop the relevant large deviations estimates. Indeed,
the random walk being transient, the classical Donsker--Varadhan theory
cannot be applied directly.

\medskip {\bf Warning.} In the probability literature, one usually works with the half--Laplacian
$\Delta/2$, which is the infinitesimal generator of the Brownian motion. In Bolthausen's paper,
the notation
$\lambda(G)$ corresponds to half of the quantity defined above. We choose here to stick
to the convention employed in the papers on the Faber--Krahn inequality.
\medskip

\subsection{Relation with other works}

A Brownian analogue of Corollary \ref{T:main} was proved in dimension $d=2$ by Sznitman \cite{Sznitman} using the method of enlargement of obstacles. Very briefly, the starting point of this method (adapted to the discrete setting of this paper) consists in viewing the weighted probability measure $\tilde P_N$ as the annealed probability measure corresponding to a random walk in a random medium in which there is an obstacle at every site with probability $1-e^{-1}$. Note that $e^{- |R_N|}$ then corresponds to the annealed probability that the walk has not encountered any obstacle for time $N$.

A refinement of this method enabled Povel \cite{Povel} to establish the same result in dimension $d \ge 3$. It is important to note however that in the continuum, one cannot of course hope that the range of Brownian motion will fill a ball completely -- there will always be small holes. For this reason,
both results in \cite{Sznitman} and \cite{Povel} are restricted to a statement of the so-called \emph{confinement property}, i.e., a statement that the range is contained in a ball of the appropriate radius (corresponding to $R = \rho_d n (1+ \eps)$ in our setup). The question of whether the range will visit any macroscopic ball within this ball of radius $R$ is only addressed tangentially, see e.g. Theorem 4.3 in \cite{Sznitman} and the discussion at the end of Section 1 in \cite{Povel}. Needless to say, the method of enlargement of obstacles is very different from the strategy employed by Bolthausen in \cite{BO}. Curiously, neither \cite{Sznitman} nor \cite{Povel} discuss what their results imply for the discrete case, though both briefly mention the paper \cite{BO}.

\medskip
At the time we were finishing this paper, we learnt of the independent and nearly simultaneous work of Ding, Fukushima, Sun and Xu \cite{DFSX}, who obtained an alternative proof of Corollary \ref{T:main}. In fact, their result implies a more precise control on the size of the boundary $\partial R_N$ under $\tilde P_N$, showing that with high probability, $|\partial R_N| \le (\log n)^c n^{d-1}$ for some $c>0$ and all for all $N$ large enough. Their starting point is the paper of Povel \cite{Povel}, whose results are used freely in the discrete setting rather than in the continuum. (As pointed out in \cite{DFSX}, a translation of Sznitman's method of enlargement of obstacles to a discrete setup was undertaken previously in \cite{Antal} -- interestingly this predates \cite{Povel}). Given this, what remains to be proved is that the range of the random walk covers all of the ball of radius $\rho_d n (1- \eps)$, i.e., our Theorem \ref{T:filledball_intro}. Hence the overlap with our paper is reduced to the proof of this theorem, which occupies Section \ref{S:filling} of this paper. The major differences with the approach of \cite{DFSX} can be summarised as follows:

-- once we have proved Theorem \ref{T:Prop31_intro}, we know a bit more than the confinement property, since we know that the local time profile is close in the $L^1$ sense to the eigenfunction. This implies in particular that mesoscopic balls are visited frequently, a step which is therefore easy for us (Lemma \ref{L:det_time}) but which requires an argument in \cite{DFSX} (more precisely, the authors of \cite{DFSX} argue that if a ball is not frequently visited it must be close to the boundary).

-- We have found a way to control uniformly the probability that a given set of $k$ points is avoided by the random walk in such a mesoscopic ball. Surprisingly, the control we get here is good enough that it works for \emph{any} configuration of points, no matter what its geometry, and depends only on its cardinality.
This is a major technical difference with \cite{DFSX} where the bound given depends on how many points in the set are far from one another. The key additional idea which allows us to do this here is to partition the points in the set according to their distance to the boundary and only consider points at a given distance from the boundary, where this distance is chosen to maximise the number of such points. This results in an arguably simpler line of reasoning from the conceptual point of view. 

\medskip We also mention the related work \cite{BY} which was the initial motivation of our investigation. In this paper, the penalisation by the range $e^{ - |R_N|}$ is replaced by the size of the boundary of the range $e^{ - | \partial R_N|}$. This turns the random walk into a polymer interface model. A conjecture in \cite{BY} states that a shape theorem takes place on the scale $n = N^{ 1/ (d+1)}$ instead of $n = N^{ 1/ (d+2)}$. This shape could then be thought of as a Wulff crystal shape for the random walk. Despite partial results in \cite{BY}, this conjecture remains wide open at positive  temperature. However, in the limit of zero temperature (for a more general model where the boundary size is measured through i.i.d. random variables attached to edges), Biskup and Procaccia \cite{BP1, BP2} were able to prove this conjecture. Observe that the random media representation which is the starting point of the method of enlargement of obstacles is not available for such a model.

\medskip Finally, for another approach to large deviations without compactness, see \cite{MV}.

\paragraph{Acknowledgements.} We thank the authors of \cite{DFSX} for useful discussions regarding their work. Part of this research was carried out when NB was a guest at Ecole Normale Sup\'erieure, Paris, whose support and hospitality is gratefully acknowledged. This project started in 2013, while R. Cerf was visiting Cambridge,
thanks to the support of the IUF.
NB's research was partly supported by EPSRC grants EP/L018896/1 and EP/I03372X/1.

\section{Further results and organisation of the paper}

\subsection{Preliminary lower bound on the partition function}

To explain some further details about our approach (including intermediate theorems of interest in their own right, see below), it will be useful to start by recalling the following lower bound due to Bolthausen on the partition function which is a quantitative improvement on the result of Donsker and Varadhan.

Let $n$ be an integer such that $n^{d+2}=N$.
Without loss of generality, we can assume that $N$ is such an integer power,
and we do so throughout the paper.
\begin{proposition}\label{P:lbZ}
There exists a constant $c \in \R$, which depends on the dimension $d$ only, such that,
for $N$ large enough, we have
$$Z_N\,\geq\,\exp\Big(-\chi_d n^{{d}}-cn^{{d-1}}\Big)\,  $$
where
$$
\chi_d = \frac{\lambda}{\rho_d^2} + \omega_d \rho_d^d
$$
is the same constant which appears in \eqref{DV}.
\end{proposition}


\begin{proof} See Proposition 2.1 in \cite{BO} (note that the proof is valid in any dimension).
\end{proof}

\subsection{Upper bound on the numerator}
Once we have a lower bound on the normalizing constant, the main problem is to
obtain an adequate upper bound on the integral of
$\exp\big(-|R_N|\big)$ over an arbitrary event $\cA$.
We can then rule out
those events $\cA$ for which we can obtain an upper bound which is negligible compared to
the previous lower bound.
Throughout the computations, we use the following convention. For $\cA$ an event, we write
\begin{equation}\label{Druleout}
\,E\Big( \exp\big(-|R_N|\big); \cA\Big)\,=\,
\,E\Big( \exp\big(-|R_N|\big)1_\cA\Big)
\,.
\end{equation}
The central object in our study is the local time of the random walk, defined as
\begin{equation}\label{D:localtimes}
\forall x\in\Zd\qquad L_N(x)\,=\,
\sum_{k=0}^{N-1}1_{\{\,S_k=x\,\}}\,.
\end{equation}
Our estimates will involve its square root, which we denote by $f_N$:
\begin{equation}\label{D:fN}
\forall x\in\Zd\qquad f_N(x)\,=\,\sqrt{L_N(x)}\,.
\end{equation}
In order to compare functions to one another, we will make use of various $\ell^p$ norms, which, unlike Bolthausen \cite{BO}, we take to be unscaled. Thus we define, for a function $f : \Z^d \to \R$,
\begin{equation}\label{D:lp}
\|f \|_p = \Big( \sum_{x \in \Z^d} | f(x)|^p \Big)^{1/p}.
\end{equation}
A fundamental step in the proof of Theorem \ref{T:Prop31_intro}, which takes up a substantial portion of this paper, is the following quantitative result which allows us to get an upper bound on \eqref{Druleout}.

\begin{theorem}
\label{upbo}
Let $\cF$ be a collection of functions from~$\Zd$ to $\R^+$.
For any $\kappa\geq 1$,
we have, for $n$ large enough,
$$\displaylines{
E\big(e^{-|R_N|};
f_N\in\cF)\,\leq\,
\exp(-\kappa n^d)+
\hfill\cr
\exp\Bigg(n^{d-1/8}-
\inf\,\Big\{\,
\Big|\,\Big\{\,x\in \Zd:
h(x)>0
\,\Big\}\Big|+
\frac{N}
{2}
(1-n^{-1/4})
\bigg(\max\Big(
\sqrt{\cE\big(\sqrt{h}\big)}-
\frac{1}{n^{9/8}} ,0\Big) \bigg)^2
\cr
\hfill
:h\in \ell^1(\Zd),\,h\geq 0,\,
\exists\,f\in\cF\quad
\Big|\Big|h-\frac{1}{N}f^2\Big|\Big|_{1}
\,\leq\,\frac{1}{n^{1/16}}
\,\Big\}\Bigg)\,.
}$$
\end{theorem}

\subsection{Organisation of paper}

Sections \ref{S:DV}, \ref{S:apriori}, and \ref{S:proof_upbo} are devoted to the proof of Theorem \ref{upbo}. Section \ref{S:continuous} and \ref{S:FK} explain how to deduce Theorem \ref{T:Prop31_intro} from Theorem \ref{upbo}. Section \ref{S:filling} deals with a proof of Theorem \ref{T:filledball_intro}.

As the proofs are rather lengthy, let
us sum up the main steps of the proofs.

\paragraph{Main steps in proof of Theorem \ref{upbo}.} \hfill 
%

\medskip \noindent
(i)
(Section \ref{S:DV})
We estimate probabilities of the form
$P(\|L_N-g^2\|_{1,D}<\Gamma)$
for some small $\Gamma$, where $\| f \|_{1,D}$ is the norm of the function $f$ restricted to $D$.
To this end, we
develop a new type of large deviations estimates for the random walk (Theorem \ref{GD}, see Section \ref{devi} for its proof).
In fact, that result involves the infimum over a
set of admissible starting points. In order to apply it to the random walk starting from the origin,
we have to introduce a correcting factor, which however does not destroy the
leading term in the estimates (Section \ref{corr}).

\medskip

\noindent
(ii) (Section~\ref{conran})
We show that, up to events whose $\wP_N$ probability is negligible,
for some $c>0$, we have
$|R_N|\,\leq\, cn^d$.
This is a direct consequence of the definition of $\wP_N$.
\medskip

\noindent
(iii) (Section~\ref{codir})
We show that, up to events whose $\wP_N$ probability is negligible,
for some $\kappa\geq 1$, we have
$\cE(f_N)\leq\kappa \,n^d\ln n$, where $\cE$ is the Dirichlet energy.
First we estimate the probability that $f_N$ is equal to a fixed function $f$. This estimate
relies on the classical martingale used by Donsker and Varadhan. We then bound the number
of functions satisfying the constraint
$\supp f\,\leq\, cn^d$.
\medskip

\medskip \noindent From now onwards, we need only to consider trajectories satisfying the points (ii) and (iii).
Having a control on the Dirichlet energy
yields
automatically a control
on the norm in $\ell^{2^*}(\Zd)$, where $2^*=2d/(d-2)$, via
a discrete Sobolev--Poincar\'e inequality. We will use the three bounds
$$\cE(f_N)\leq\kappa \,n^d\ln n\,,\quad
|R_N|\leq cn^d\,,\quad
||f_N||_{2^*}
\,\leq\,
c_{PS}
\sqrt{\kappa \,n^d\ln n}\,,
$$
to develop an adequate coarse--grained image
of the local time.
\medskip

\noindent
(iv) (Section~\ref{deco}) We partition the space into blocks of side length~$n$.
We focus on the blocks $B$ such that
$||f_N||_{2^*,B}\geq \delta n^{d/2^*+1}$.
The exponent $d/2^*+1$ corresponds to the typical situation for a block actively
visited by the random walk until time $N$: the number of visits per site should
be of order $n^2$, so $f_N$ is of order $n$ throughout the block.
We keep record of the indices of these blocks.
More precisely, we denote by $X$ the set of
the centers of these blocks; of course the
set $X$ depends on $f_N$ and is random. We denote by $E$ the union of these
blocks and by
$D$ the region~$E$
 enlarged with all the blocks on the frontier.
%
We control the norm of $f_N$ outside the region~$E$.
This is done with the help of
a discrete Poincar\'e--Sobolev inequality and the control of the $\ell^{2^*}$
norm:
\begin{equation*}
\sum_{x\in
\Zd\setminus E}
\big(f_N(x)\big)^{2^*}\,\leq\,
c_d\big(
\delta^{2^*} n^{d+2^*}
\big)^{1-2/2^*}
\big(
n^d+ {\kappa \,n^d\ln n}
\big)\,.
\end{equation*}

\noindent
(v)  (Section~\ref{loca}) 
We introduce a length scale $M$. We partition the space into blocks $B'(\xx)$
of side length $M$. We perform a local average
of $f_N$ on each such block and we get a function
$\fmn$. We control the norm of the difference $f_N-\fmn$
with the help of
a discrete Poincar\'e--Wirtinger inequality.
\medskip

\noindent
(vi)  (Section~\ref{coars})
We discretise next the values of the functions $\fmn$ over the blocks $B(\xx)$
which are included in $E$ with a discretisation step $\eta>0$.
This way we obtain a function $\fmne$, which is the coarse grained profile.
\medskip

At this point, we can write
$$E\big(e^{-|R_N|};f_N\in\cF
\big)
\,\leq\,
\sum_{X,g}
E\big(e^{-|R_N|};
f_N\in\cF,\,
\fmne=g
\big)\,,
\nonumber
$$
where the summation extends over the admissible set of blocks~$X$
and profiles~$g$. Whenever
$\fmne=g$, we can show that
$||f_N-g||_{2,D}<\Gamma_0$,
and $||L_N-g^2||_{1,D}<\Gamma_1$, where $\Gamma_0$, $\Gamma_1$ are explicit functions
of $N$.
\medskip

\noindent
(viii) (Section~\ref{contin}) We apply our large deviations inequality to bound the expectation
$$E\big(e^{-|R_N|};
\|L_N-g^2\|<\Gamma_1
 \big)\,.
$$
To do so, we
approximate the cardinality of the range $|R_N|$ by the
cardinality of the points where the coarse grained profile is quite large.
The resulting upper bound depends on an infimum over a set of functions~$h$
defined on the
domain~$D$, and more specifically on their Dirichlet energy inside~$D$.
\medskip

\noindent
(ix) (Section~\ref{trun})
With the help of a truncation and the control of the $\ell^{2^*}$ norm
outside~$E$, we relate the Dirichlet energy in the full space to the Dirichlet
energy restricted to~$D$.
This involves essentially a
discrete integration by parts.
\medskip

\noindent
(x) (Section~\ref{conc})
We use the inequality proved in step (ix). This way we get an upper bound involving
the Dirichlet energy in the full space, however we have to work further to get rid
of the truncation operator. After some tedious computations,
we obtain an upper bound depending only on the collection~$\cF$.
We plug this upper bound in the previous sums. It remains only to count
the number of terms in the sums.
We choose finally the parameters $M,\delta,\eta$ adequately to get the desired upper bound.

\paragraph{Main steps in the proof of Theorem \ref{T:filledball_intro}.} We fix a mesoscopic scale $m = n^{1- 2\kappa}$ where $\kappa>0$ is a small parameter. Suppose $\cG_{loc,x}$ holds and take $x = 0$ without loss of generality. We aim to show that the ball of radius $m$ around a point $z$ in the desired range $B(0, \rho_d n (1- n^{ - \kappa}))$ is entirely visited.
\medskip

\noindent
(i) (Section \ref{S:time_meso}) We first note that deterministically on $\cG_{loc,0}$, the walk spends a lot of time in $B(z, m)$. This is because otherwise there would be a polynomial error in the $L^1$ distance between $\ell_N$ and $(\phi_x)^2$.
\medskip

\noindent
(ii) (Section \ref{S:bridge_meso}) We show that under $\tilde P_N$ there are many disjoint portions of the walk of duration $m^2$ where the walk starts and ends inside the bulk of the ball (say, $B(z, m/2)$) and never leaves $B(z,m)$ throughout this interval of time. We call such a portion a bridge. To do so we use a change of measure argument whose cost (i.e., the value of the Radon--Nikodym derivative of this change of measure) is comparable to the entropic term in the partition function $Z_N$: that is, of order $\exp( - (\lambda /2 \rho_d^2) n^d)$. Since on $\cG_{loc, 0}$ the size of the range $|R_N|$ is also essentially deterministically lower bounded up to a small error, the energetic term $e^{- |R_N|}$ together with the cost of the change of measure is of order at most $Z_N$, which allows us to compare effectively $\tilde P_N$ to this new measure. A technical difficulty is that the change of measure technique is better implemented in continuous time rather than discrete time. Once we work under this change of measure, it is easy to check that the number of bridges is as desired: indeed, every time the walk is in the bulk of the ball, there is a decent chance that the next $m^2$ units of time will result in a bridge. Moreover, by Step 1 we control the number of trials, so ultimately the desired result follows from standard large deviations for Binomial random variables.
\medskip

\noindent
(iii) (Section \ref{S:sum}) We fix $k \ge 1$ and a set $\cX$ of $k$ points in $B(z, m)$, and try to estimate the $\tilde P_N$ probability that $\cX$ is avoided by the walk. We can condition on everything that happens outside of $B(z,m)$; on the event that the range avoids exactly the set $\cX$ the size of the range is then deterministic. Furthermore, if we condition on the number and endpoints of the bridges the probability that all the bridges avoid $\cX$ is exactly the product for individual bridges to avoid $\cX$. Hence the probability that the range restricted to $B(z,m)$ is strictly smaller than this ball can be written as a sum over $k\ge 1$, and over all subsets $\cX$ of size $k$ of the product of probabilities that bridges avoid $\cX$.
\medskip

\noindent
(iv) (Section \ref{S:bridge_trans}) To estimate the latter we need to control transition probabilities for bridges that are uniform in the starting and end points of the bridge. It is here that it is useful to have taken the starting and end points of the bridge in the bulk of the ball $B(z, m/2)$ and not near the boundary. This shows in particular that the probability for a bridge to find itself at a specific point at distance $r$ from the boundary at some specific time which is neither close to the start or the end of the bridge, is proportional to $(r/m)^2$. This follows essentially from a gambler's ruin probability argument.

\medskip\noindent
(v) (Section \ref{S:moment}) By the previous step it suffices to estimate the probability that a given bridge avoids $\cX$. We aim to find a bound that is uniform on the geometry of $\cX$ and depends only on the number of points $k$ in $\cX$. Intuitively, the easiest configuration to avoid is when $\cX$ is clumped together as a solid ball of radius $R = k^{1/d}$, so this should provide a lower bound on the desired probability. In that case our estimates on the transition probabilities from the previous step show via a moment computation that the probability for a bridge to hit $\cX$ should be at least $km^{2-d}/R^2 = k^{1- 2/d} m^{2-d}$, independently of the geometry of $\cX$ (ignoring boundary effects). This turns out to be true and can be deduced relatively easily from the fact that the Green function of the random walk in $\mathbb{Z}^d$, $d \ge 3$, is essentially monotone in the distance. (Such arguments can be used to prove isoperimetric inequalities for the capacity of a set, but we did not include this here for the sake of brevity).

In order to deal with boundary effects, in a way that is still uniform in the geometry of $\cX$, we divide the ball $B(z, m)$ into concentric annuli $A_j$ at distance $2^j$ from the boundary of the ball ($j \ge 1$). If all the points of $\cX$ were in the annulus $A_j$ it would be possible to control the boundary effects in a uniform way. Indeed the expected time spent in $\cX$ would be, by the gambler's ruin estimate from Step (iv),  proportional to $r^2$ where $r = 2^j$ is the distance to the boundary. However the expected time spent in $\cX $ starting from a point in $\cX$ would also be bounded by a factor proportional to $r^2$ as well using Step (iv) again. Hence the $r^2$ terms cancel in the moment computation, and we could use the above bound.
When $\cX$ is not contained in a single annulus $A_j$, we can instead consider $\cX_j = \cX \cap A_j$, where $A_j$ is chosen so that it contains most points of $\cX$. Then since it suffices to hit $\cX_j$, we can apply the above lower bound with $k$ replaced by $|\cX_j|$.
By the choice of $j$, this is at least $k /\log m$ instead of $k$. It turns out that plugging this extra logarithmic factor does not substantially alter the conclusion.


\section{Martingale estimates}
\label{S:DV}

\subsection{The classical martingale}
\label{clama}
The crucial ingredient to derive the relevant probabilistic estimates
on the random walk is the family of martingales used by Donsker and Varadhan,
which we define thereafter.
Let $u$ be a positive function defined on $\Zd$.
To the function $u$ we associate the function $V$ defined on $\Zd$ by
$$\forall x\in\Zd\qquad V(x)\,=\,\frac{1}{2d}
\sum_{{y\in\Zd},{|x-y|=1}}u(y)\,.$$
For $n\geq 0$, we set
$$M_n\,=\,\Big(
\prod_{k=0}^{n-1}\frac{u(S_k)}{V(S_k)}\Big)\,u(S_n)\,.$$
We claim that the process
$(M_n)_{n\in\N}$ is a martingale.
Indeed, for any $n\geq 0$,
\begin{align*}
E\big(M_n\,\big|\,S_0,\dots,S_{n-1}\big)\,&=\,
\Big(\prod_{k=0}^{n-1}\frac{u(S_k)}{V(S_k)}\Big)\,
E\big(u(S_n)\,|\,S_{n-1}\big)\cr
\,&=\,\Big(\prod_{k=0}^{n-1}\frac{u(S_k)}{V(S_k)}\Big)\,
V(S_{n-1})\,=\,M_{n-1}\,.
\end{align*}
In the same way, if the random walk
$(S_n)_{n\in\N}$
starts from an arbitrary point $x\in\Zd$, and if we denote by $P_x$
and $E_x$ the associated probability and expectation, then, under $P_x$,
the process
$(M_n)_{n\in\N}$ is again a martingale.

\subsection{The fundamental inequality}
Since $(M_n)_{n\in\N}$ is a martingale, then
\begin{equation}
\label{martineq}
E_x(M_N)\,=\,E_x(M_0)\,=\,u(x)\,.
\end{equation}
Let us express $M_N$ with the help of the local time $L_N$:
\begin{align*}
M_N\,
&=\,\exp\Big(\sum_{0\leq k<N}\ln
\frac{u(S_k)}{V(S_k)}\Big)\,u(S_N)\cr
&=\,\exp\Big(\sum_{y\in\Zd}\ln
\frac{u(y)}{V(y)}L_N(y)\Big)\,u(S_N)
\,.
\end{align*}
Since $|S_N-S_0|\leq N$, then
$$u(S_N)\,\geq\, \inf\,\{\,u(y):|y-x|\leq N\,\} \,.$$
Reporting this inequality in the martingale equality~(\ref{martineq}),
we get the following fundamental inequality.
\begin{lemma}
\label{funine}
For any $N\geq 1$, any $x\in\Zd$ and any positive function $u$ defined on $\Zd$,
we have
$$E_x\Big(\exp\Big(\sum_{y\in\Zd}\ln
\frac{u(y)}{V(y)}L_N(y)\Big)\Big)\,\leq\,
\frac{u(x)}
{\inf\,\{\,u(y):|y-x|\leq N\,\}}\,.$$
\end{lemma}

\subsection{Estimate for a fixed profile}
For $f$ a function from $\Zd$ to $\R$,
we define its discrete Dirichlet energy $\cE(f)$ by
$$\cE(f)\,=\,
\frac{1}{2d}
\sum_{\tatop{y,z\in\Zd}{|y-z|=1}}
\big(f(y)-f(z)\big)^2\,.$$
In this section, we shall prove the following estimate for the probability that
the square root of the local time is equal to a fixed profile.
\begin{proposition}
\label{profix}
Let $\phi$ be a function from $\Zd$ to $[0,+\infty[$ such that
$$\sum_{y\in\Zd}\phi(y)^2\,=\,N\,.$$
For any $N\geq 1$, any $\alpha\in ]0,1[$,
we have
$$P(f_N=\phi)\,\leq\,
\frac{
\phi(0)
}{\alpha}
\exp\Big(
-\frac{1}{2}\cE(\phi)
+\alpha
\sqrt{N\big|\supp\phi\big|}
\Big)\,.
$$
\end{proposition}
\begin{proof}
To bound the probability
$P(f_N=\phi)$, we proceed as follows.
Let $\alpha\in ]0,1[$ and
let $u:\Zd\to]0,+\infty[$ be the positive function defined on $\Zd$ by
$$\forall y\in\Zd\qquad u(y)\,=\,\max\big(\phi(y),\alpha\big)\,.$$
Obviously, we have
$$E\Big(
\exp\Big(\sum_{y\in\Zd}\ln
\frac{u(y)}{V(y)}L_N(y)\Big)
\Big)\,\geq\,
P(f_N=\phi)\,
\exp\Big(\sum_{y\in\Zd}\ln
\frac{u(y)}{V(y)}\phi(y)^2\Big)\,.
$$
Since the random walk starts from~$0$, then $f_N(0)\geq 1$, so we need only to
consider functions~$\phi$ such that $\phi(0)\geq 1$. In this case, we have
$u(0)=\phi(0)\geq\alpha$.
Applying the fundamental estimate of lemma~\ref{funine},
we get
$$P(f_N=\phi)\,\leq\,
\frac{\phi(0)}{\alpha}
\exp\Big(-\sum_{y\in\Zd}\ln
\frac{u(y)}{V(y)}\phi(y)^2\Big)\,.
$$
Next, we have, for $y\in\Zd$,
$$
-\ln
\frac{u(y)}{V(y)}\,=\,
\ln\Big(
1+\frac{V(y)-u(y)}{u(y)}
\Big)
\,\leq\,\frac{\Delta_1 u(y)}{u(y)}\,,
$$
where
$\Delta_1$ is the discrete Laplacian operator, defined by
$$\Delta_1u(y)\,=\,V(y)-u(y)\,.$$
Reporting in the previous inequality, we arrive at
$$P(f_N=\phi)\,\leq\,
\frac{\phi(0)}{\alpha}
\exp\Big(\sum_{y\in\Zd}
\frac{\Delta_1u(y)}{u(y)}\phi(y)^2\Big)\,.
$$
We evaluate next the sum in the exponential. For
$y\in\supp\phi$, we have $\phi(y)\geq 1>\alpha$, whence $\phi(y)=u(y)$,
therefore
\begin{align*}
\sum_{y\in\Zd}
\frac{\Delta_1u(y)}{u(y)}&\phi(y)^2\,=\,
\sum_{y\in\supp\phi}\kern-7pt
{\Delta_1u(y)}\,\phi(y)\cr
\,&=\,
\sum_{y\in\supp\phi}
\frac{1}{2d}
\sum_{z:|z-y|=1}\kern-7pt
\big(
u(z)-u(y)\big)
\,\phi(y)\cr
\,&\leq \,
\sum_{y\in\supp\phi}
\frac{1}{2d}
\sum_{z:|z-y|=1}\kern-7pt
\big(
\phi(z)-\phi(y)+\alpha\big)
\,\phi(y)\cr
\,&= \,
\frac{1}{2d}
\sum_{\tatop{y,z\in\supp\phi}{|y-z|=1}}
\phi(y)\phi(z)+
\sum_{y\in\supp\phi}\big(\alpha\phi(y)-\phi(y)^2\big)\cr
\,&= \,
-\frac{1}{4d}
\sum_{\tatop{y,z\in\supp\phi}{|y-z|=1}}
\big(\phi(y)-\phi(z)\big)^2+
\alpha\sum_{y\in\supp\phi}\phi(y)\,.
\end{align*}
Moreover
$$\Big(\sum_{y\in\supp\phi}\phi(y)\Big)^2\,\leq\,
\Big(\sum_{y\in\supp\phi}\phi(y)^2\Big)
\,\big|\supp\phi\big|
\,=\,
N\big|\supp\phi\big|
\,.$$
Putting together the previous inequalities, we obtain
the inequality stated in the proposition.
\end{proof}

\subsection{Donsker--Varadhan estimate for the random walk in the full space}
\label{devi}
In order to estimate the probability that $L_N$ belongs to a ball centered
at a fixed
function $g^2$, we develop here a deviation inequality, valid
for any value of $N$ and for any convex set of functions.
This deviation inequality is stated in Theorem~\ref{GD} of the introduction and we prove it here.
We use the notation introduced just before the statement of Theorem~\ref{GD}.

\begin{proof} Let $u$ be a positive function defined on $\Zd$ and let $x\in\Zd$.
Let $(M_n)_{n\in\N}$ be the martingale constructed with the function $u$
and the random walk $(S_n)_{n\in \N}$, defined in section~\ref{clama}.
We first remark that
${\tau(D,t)}$ is a stopping time for
	$(M_n)_{n\in\N}$. Indeed, the event $\{\,\tau(D,t)=n\,\}$ is measurable with respect to $L_n(D)$, hence also to
	$S_1,\dots,S_{n-1}$.
	We apply next the optional stopping theorem. Let $t,n\geq 1$,
we have
$$E_x\big(
M_{n\wedge \tau(D,t)}\big)
\,=\,E_x(M_0)\,=\,u(x)\,.$$
We note simply $\tau$ instead of $\tau(D,t)$ and
we
bound from below the lefthand member:
\begin{align}
\label{belo}
\nonumber
E_x\big(
M_{n\wedge \tau}\big)
&\geq\,
E_x\Big(
M_{n\wedge \tau}
1_{\tau<\infty}
\Big)\\
&=\,
E_x\Bigg(
\,\exp\Big(\sum_{y\in\Zd}\Big(\ln
\frac{u(y)}{V(y)}\Big)L_{n\wedge\tau}(y)\Big)\,
u(S_{n\wedge\tau})
1_{\tau<\infty}
\Bigg)
\,.
\end{align}
From now onwards, we suppose that the function~$u$ is superharmonic on $\Zd\setminus D$, i.e., we suppose that $u$
is such that
\begin{equation}
\label{subhar}
\forall y\in\Zd\setminus D\qquad{u(y)}
\,\geq\,
{V(y)}
\,=\,
\frac{1}{2d}
\!\!\!
\sum_{\tatop{z\in\Zd}{|y-z|=1}}
\!\!\!
u(z)\,.
\end{equation}
We denote by $\mathcal{S^+}$ the collection of the positive functions $u$ defined on $\Zd$ which satisfy
\eqref{subhar}.
The superharmonicity of $u$ implies that
\begin{equation}
	\label{inte}
\sum_{y\in\Zd}\Big(\ln
\frac{u(y)}{V(y)}\Big)L_{n\wedge\tau}(y)
\,\geq\,
\sum_{y\in D}\Big(\ln
\frac{u(y)}{V(y)}\Big)L_{n\wedge\tau}(y)\,.
\end{equation}
Reporting inequality~\eqref{inte} in inequality~\eqref{belo}, we get
$$u(x)\,\geq\,
E_x\Bigg(
\,\exp\Big(\sum_{y\in D}\Big(\ln
\frac{u(y)}{V(y)}\Big)L_{n\wedge\tau}(y)\Big)\,
u(S_{n\wedge\tau})
1_{\tau<\infty}
\Bigg)
\,.
$$
On the event $\tau<\infty$, we have
$$\displaylines{
\forall y\in D\qquad
\lim_{n\to\infty}\,
L_{n\wedge\tau}(y)\,=\,
L_{\tau}(y)\,=\,
L_t^D(y)\,,\cr
\lim_{n\to\infty}\,
S_{n\wedge\tau}\,=\,
S_{\tau}
\,.}$$
By Fatou's lemma, we have
\begin{equation*}
u(x)
\,\geq\,
\,E_x\Bigg(
\,\exp\Big(\sum_{y\in D}\Big(\ln
\frac{u(y)}{V(y)}\Big)L_{t}^D(y)\Big)\,
u(S_\tau)
1_{\tau<\infty}
\Bigg)\,.
\end{equation*}
Yet $S_\tau$ belongs to $\bD$, thus
\begin{equation}
u(x)
\,\geq\,
\Big(\inf_{y\in \bD}\,u(y)\Big)
\,E_x\Bigg(
\,\exp\Big(\sum_{y\in D}\Big(\ln
\frac{u(y)}{V(y)}\Big)L_{t}^D(y)\Big)\,
1_{\tau<\infty}
\Bigg)\,.
\end{equation}
Recall that
$$\bD\,=\,D\cup\{\,x\in\Zd:\exists\, y\in D\quad |x-y|=1\,\}\,.$$
Taking now the infimum over $x\in \bD$, we obtain
\begin{equation}
	\label{inbd}
\inf_{x\in \bD}
\,E_x\Bigg(
\,\exp\Big(\sum_{y\in D}\Big(\ln
\frac{u(y)}{V(y)}\Big)L_{t}^D(y)\Big)\,
1_{\tau<\infty}
\Bigg)\,\leq\,1\,.
\end{equation}
We proceed by bounding from below the lefthand member
	of \eqref{inbd}
as follows:
for any $x\in  \bD$,
\begin{multline}
	\label{foui}
\,E_x\Bigg(
\,\exp\Big(\sum_{y\in D}\Big(\ln
\frac{u(y)}{V(y)}\Big)L_{t}^D(y)\Big)\,
1_{\tau<\infty}
\Bigg)\,\cr \geq
\,E_x\Bigg(
\,\exp\Big(\sum_{y\in D}\Big(\ln
\frac{u(y)}{V(y)}\Big)L_{t}^D(y)\Big)\,;
\,\tau<\infty,\frac{1}{t}L_t^D\in C
\Bigg)\\
 \geq\,
\,\exp\Bigg(
\inf_{g\in C}\,\,\sum_{y\in D}\Big(\ln
\frac{u(y)}{V(y)}\Big){t}g(y)\Bigg)\,
P_x\Big(\frac{1}{t}L_t^D\in C,\,\tau<\infty\Big)\,.
\end{multline}
Let us define
\begin{equation}
\label{phitc}
\phi_t(C)\,=\,
\inf_{x\in D}\,
P_x\Big(\frac{1}{t}L_t^D\in C,\,\tau<\infty\Big)\,.
\end{equation}
Whenever $\tau<\infty$, the
function
$\frac{1}{t}L_t^D$ belongs to the set $M_1(D)$ defined by
$$
M_1(D)\,=\,\big\{\,\phi
\in\ell^1(D):\sum_{x\in D}\phi(x)=1\,,\forall x\in D\quad 0\leq\phi(x)\leq 1
\big\}\,.$$
Therefore we can replace the set $C$ by its intersection with $M_1(D)$ in formula~\eqref{foui} and
the previous inequalities yield that
\begin{equation}
\label{cvb}
\phi_t(C)\,\leq\,
\,\exp\Bigg(
-t\inf_{g\in C\cap M_1(D)}\,\,\sum_{y\in D}\Big(\ln
\frac{u(y)}{V(y)}\Big)g(y)\Bigg)\,.
\end{equation}
This inequality holds for any function $u$ in $\mathcal{S}^+$.
In order to get a functional which is convex, we perform a change of functions and we set
$\phi=\ln u$. We denote by $\T$ the image of $\S^+$ under this change of functions, i.e.,
$$\T\,=\,\big\{\,\ln u:u\in\S^+\,\big\}\,.$$
We rewrite inequality~\eqref{cvb} as follows: for any $\phi\in\T$,
\begin{equation}
\label{cvc}
\phi_t(C)\,\leq\,
\,\exp\Bigg(
-t
\inf_{g\in
C\cap M_1(D)
}\,\,\sum_{y\in D}
-\ln
\Big(
\frac{1}{2d}
\!\!\!
\sum_{\tatop{z\in\Zd}{|y-z|=1}}
\!\!\!
\exp\big(\phi(z)-\phi(y)\big)
\Big)g(y)\Bigg)\,.
\end{equation}
We define a map $\Phi$ on $\T\times M_{1}(D)$ by
$$\Phi(\phi,g)\,=\,
\sum_{y\in D}
-\ln
\Big(
\frac{1}{2d}
\!\!\!
\sum_{\tatop{z\in \Zd}{|y-z|=1}}
\!\!\!
\exp\big(\phi(z)-\phi(y)\big)
\Big)g(y)\,.
$$
Optimizing the previous inequality~\eqref{cvc} over the function $\phi$, we get
\begin{equation}
\label{hui}
\phi_t(C)\,\leq\,
\,\exp\Bigg(
	-t\,\,\sup_{\phi\in\T}\inf_{g\in
C\cap M_1(D)
}\,\,
\Phi(\phi,g)
\Bigg)\,.
\end{equation}
The map $\Phi$ is linear in $g$. We shall next prove that it is convex in $\phi$.
In fact, the convexity in $\phi$ is a consequence of the convexity of the functions
$t\in\R\mapsto \exp(t)$ and
$t\in\R^+\mapsto -\ln(t)$. The delicate point is to check that the domain of definition of
$\Phi$ is convex. This is the purpose of the next lemma.
\begin{lemma}
\label{convex}
The set $\T$ is convex.
\end{lemma}
\begin{proof}
Let $\phi,\psi$ belong to $\T$ and let $\alpha,\beta\in]0,1[$ be such that $\alpha+\beta=1$.
There exist $u,v\in\S^+$ such that
$\phi=\ln u$, $\psi=\ln v$, whence
$$\alpha\phi+\beta\psi\,=\,\ln\big(u^\alpha v^\beta)\,,$$
and we have to check that
$w=u^\alpha v^\beta$ is in $\S^+$.
Let $y$ be a fixed point in $\Zd\setminus D$. We apply the discrete H\"older inequality to the
functions $u^\alpha$, $v^\beta$, with respect to the measure $\nu_y$ which is the uniform measure
on the neighbours of $y$,
$$\nu_y\,=\,
\frac{1}{2d}
\!\!\!
\sum_{\tatop{z\in\Zd}{|y-z|=1}}
\!\!\!
\delta_z\,,$$
	and with the exponents $p=1/\alpha$, $q=1/\beta$.
We obtain
$$
	\int u^\alpha v^\beta\,d\nu_y
	\,\leq\,
	\Big(\int u \,d\nu_y\Big)^\alpha
	\Big(\int v \,d\nu_y\Big)^\beta
\,.
$$
This inequality can be rewritten as
$$
\frac{1}{2d}
\!\!\!
\sum_{\tatop{z\in\Zd}{|y-z|=1}}
\!\!\!
	u^\alpha(z) v^\beta(z)
	\,\leq\,
	\Bigg(
\frac{1}{2d}
\!\!\!
\sum_{\tatop{z\in\Zd}{|y-z|=1}}
\!\!\!
	u(z)
	\Bigg)^\alpha
	\Bigg(
\frac{1}{2d}
\!\!\!
\sum_{\tatop{z\in\Zd}{|y-z|=1}}
\!\!\!
	v(z)
	\Bigg)^\beta
	$$
We finally use the fact that $u,v$ are superharmonic to conclude.
\end{proof}

\noindent
Since the set $D$ is finite, the set
$\ell^{1}(D)$ is finite
dimensional, and the map $\Phi$ is continuous with respect to $g$ and any norm on $\ell^1(D)$.
Similarly, for any $\phi\in\T$ and $g\in M_1(D)$, the quantity $\Phi(\phi,g)$ depends only on the values of $\phi$
on the set $\bD$,
which is finite, thus
the map $\Phi$ is also continuous with respect to $\phi$ and the $\ell^1$ norm.
Moreover the set $C\cap M_1(D)$ is compact and convex (these are essential conditions in order to apply the minimax theorem).
Therefore, by the famous minimax theorem
(see \cite{FA}),
\begin{equation}
\label{supinf}
\sup_{\phi\in\T}\inf_{g\in
C\cap M_1(D)
}\,\,
\Phi(\phi,g)
\,=\,
\inf_{g\in C\cap M_1(D) }
\sup_{\phi\in\T}
\,\,
\Phi(\phi,g)
\,.
\end{equation}
Let us fix
${g\in C\cap M_1(D) }$ and let us bound from below
$\smash{\sup_{\phi\in\T} \Phi(\phi,g)}$.
Let us fix
$\phi\in\T$, we have, by convexity of $-\ln$,
$$\Phi(\phi,g)\,\geq\,
-\ln\Bigg(
\sum_{y\in D}
\Big(
\frac{1}{2d}
\!\!\!
\sum_{\tatop{z\in \Zd}{|y-z|=1}}
\!\!\!
\exp\big(\phi(z)-\phi(y)\big)
\Big)g(y)\Bigg)\,.
$$
Using the inequality $-\ln(t)\geq 1-t$ for $t>0$, we get
\begin{equation}
\label{ina}
\Phi(\phi,g)\,\geq\,
1
-
\sum_{y\in D}
\Big(
\frac{1}{2d}
\!\!\!
\sum_{\tatop{z\in \Zd}{|y-z|=1}}
\!\!\!
\exp\big(\phi(z)-\phi(y)\big)
\Big)g(y)
\end{equation}
Let $\varepsilon$ be such that
$$0\,<\,\varepsilon\,<\,\min\,\big\{\,g(x):x\in D,\,g(x)>0\,\big\}\,.$$
We define next an adequate function $u_\varepsilon$.
The entrance time $T_0$ of $D$ is
$$T_0\,=\,\inf\,\big\{\,n\geq 0:S_n\in D\,\big\}\,.$$
We define
$$\forall x\in \Zd\qquad u_\varepsilon(x)
\,=\,E_x\Big(
\sqrt{\max\big(
	g(S_{T_0})
,\varepsilon\big)}
	\,\big|\,T_0<\infty\big)\,.
$$
With this definition, we have that
$$\forall x\in D\qquad u_\varepsilon(x)
\,=\,\sqrt{\max\big(g(x),\varepsilon\big)}\,.$$
We claim that this function $u_\varepsilon$ belongs to $\S^+$. Obviously it is strictly positive everywhere.
The function $u_\varepsilon$ is harmonic on $\Zd\setminus D$, hence it is also superharmonic.
We apply inequality~\eqref{ina} with the function
$\phi_\varepsilon(x)=\ln u_\varepsilon$ and we get
\begin{align}
{\sup_{\phi\in\T} \Phi(\phi,g)}
\,\geq\,
\Phi(\phi_\varepsilon,g) &=
1 -
	\sum_{y\in D, g(y)>0}
\frac{1}{2d}
\!\!\!
\sum_{\tatop{z\in \Zd}{|y-z|=1}}
\!\!\!
	\frac{u_\varepsilon(z)}
	{u_\varepsilon(y)}
g(y)
\nonumber
\\
& =
1 -
	\sum_{y\in D, g(y)>0}
\frac{1}{2d}
\!\!\!
\sum_{\tatop{z\in \Zd}{|y-z|=1}}
\!\!\!
	\frac{u_\varepsilon(z)}{ \sqrt{\max\big(g(y),\varepsilon\big)}}
g(y)
\nonumber
\\
&=
1 -
	\sum_{y\in D}
\frac{1}{2d}
\!\!\!
\sum_{\tatop{z\in \Zd}{|y-z|=1}}
\!\!\!
	{u_\varepsilon(z)}{ \sqrt{g(y)}}
	\,.
\label{inb}
\end{align}
This inequality holds for $\varepsilon>0$ sufficiently small. We send next $\varepsilon$ to $0$. We have
$$\forall x\in \Zd\qquad \lim_{\varepsilon\to 0}u_\varepsilon(x)
\,=\,E_x\Big( \sqrt{ g(S_{T_0}) } \,\big|\,T_0<\infty\Big)\,.
$$
Passing to the limit as $\varepsilon$ goes to $0$ in inequality~\eqref{inb}, we get
\begin{equation}
\label{inc}
{\sup_{\phi\in\T} \Phi(\phi,g)}
\,\geq\,
1 -
	\sum_{y\in D}
\frac{1}{2d}
\!\!\!
\sum_{\tatop{z\in \Zd}{|y-z|=1}}
\!\!\!
E_z\Big( \sqrt{ g(S_{T_0}) } \,\big|\,T_0<\infty\big)
	{ \sqrt{g(y)}}
	\,.
\end{equation}
Let us introduce
the hitting time $T_1$ of $D$ is
$$T_1\,=\,\inf\,\big\{\,n\geq 1:S_n\in D\,\big\}\,.$$
Notice that the entrance and hitting time are equal whenever the starting point is in $\Zd\setminus D$.
However, if the starting point belongs to $D$, then $T_0=0$ and $T_1\geq 1$.
With the help of $T_1$, we rewrite formula~\eqref{inc} as
\begin{multline}
\label{ind}
{\sup_{\phi\in\T} \Phi(\phi,g)}
\,\geq\,
1 -
	\sum_{y\in D}
E_y\Big( \sqrt{ g(S_{T_1}) } \,\big|\,T_0<\infty\big)
	{ \sqrt{g(y)}}\cr
\,=\,
	1 -
	\sum_{y\in D}
	\sum_{z\in D}
	P_y\big( S_{T_1}=z \,\big|\,T_0<\infty\big)
	\sqrt{ g(z) }
	{ \sqrt{g(y)}}
	\,.
\end{multline}
Using the time reversibility of the random walk
$(S_n)_{n\in \N}$, we have
\begin{equation}
\label{wgyu}
\forall y,z\in D\qquad
	P_y\big( S_{T_1}=z \,\big|\,T_0<\infty\big)
	\,=\,
	P_z\big( S_{T_1}=y \,\big|\,T_0<\infty\big)
	\,,
\end{equation}
whence
\begin{equation}
\label{fgyu}
	\sum_{y\in D}
	\sum_{z\in D}
	P_y\big( S_{T_1}=z \,\big|\,T_0<\infty\big)
	g(z) \,=\,
	\sum_{z\in D}
	\sum_{y\in D}
	P_z\big( S_{T_1}=y \,\big|\,T_0<\infty\big)
	g(z)
	\,=\,
	\sum_{z\in D}
	g(z)
	\,=\,1\,.
\end{equation}
The identities~\eqref{fgyu} yield that
\begin{equation}
\label{qgyu}
	1 -
	\sum_{y\in D}
	\sum_{z\in D}
	P_y\big( S_{T_1}=z \,\big|\,T_0<\infty\big)
	\sqrt{ g(z) }
	{ \sqrt{g(y)}}\,=\,
	\frac{1}{2}
	\sum_{y\in D}
	\sum_{z\in D}
	P_y\big( S_{T_1}=z \,\big|\,T_0<\infty\big)
	\big(\sqrt{ g(z) }-
	{ \sqrt{g(y)}}\big)^2\,.
\end{equation}
Reporting~\eqref{qgyu} in~\eqref{ind}, we get
\begin{equation}
\label{wqx}
{\sup_{\phi\in\T} \Phi(\phi,g)}
\,\geq\,
	\frac{1}{2}
	\sum_{y\in D}
	\sum_{z\in D}
	P_y\big( S_{T_1}=z \,\big|\,T_0<\infty\big)
	\big(\sqrt{ g(z) }-
	{ \sqrt{g(y)}}\big)^2\,.
\end{equation}
Now, for any $y,z\in D$ such that
${|y-z|=1}$, we have
	$$P_y\big( S_{T_1}=z \,\big|\,T_0<\infty\big)
	\,\geq\, \frac{1}{2d}\,,$$
thus we have the lower bound
\begin{equation}
\label{ine}
{\sup_{\phi\in\T} \Phi(\phi,g)}
\,\geq\,
\sum_{\tatop{y,z\in D}{|y-z|=1}}
\frac{1}{2d}
	\big(\sqrt{ g(z) }-
	{ \sqrt{g(y)}}\big)^2
\,=\,\frac{1}{2}\cE(\sqrt{g,D})
	\,.
\end{equation}
Taking the infimum with respect to
${g\in C\cap M_1(D) }$ and coming back to inequality~\eqref{hui}, we obtain
the desired result.
\end{proof}

\begin{remark}\label{R:dv_exc} To clarify the proof of Theorem \ref{GD} it may be useful to consider the \emph{finite} Markov chain on $D$ obtained from the simple random walk on $\mathbb{Z}^d$ by restricting it to the times when it visits $D$. This is the Markov chain where the transition probabilities are given by
$$q(y,z) = P_y(S_{T_1} = z).$$
 In particular, this chain coincides with the simple random walk on the vertices $y$ for which all neighbours are in $D$. However all the vertices on the boundary of $D$ are also connected to one another, and so this chain can jump from any boundary vertex to any other boundary vertex (corresponding to an excursion of the simple random walk away from $D$).

Then note that the right hand side of \eqref{wqx} is nothing but the Dirichlet energy for this Markov chain. On the other hand, \eqref{ine} shows that the Dirichlet energy for this Markov chain of a function $g: D \to \R$ is always lower bounded by the Dirichlet energy $\mathcal{E}(g; D)$ since we can simply restrict to the transitions on $\mathbb{Z}^d$ and ignore the additional connections along the boundary.

Effectively, this argument reduces the simple random walk on the infinite state space $\mathbb{Z}^d$ to a finite state space, in a way that is conceptually reminiscent of the compactification arguments to use Donsker--Varadhan large deviation estimates in Bolthausen's work \cite{BO} and many other works on large deviations. However, the advantage of this approach is that it does not alter significantly the geometry of the ambient space and so eventually lets us use functional inequalities that are more readily available in $\R^d$ than on a torus.
\end{remark}

\subsection{Correction for the origin}
\label{corr}
The problem with the inequality of Theorem~\ref{GD} is the presence of the infimum
over $x\in\bD$. In order to go around it, we shall take advantage of the fact
that our trajectories are constrained to have a range of cardinality less than $cn^d$.
We first bound from below the probability to travel between an arbitrary
point of the range and the origin. To this end, we shall use a standard estimate on
multinomial coefficients, that we recall next.
\begin{lemma}
\label{lf}
For any $k\geq 1$, $r\geq 2$,
any $k_1,\dots,k_r\in\{\,0,\dots, k\,\}$
such that
$k_1+\dots+k_r=k$, we have
$$
\Big|\frac{1}{k}\ln
\frac{k!}{ k_1!\cdots k_r!}
+\sum_{i=1}^r
\frac{k_i}{ k}\ln \frac{k_i}{ k}
\Big|
\,\leq\,
\frac{r}{ k}(\ln k+2)
\,.$$
\end{lemma}
\noindent
\begin{proof}
The proof of this estimate is standard (see for instance \cite{EL}).
Setting, for $k\in\mathbb N$, $f(k)=\ln k!-k\ln k+k$,
we have
\begin{align*}
\ln\frac{k!}{ k_1!\cdots k_r!}
\,&=\,\ln k!-\sum_{i=1}^r\ln k_i!\cr
\,&=\,k\ln k-k+f(k)-\sum_{i=1}^r(k_i\ln k_i-k_i+f(k_i))\\
\,&=\,-\sum_{i=1}^r
{k_i}\ln \frac{k_i}{k}
+f(k)-\sum_{i=1}^rf(k_i)\,.
\end{align*}
Comparing the discrete sum
$$\ln k!=\sum_{1\leq i\leq k}\ln k$$
to the integral $\int_1^k\ln x\,dx$, we see that
$1\leq f(k)\leq \ln k +2$ for all $k\geq 1$. On one hand,
$$ f(k)-\sum_{i=1}^rf(k_i)\,\leq\, \ln k+2-r\leq r(\ln k+2)\,,$$
on the other hand,
$$ f(k)-\sum_{i=1}^rf(k_i)\,\geq\, 1-
\sum_{i=1}^r(\ln k_i+2)\,\geq\,
1-2r-r\ln k\,\geq\, -r(\ln k+2)\,$$
and we have the desired inequalities.
\end{proof}

\noindent
The box
of side length~$r>0$
centered at the origin
is the set
\index{$\Lambda(r)$}
\[\Lambda(r) \,=\, \big\{\,(x_1,\cdots,x_d)\in\R^d
:
\forall i\in \{1,\cdots, d\}\,\,
-r/2<x_i\leq r/2
\,\big\}\,. \]
With the help of Lemma~\ref{lf}, we obtain the following lower bound for the
symmetric random walk.
\begin{lemma}
\label{afaire}
Let $c>0$. There exists $c'>0$ such that, for $n$ large enough,
\begin{multline*}
\forall k\in\{\,n^{d+1},\dots,n^{d+2}\,\},\quad
\forall x\in\La(cn^d),\quad
P_x(S_{k}=0)+
P_x(S_{k-1}=0)\,\geq\,\exp(-c'n^{2d}/k)\,.
\end{multline*}
\end{lemma}
\begin{proof}
	Let $x=(x_1,\dots,x_d)\in \La(cn^d)$ and let
$k\in\{\,n^{d+1},\dots,n^{d+2}\,\}$.
We have
$$P_x(S_{k}=0)\,=\,
P_0(S_{k}=x)\,=\,
	\sum_{
\tatop{{i_1,\dots,i_d} } {{j_1,\dots,j_d} }
}
	\frac{1}{(2d)^k}
	\frac{k!}{i_1!j_1!\cdots i_d!j_d!}
	\,,$$
	where the sum runs over the indices $i_1,j_1,\dots,i_d,j_d$ such that
	$$i_1-j_1=x_1,\dots,i_d-j_d=x_d\,,\quad
	i_1+j_1+\cdots+i_d+j_d=k\,.$$
	The index $i_1$ corresponds to the number of moves associated to the vector $(1,0,\dots,0)$,
	the index $j_1$ to
	the number of moves associated to the vector $(-1,0,\dots,0)$, and so on.
This set of indices is empty in the case where $k$ and $x_1+\dots+x_d$ don't have the same parity,
that is why we have to consider the sum
$P_x(S_{k}=0)+
P_x(S_{k-1}=0)$. To simplify the discussion, we assume that the term
$P_x(S_{k}=0)$ is non--zero. We will get the desired lower bound by considering only
one term in the sum, the term corresponding to
	$$i_1=\frac{k}{2d}+\frac{x_1}{2},\,
	j_1=\frac{k}{2d}-\frac{x_1}{2},\,\dots,\,
	i_d=\frac{k}{2d}+\frac{x_d}{2},\,
	j_d=\frac{k}{2d}-\frac{x_d}{2}\,.$$
	To be precise, we should take integer parts in the above formula, but this would become
	a bit messy, so we do as if all the fractions were integers.
	We get that
$$P_x(S_{k}=0)\,\geq\,
	\frac{1}{(2d)^k}
	\frac{k!}{
	\Big(\displaystyle \frac{k}{2d}+\frac{x_1}{2}\Big)!
	\Big(\displaystyle \frac{k}{2d}-\frac{x_1}{2}\Big)!
	\cdots
	\Big(\frac{k}{2d}+\frac{x_d}{2}\Big)!
	\Big(\frac{k}{2d}-\frac{x_d}{2}\Big)!
	}
	\,.$$
	Taking the $\ln$ and using the inequality of Lemma~\ref{lf}, we obtain
\begin{multline*}
	\ln P_x(S_{k}=0)\,\geq\,-k\ln (2d)\hfill\cr
	-\sum_{i=1}^d
	\Big(\displaystyle \frac{k}{2d}+\frac{x_i}{2}\Big)\ln
	\Big(\frac{1}{2d}+\frac{x_i}{2k}\Big)+
	\Big(\displaystyle \frac{k}{2d}-\frac{x_i}{2}\Big)\ln
	\Big(\frac{1}{2d}-\frac{x_i}{2k}\Big)
	-2d(\ln k+2)\cr
	\,\geq\,
	-\sum_{i=1}^d
	\Big(\displaystyle \frac{k}{2d}+\frac{x_i}{2}\Big)\ln
	\Big({1}+\frac{dx_i}{k}\Big)+
	\Big(\displaystyle \frac{k}{2d}-\frac{x_i}{2}\Big)\ln
	\Big({1}-\frac{dx_i}{k}\Big)
	-2d(\ln k+2)\cr
	\,\geq\,
	-\sum_{i=1}^d
	\Big(\displaystyle \frac{k}{2d}+\frac{x_i}{2}\Big)
	\Big(\frac{dx_i}{k}\Big)-
	\Big(\displaystyle \frac{k}{2d}-\frac{x_i}{2}\Big)
	\Big(\frac{dx_i}{k}\Big)
	-2d(\ln k+2)\cr
	\,\geq\,
	-\sum_{i=1}^d
	\Big(\frac{d}{k}\Big)x_i^2
	-2d(\ln k+2)
	\,\geq\,
	-d
	\Big(\frac{d}{k}\Big)(cn^d)^2
	-2d(\ln k+2)\,.
\end{multline*}
	Using the hypothesis on $k$, we conclude that
	$$\ln P_x(S_{k}=0)\,\geq\,
	-d^2c^2
	\Big(\frac{n^{2d}}{k}\Big)
	-2d((d+2)\ln n+2)\,.
	$$
	Since $n^{2d}/k\geq n^{d-2}\geq n$, then for $n$ large enough, we obtain
	the desired estimate.
\end{proof}

\noindent
We will use Lemma~\ref{afaire} together with the next proposition in order to have
a deviation inequality for the trajectories starting from the origin. The idea is to
let the random walk evolve naturally over a certain time interval, so that it has a chance
to visit all the points of $\bD$, and in particular the point of $\bD$ realizing the infimum
in the inequality of Theorem~\ref{GD}.
Of course, we are forced to introduce a correcting factor, but we will adjust our parameters
so that this correcting factor is not disturbing.
\begin{proposition}
\label{depe}
Let $D$ be an arbitrary subset of $\Zd$ and let $g$ be a function in $\ell^{1}(D)$.
Let $k,t$ be positive integers with $k<t$. Let $r>0$.
For any $x\in\bD$, we have
\begin{equation*}
P\Big(
\Big|\Big|\frac{1}{t}L_t^D-g^2
\Big|\Big|_{1,D}\,\leq\,r
,\,\tau<\infty
\Big)\,\leq\,
\frac{\displaystyle 2\,
P_x\Big(
\Big|\Big|\frac{1}{t}L_t^D-g^2
\Big|\Big|_{1,D}\,\leq\,r
+4\frac{k}{t-k}
,\,\tau<\infty
\Big)}
{\displaystyle P_x(S_{k}=0)+ P_x(S_{k-1}=0)}\,.
\end{equation*}
\end{proposition}
\begin{proof}
Let $s,t$ be such that $0<s<t$. We have
\begin{align}
\Big|\Big|\frac{1}{t}L_t^D-
\frac{1}{s}L_s^D\Big|\Big|_{1,D}\,&\leq\,
\Big|\Big|\frac{1}{t}L_t^D-
\frac{1}{s}L_t^D\Big|\Big|_{1,D}\,+\,
\Big|\Big|\frac{1}{s}L_t^D-
\frac{1}{s}L_s^D\Big|\Big|_{1,D}\,\cr
&\leq\,
\Big|\frac{1}{t}
-\frac{1}{s}\Big|\,t+\frac{t-s}{s}\,\leq\,
2\frac{t-s}{s}
\,.
\label{prin}
\end{align}
Let $x\in\Zd$ be an arbitrary starting point. Let also $r>0$ and
let $g$ be an arbitrary function
from $D$ to $\R$. We condition on the state visited
by the random walk at time $t-s$, we apply the strong Markov property
and we use twice the inequality~\eqref{prin}
to get:
\begin{multline}
P_x\Big(
\Big|\Big|\frac{1}{t}L_t^D-g^2
\Big|\Big|_{1,D}
	\,\leq\,r
,\,\tau<\infty
	\Big)
	\cr
 \geq\,
\sum_{y\in\Zd}
P_y\Big(
	S_{t-s}=y\,,
\Big|\Big|
	\frac{1}{s}\big(
	L_t^D-
	L_s^D\big)-g^2
\Big|\Big|_{1,D}\,\leq\,r
-2\frac{t-s}{s}
,\,\tau<\infty
\Big)\\
 \geq\,
\sum_{y\in\Zd}
P_x(S_{t-s}=y)\,
P_y\Big(
\Big|\Big|\frac{1}{s}L_s^D-g^2
\Big|\Big|_{1,D}\,\leq\,r
-2\frac{t-s}{s}
,\,\tau<\infty
\Big)\\
\geq\,
P_x(S_{t-s}=0)\,
P\Big(
\Big|\Big|\frac{1}{t}L_t^D-g^2
\Big|\Big|_{1,D}\,\leq\,r
-4\frac{t-s}{s}
,\,\tau<\infty
\Big)\,.
\label{firt}
\end{multline}
It might happen that
$P_x(S_{t-s}=0)$ vanishes if $|x|_{1}$ and $t-s$ don't have the same parity. To avoid
this nasty detail, we apply the inequality with $s+1$ instead of $s$. Noting that
	$$\frac{t-s-1}{s+1}\,\leq\,\frac{t-s}{s}\,,$$
we get
\begin{equation*}
P_x\Big(
\Big|\Big|\frac{1}{t}L_t^D-g^2
\Big|\Big|_{1,D}
\kern-7pt
\leq\,r
,\,\tau<\infty
	\Big)
\geq
P_x(S_{t-s-1}=0)\,
P\Big(
\Big|\Big|\frac{1}{t}L_t^D-g^2
\Big|\Big|_{1,D}
\kern-7pt
\leq\,
r
-4\frac{t-s}{s}
,\,\tau<\infty
\Big).
\end{equation*}
Summing this inequality and inequality~\eqref{firt}, and setting $k=t-s$, we obtain,
for any $0<k<t$,
\begin{multline*}
2\,P_x\Big(
\Big|\Big|\frac{1}{t}L_t^D-g^2
\Big|\Big|_{1,D}\,\leq\,r
,\,\tau<\infty
	\Big)
\,\geq\,\cr
\big(
P_x(S_{k}=0)+
P_x(S_{k-1}=0)\big)\,
P\Big(
\Big|\Big|\frac{1}{t}L_t^D-g^2
\Big|\Big|_{1,D}\,\leq\,r
-4\frac{k}{t-k}
,\,\tau<\infty
\Big)\,.
\end{multline*}
We finally make the change of variable $r'=r
-4\frac{k}{t-k}$ and we obtain the desired inequality.
\end{proof}

\section{A priori estimates}
\label{S:apriori}

\subsection{Control of the range}
\label{conran}
We shall obtain a control on the size of $|R_N|$.
Let $c>0$. We write
$$
\wP_N\big(\,|R_N|\,>\,cn^d\,\big)\,\leq\,
\frac{\exp(-cn^d)}{
E\big(
\exp\big(-|R_N|\big)\big)}\,.$$
From theorem~$1$ of~\cite{DVRA}, there exists a positive constant $k(1,d)$ such
that
$$\lim_{N\to\infty}\,\frac{1}{n^d}\,\ln
E\big(
\exp\big(-|R_N|\big)\big)\,=\,-k(1,d)\,.$$
We conclude that, for any $c>0$ and for $N$ large enough, we have
\begin{equation}
\label{range}
\wP_N\big(\,|R_N|\,>\,cn^d\,\big)\,\leq\,
\exp\big(-(c-2k(1,d))n^d\big)\,.
\end{equation}
\subsection{Control of the Dirichlet energy}
\label{codir}
We shall obtain a bound on $\cE(f_N)$ conditionally on the size of
$|R_N|$.
\begin{proposition}
\label{goode}
Let $c\geq 1$ and $\kappa\geq 1$ be such that
${\kappa}/{2}-4dc>0$.
For $n$ large enough, we have
$$\displaylines{
P\big(
\cE(f_N)\geq\kappa \,n^d\ln n,\,
|R_N|\leq cn^d\big)\,\leq\,
\exp\Big(
-\big(\frac{\kappa}{2}-4dc\big)n^d\ln n\Big)
\,.
}$$
\end{proposition}
\begin{proof}
Let $\lambda,c>0$. We write
$$P\big(
\cE(f_N)\geq\lambda n^d,\,
|R_N|\leq cn^d\big)\,=\,
\sum_{\phi}P(f_N=\phi)\,,$$
where the summation extends over the profiles $\phi:\Z^d\to[0,+\infty[$ such that
$$\phi(0)\geq 1\,,\quad
\sum_{y\in\Zd}\phi(y)^2\,=\,N\,,\quad
|\supp \phi|\leq cn^d\,,\quad
\cE(\phi)\geq\lambda n^d
\,.$$
For such a profile $\phi$, we have,
thanks to the inequality of proposition~\ref{profix}:
$$P(f_N=\phi)\,\leq\,
\frac{
N
}{\alpha}
\exp\Big(
-\frac{\lambda}{2}n^d
+\alpha
\sqrt{Ncn^d}
\Big)\,.
$$
For the probability
$P(f_N=\phi)$ to be positive, it is necessary that $\supp\phi$ is a connected
subset of $\Zd$ containing the origin, and that $1\leq \phi(y)^2\leq N$ for
all $y\in\supp\phi$.
The number of possible choices for the support of $\phi$ is bounded by
	$\smash{(c_d)^{cn^d}}$, where $c_d$ is a constant depending on the dimension $d$ only.
Once the support is fixed, the number of choices for the profile $\phi$
is bounded by
$N^{cn^d}$.
In the end, the total number of
profiles $\phi$ satisfying the previous constraints
is bounded by
$\smash{(c_dN)^{cn^d}}$.
Therefore
$$P\big(
\cE(f_N)\geq\lambda n^d,\,
|R_N|\leq cn^d\big)\,\leq\,
\smash{(c_dN)^{cn^d}}
\frac{
N
}{\alpha}
\exp\Big(
-\frac{\lambda}{2}n^d
+\alpha
\sqrt{Ncn^d}
\Big)\,.
$$
We choose $\alpha$ and $\lambda$ of the form
$$\alpha\,=\,\frac{\ln n}{n}\,,\qquad\lambda\,=\,\kappa\ln n\,.$$
Recalling that $n^{d+2}=N$, the previous inequality can be rewritten as
\begin{align*}
P\big(
\cE(f_N)\geq\kappa \,n^d\ln n,\,
|R_N|\leq cn^d\big)\,& \leq\,
\frac{ n^{d+3} }{\ln n}
\exp\Big(
{cn^d}\big(\ln c_d+(d+2)\ln n\big)
-\frac{\kappa}{2}n^d\ln n
+
\sqrt{c} n^d\ln n
\Big)\\
&\,\leq\,
\exp\Big(
-\big(\frac{\kappa}{2}-4dc\big)n^d\ln n\Big)
\end{align*}
where the last inequality holds for $n$ large enough.
\end{proof}

\subsection{High density blocks}
\label{deco}
We divide $\Zd$ into boxes called blocks in the
following way.
The box
of side length~$r>0$
centered at the origin
is the set
\index{$\Lambda(r)$}
\[\Lambda(r) \,=\, \big\{\,(x_1,\cdots,x_d)\in\R^d
:
\forall i\in \{1,\cdots, d\}\,\,
-r/2<x_i\leq r/2
\,\big\}\,. \]
Let $n$ be a positive integer.
For $\xx\in \Zd$, we define the block indexed by $\xx$
as
$$B(\xx)\, =\, n\xx+\La(n)\,.$$
Note that the blocks partition $\Rd$.
We perform here a deterministic construction on any function~$f:\Zd\to\R^+$
in order to record the blocks where the function~$f$
has a high density.
This construction will be applied to build the coarse grained profile of the function~$f_N$.
Let $f$ be a function from $\Zd$ to $\R^+$.
Let $n\geq1$, $\delta>0$. We recall the notation $2^*=2d/(d-2)$ and we define
$$\XX(f,\delta)\,=\,\Big\{\,\xx\in\Zd:
\sum_{y\in B(\xx)}\big(f(y)\big)^{2^*}\geq
\delta^{2^*} n^{d+2^*}
\,\Big\}\,.$$
The exponent $d+2^*$ corresponds to the typical situation for a block actively
visited by the random walk until time $N$: the number of visits per site should
be of order $n^2$, so $f_N$ is of order $n$ throughout the block.
The blocks corresponding to vertices outside of
$\XX(f,\delta)$ are low density blocks. Our goal is to control the total
contribution of these blocks to the $2^*$--norm of~$f$.
\begin{lemma}
\label{contout}
There exists a constant $c_d$ depending on the dimension $d$ only such that,
for any function $f$ from $\Zd$ to $\R^+$, we have
$$\sum_{\xx\in\Zd\setminus\XX(f,\delta)}
\sum_{y\in B(\xx)}
\big(f(y)\big)^{2^*}\,\leq\,
c_d\big(
\delta^{2^*} n^{d+2^*}
\big)^{1-2/2^*}
\Big(\frac{1}{n^2}
\big(||f||_{2}\big)^2+{\cE(f)}\Big)\,.
$$
\end{lemma}
\begin{proof}
Let $f$ be a function from $\Zd$ to $\R^+$.
Let $\xx\in\Zd$. We write
\begin{equation}
\label{qze}
\sum_{y\in B(\xx)}\big(f(y)\big)^{2^*}\,=\,
\Big(\sum_{y\in B(\xx)}\big(f(y)\big)^{2^*}\Big)^{1-2/2^*}
\Big(\sum_{y\in B(\xx)}\big(f(y)\big)^{2^*}\Big)^{2/2^*}\,.
\end{equation}
To control the last factor, we apply the
discrete Poincar\'e--Sobolev inequality stated
in corollary~\ref{UPS}:
\begin{align}
\label{qzf}
\Big(\sum_{y\in B(\xx)}\big(f(y)\big)^{2^*}\Big)^{2/2^*}
&\,\leq\,
(c_{PS})^2\Big(\frac{1}{n}
\nonumber
||f||_{2,B(\xx)}+\sqrt{\cE(f,B(\xx))}\Big)^2\\
&\,\leq\,
2(c_{PS})^2\Big(\frac{1}{n^2}
\big(||f||_{2,B(\xx)}\big)^2+{\cE(f,B(\xx))}\Big)
\,.
\end{align}
If $\xx$ does not belong to
$\XX(f,\delta)$, then we have
\begin{equation}
\label{qzg}
\sum_{y\in B(\xx)}\big(f(y)\big)^{2^*}\,<\,
\delta^{2^*} n^{d+2^*}\,.
\end{equation}
Plugging inequalities~(\ref{qzf}) and~(\ref{qzg}) into~(\ref{qze}), we obtain
$$\sum_{y\in B(\xx)}\!
\big(f(y)\big)^{2^*}\!\leq\,
\big(
\delta^{2^*} n^{d+2^*}
\big)^{1-2/2^*}
2(c_{PS})^2\Big(\frac{1}{n^2}
\big(||f||_{2,B(\xx)}\big)^2+{\cE(f,B(\xx))}\Big)
\,.
$$
Summing this inequality over the blocks
outside of $\XX(f,\delta)$, we get
the estimate stated in the lemma.
\end{proof}
\subsection{Local averaging}
\label{loca}
In order to build a coarse grained image of the local time, we shall
perform a local averaging.
The averaging operation is deterministic, so we define it here
for
any function $f:\Zd\to\R$.
Let $M$ be an integer which is a divisor of $n$. Let $f$ be a function from $\Zd$ to $\R$.
To $f$ we associate a function $\fm$ obtained by performing a local average of $f$
over boxes of side length~$M$.
Let us define the function $\fm$ precisely.
For $\xx\in \Zd$, we define the block indexed by $\xx$ of side $M$
as
$$B'(\xx)\, =\, M\xx+\La(M)\,.$$
We define a function $\fm$ from $\Z^d$ to $\R$ which is constant on
the blocks
$B'(\xx)$, $\xx\in\Zd$, by setting
$$\forall \xx\in\Zd\quad\forall y\in
B'(\xx)\qquad
\fm(y)\,=\,\frac{1}{\big|B'(\xx)\big|}\sum_{z\in B'(\xx)}f(z)\,.$$
We have
\begin{align}
\nonumber
||\fm||_2^2&\,=\,\sum_{x\in\Zd}\big(\fm(x)\big)^2\\
&\,=\,\sum_{\xx\in\Zd}
\big|B'(\xx)\big|\Big(
\nonumber
\frac{1}{\big|B'(\xx)\big|}\sum_{z\in B'(\xx)}f(z)\Big)^2\\
&\,\leq\,
\sum_{\xx\in\Zd}
\sum_{z\in B'(\xx)}\big(f(z)\big)^2
\,=\,||f||_2^2
\,.
\label{normctl}
\end{align}
We shall next bound the $\ell^2$ norm of the difference between $f$ and $\fm$.
Let $\xx\in \Zd$.
We apply the
discrete Poincar\'e--Wirtinger inequality stated in corollary~\ref{UPW}
to the function $f$ and the block $B'(\xx)$:
$$\big(||f-\fm||_{2,B'(\xx)}\big)^2
\,\leq\,\big(M\,c_{PW}\big)^2
{\cE(f,B'(\xx))}\,.$$
We sum over $\xx\in\Zd$:
$$\big(||f-\fm||_{2}\big)^2
\,\leq\,\big(M\,c_{PW}\big)^2
\sum_{\xx\in\Zd}{\cE(f,B'(\xx))}
\,\leq\,\big(M\,c_{PW}\big)^2
{\cE(f)}\,.$$
Taking the square root, we conclude that
\begin{equation}
\label{difa}
||f-\fm||_{2}
\,\leq\,M\,c_{PW}
\sqrt{\cE(f)}\,.
\end{equation}
\section{Upper bound: proof of Theorem \ref{upbo}}
\label{S:proof_upbo}
\label{star}

We start here the proof of Theorem~\ref{upbo}.
Let $c\geq 1$ and $\kappa\geq 1$.
We consider the event
\begin{equation*}
\label{defa}
\cA\,=\,\big\{\,
\cE(f_N)\leq\kappa \,n^d\ln n,\,
|R_N|\leq cn^d\,\big\}\,.
\end{equation*}
Let $\cF$ be a collection of functions from~$\Zd$ to $\R^+$.
The estimate on the range~(\ref{range}) and
proposition~\ref{goode} yield that, for $n$ large enough,
\begin{multline}
E\big(e^{-|R_N|};
f_N\in\cF)\,\leq\,
E\big(e^{-|R_N|};
f_N\in\cF,\cA)
\cr
	\,+\,
\exp\big(-(c-2k(1,d))n^d\big)\,+\,
\exp\Big(
-\big(\frac{\kappa}{2}-4dc\big)n^d\ln n\Big)\,.
\label{inisum}
\end{multline}
From now onwards, we focus on estimating the term
$E\big(e^{-|R_N|}; f_N\in\cF,\cA)$. So
we suppose that the event~$\cA$
occurs
and we deal only with trajectories of the random walk belonging to~$\cA$.
Applying corollary~\ref{TPS}, we have
\begin{equation}
\label{afa}
||f_N||_{2^*}
\,\leq\,
c_{PS}
\sqrt{2d\cE(f_N)}\,\leq\,
c_{PS}
\sqrt{2d\kappa \,n^d\ln n}
\,.
\end{equation}
We apply next the deterministic construction of section~\ref{deco} to the function~$f_N$.
The corresponding set $\XX(f_N,\delta)$ satisfies:
\begin{equation}
\label{afb}
\big|\XX(f_N,\delta)\big|\,
\delta^{2^*} n^{d+2^*}
\,\leq\,
\big(||f_N||_{2^*}\big)^{2^*}\,.
\end{equation}
Therefore, combining inequalities~(\ref{afa}) and~(\ref{afb}),
\begin{equation}
\label{cardx}
\big|\XX(f_N,\delta)\big|\,\leq\,
\frac{1}{ \delta^{2^*} n^{d+2^*} }
\Big( c_{PS}
\sqrt{2d\kappa \,n^d\ln n}\Big)^{2^*}\,\leq\,
\frac{1}
{ \delta^{2^*}}
(c_{PS})^{2^*}
(2d\kappa\ln n)^{2^*/2}
\,.
\end{equation}
In addition, since the range $R_N$ is connected and has cardinality at most $cn^d$,
then certainly the set $\XX(f_N,\delta)$ is included in the box $\La(3cn^{d-1})$;
otherwise, the range $R_N$ would contain a vertex which is outside the box $\La(3cn^d)$ (when coming back to the original lattice, the scale is multiplied by $n$), and a connected set containing $0$ which exits the box
$\La(3cn^{d})$ must contain a path of length at least
$3cn^{d}/2$.
We denote by $\cX(N,\delta)$ the collection of the subsets $X$ of $\Zd$
satisfying these constraints.
There exists a constant $c_d$ depending on the dimension $d$ only such that, for $n$ large enough,
\begin{equation}
\label{siq}
\big|
\cX(N,\delta)
\big|\,\leq\,
\exp\big(c_dC_0\big)
\,,
\end{equation}
where
\begin{equation}
\label{saq}
C_0\,=\,
\frac{1}
{ \delta^{2^*}}
\kappa^4(\ln n)^{5}\,.
\end{equation}
We decompose the expectation according to
the value $X$ of $\XX$ in $\cX(N,\delta)$:
\begin{align}
\label{suma}
E\big(e^{-|R_N|};
f_N\in\cF,\cA)\,=\,&
\sum_{X\in \cX(N,\delta)}
E\big(e^{-|R_N|};
f_N\in\cF,\,\cA,\,\XX(f_N,\delta)=X\big)\,.
\end{align}
We fix next ${X\in \cX(N,\delta)}$ and we shall
estimate the probability appearing in
the sum.
We perform the local averaging on
$f_N$, thereby getting the function $\fmn$.
Using inequality~(\ref{difa}), we have
\begin{equation}
\label{dba}
||f_N-\fmn||_2
\,\leq\,M\,c_{PW}
\sqrt{\cE(f_N)}
\,\leq\,M\,c_{PW}
\sqrt{\kappa \,n^d\ln n}
\,.
\end{equation}
\subsection{The coarse grained profile}
\label{coars}
We build here the coarse grained image of the local time.
Since $M$ is a divisor of $n$, then each block $B(\xx)$, for $\xx\in\Zd$,
is the disjoint union of the blocks $B'(\yy)$, $\yy\in\Zd$, which are included in it.
Let us make this statement more precise.
For $\xx\in\Zd$, we denote by $Y(\xx)$ the subset of $\Zd$ defined by:
$$Y(\xx)\,=\,\big\{\,\yy\in\Zd:B'(\yy)\subset B(\xx)\,\big\}\,.$$
With this definition, we have
$$\forall \xx\in\Zd\qquad
B(\xx)\,=\,\bigcup_{\yy\in Y(\xx)}B'(\yy)\,.$$
We recall that the norm $|\cdot|_\infty$ is defined by
$$\forall (x_1,\dots,x_d)\in\Rd\qquad
|(x_1,\dots,x_d)|_\infty\,=\,
{\max_{1\leq i\leq d}|x_i|}\,.$$
For $X$ a subset of $\Zd$, we define
$$\hX\,=\,X\,\cup\,
 \set{x\in \Zd\setminus X}{\ex y\in X\quad |x-y|_\infty=1 }\,.
$$
We have the simple bound
\begin{equation}
\label{neighb}
|\hX|\,\leq \,3^d\, |X|\,.
\end{equation}
We define also
$$Y(X)\,=\,\bigcup_{\xx\in X}Y(\xx)\,.$$
Since the blocks $B'(\xx)$, $\xx\in\Zd$, are pairwise disjoint, we have
$$|Y(\hX)|\,\times\,|\La(M)\cap\Zd|\,\leq\,
|\hX|\,\times\,|\La(n)\cap\Zd|\,,$$
whence, using inequality~\eqref{neighb},
\begin{equation}
\label{heb}
|Y(\hX)|\,\leq\,
\frac{n^d}{M^d}\,|\hX|
\,\leq\,
\frac{3^dn^d}{M^d}\,|X|
\,.
\end{equation}
We take now $X=\XX(f_N,\delta)$.
The function $\fmn$ is constant on each block $B'(\yy)$, $\yy\in Y(\hX)$, and
$$\forall x\in \Zd\qquad
0\,\leq\,\fmn(x)\,\leq\,
\sup_{y\in\Zd}\,f_N(y)\,\leq\,\sqrt{N}\,.$$
We shall work in the domain
\begin{equation}
\label{bad}
D\,=\,\bigcup_{\xx\in \hX}B(\xx)\,=\,
\bigcup_{\yy\in Y(\hX)}B'(\yy)\,.
\end{equation}
We define also the set
\begin{equation}
\label{bod}
E\,=\,\bigcup_{\xx\in X}B(\xx)\,.
\end{equation}
We apply Lemma~\ref{contout}
to $f_N$. Recalling that
$$\big(||f_N||_{2}\big)^2\,=\,N\,=\,n^{d+2}\,,\qquad
\cE(f_N)\,\leq\,\kappa\,n^d\ln n\,,
$$
we obtain
\begin{equation}
\label{outd}
\sum_{x\in
\Zd\setminus E}
\big(f_N(x)\big)^{2^*}\,\leq\,
c_d\big(
\delta^{2^*} n^{d+2^*}
\big)^{1-2/2^*}
\big(
n^d+ {\kappa \,n^d\ln n}
\big)\,.
\end{equation}
We discretize next the values of the functions $\fmn$.
Let $\eta>0$. We define
the function $\fmne$ by setting
$$\forall x\in\Zd\qquad
\fmne(x)\,=\, \eta\Big\lfloor
\frac{1}{\eta}{\fmn(x)}
\Big\rfloor
\,.$$
The function $\fmne$ restricted to $D$, denoted by
$\fmne|_D$,
is the coarse--grained image of the function~$f_N$.
By construction, we have
$$\forall x\in \Zd\qquad
\big|\fmn(x)-\fmne(x)\big|\,\leq\,\eta\,,$$
whence, using inequalities~(\ref{cardx}) and~(\ref{neighb}) (recall that our $\ell^p$ norms
are unscaled, see the definition in formula~\eqref{D:lp}),
\begin{align}
\big(||\fmn-\fmne||_{2,D}\big)^2&\,\leq\,
\eta^2\,|\hX|\,n^d\cr
&\,\leq\,
\eta^2\,3^d\,
\frac{1}
{ \delta^{2^*}}
(c_{PS})^{2^*}
(\kappa\ln n)^{2^*/2}
\,n^d\,.
\label{zed}
\end{align}
Putting together inequalities~(\ref{dba}) and~(\ref{zed}), we see that,
on the event
$\XX(f_N,\delta)=X$,
we have
\begin{equation}
\label{rty}
||f_N-\fmne||_{2,D}\,<\,\Gamma_0\,,
\end{equation}
where
\begin{align}
\Gamma_0\,=\,
M\,c_{PW}
\sqrt{\kappa \,n^d\ln n}
\,+\,
\Big(\eta^2\,3^d\,
\frac{1}
{ \delta^{2^*}}
(c_{PS})^{2^*}
(\kappa\ln n)^{2^*/2}
\,n^d\Big)^{1/2}\,.
\end{align}
The function $\fmne|_D$
belongs to
the collection
$\cG$
of the functions which are constant on each block
$B'(\yy),\,\yy\in Y(\hX)$, and with values in the set
$$\Big\{\,k\eta:k\in\N,\,0\leq k< \frac{\sqrt{N}}{\eta}\,\Big\}\,.$$
Using inequality~\eqref{heb},
a simple upper bound on the cardinality of $\cG$ is given by
\begin{equation}
\label{sib}
|\cG|
\,\leq\,
\Big(\frac{\sqrt{N}}{\eta}\Big)^{\displaystyle|Y(\hX)|}
\,\leq\,
\Big(\frac{\sqrt{N}}{\eta}\Big)^{\displaystyle
\frac{3^dn^d}{M^d}\,|X|}
\,.
\end{equation}
Notice that the collection $\cG$ depends on $f_N$ only through the set of
blocks
$\XX(f_N,\delta)$.
More precisely, once we know that
$\XX(f_N,\delta)=X$, then $\cG$ depends only on $X$ and the parameters $M,\eta$,
so we write $\cG=\cG(X,M,\eta)$.
We come back to equation~(\ref{suma}) and we decompose further the expectation
as follows:
\begin{multline}
\label{sumb}
E\big(e^{-|R_N|};f_N\in\cF,\,\cA,\,\XX(f_N,\delta)=X\big)
\,\leq\,\\
\sum_{g\in \cG(X,M,\eta)}
E\big(e^{-|R_N|};
f_N\in\cF,\,\cA,\,\XX(f_N,\delta)=X,\,
\fmne|_D=g
\big)\,.
\end{multline}
Our large deviations inequality will involve the local time $L_N$, so
we try to estimate $||L_N-g^2||_{1,D}$ once we know that $\fmne|_D=g$.
Let $g\in\cG(X,M,\eta)$.
By the Cauchy--Schwarz inequality,
$$||L_N-g^2||_{1,D}\,=\,
||f_N^2-g^2||_{1,D}\,\leq\,
||f_N-g||_{2,D}\,
||f_N+g||_{2,D}\,.$$
Suppose that $\fmne|_D=g$.
Thanks to inequality~\eqref{rty}, we have
\begin{equation}
\label{bfg}
||g||_{2,D}\,\leq\,||f_N||_{2,D}+
||g-f_N||_{2,D}\,\leq\,
\sqrt{N}+\Gamma_0\,.
\end{equation}
Setting $\Gamma_1=\Gamma_0(2\sqrt{N}+\Gamma_0)$,
we deduce from the previous inequalities that
\begin{equation}
\label{cfg}
||L_N-g^2||_{1,D}\,\leq\,
\Gamma_1\,.
\end{equation}
Inequality~\eqref{sumb} and the above inequalities yield that
\begin{multline}
\label{sumc}
E\big(e^{-|R_N|};
f_N\in\cF,\,\cA,\,\XX(f_N,\delta)=X\big)
\,\leq\,\\
\sum_{\tatop{g\in \cG(X,M,\eta)}{
||g||_{2,D}\leq \sqrt{N}+\Gamma_0}}
E\big(e^{-|R_N|};
f_N\in\cF,\,\cA,\,\XX(f_N,\delta)=X,
\,||L_N-g^2||_{1,D}<\Gamma_1
\big)\,.
\end{multline}

\subsection{Continuation of proof of Theorem \ref{upbo}}
\label{contin}
We still fix ${g\in \cG(X,M,\eta)}$.
We come back to equation~\eqref{sumc} and we estimate the expectation
$$E\big(e^{-|R_N|};f_N\in\cF,\,\cA,\,\XX(f_N,\delta)=X,
\,||L_N-g^2||_{1,D}<\Gamma_1\big)\,.$$
We need first to bound from below $|R_N|$ when $L_N$ is close to $g^2$.
Let $\lambda>0$. If $x\in D$ is such that $g^2(x)\geq\lambda$
and $|L_N(x)-g^2(x)|<\lambda/2$, then certainly $L_N(x)>0$.
Therefore
\begin{align}
|R_N|& \geq\,\big|\,\big\{\,x\in D:g^2(x)\geq\lambda,\,
|L_N(x)-g^2(x)|<\lambda/2\,\big\}\,\big| \nonumber \\
&\,\geq\,
\big|\,\big\{\,x\in D:g^2(x)\geq\lambda\,\big\}\big|-
\big|\,\big\{\,x\in D:
|L_N(x)-g^2(x)|\geq\lambda/2\,\big\}\,\big|\,. \label{RNtriv}
\end{align}
Moreover, by Markov's inequality,
$$\big|\,\big\{\,x\in D:
|L_N(x)-g^2(x)|\geq\lambda/2\,\big\}\,\big|\,\leq\,
\frac{2}{\lambda}||L_N-g^2||_{1,D}\,.
$$
We conclude that
\begin{multline}
\label{imc}
E\big(e^{-|R_N|};f_N\in\cF,\,\cA,\,\XX(f_N,\delta)=X,
\,||L_N-g^2||_{1,D}<\Gamma_1\big)
\,\leq\,\qquad\qquad\\
\exp\Big(-
\big|\,\big\{\,x\in D:g^2(x)\geq\lambda\,\big\}\big|
+\frac{2\Gamma_1}{\lambda}
\Big)\,\times P(\cB)\,,
\end{multline}
where $\cB$ is the event defined by
\begin{equation}\label{E:cB}
\cB\,=\,\Big\{\,
f_N\in\cF,\,\cA,\,\XX(f_N,\delta)=X,
\,||L_N-g^2||_{1,D}<\Gamma_1\,\Big\}
\,.
\end{equation}
It remains to estimate the probability $P(\cB)$, for which we ultimately wish to use Theorem \ref{GD}.
 Recall the set $E$ from \eqref{bod}.
Then we have
\begin{equation}
\label{prel}
||L_N||_{1,\tD}\,\geq\,
||L_N||_{1,E}\,=\,
||L_N||_{1}\,-\,
||L_N||_{1,\Zd\setminus E}\,.
\end{equation}
We estimate
$||L_N||_{1,\Zd\setminus E}$ in the next lemma.
\begin{lemma}
\label{exte}
For $n$ large enough,
we have
$||L_N||_{1,\Zd\setminus E}\leq\Gamma_2$ where
\begin{equation*}
\Gamma_2
\,=\,c\,\delta^{4/d}\,
n^{d+2}
(c_d{\kappa}\ln n)^{2/2^*}\,.
\end{equation*}
\end{lemma}
\begin{proof}
We use H\"older's inequality with the exponents $2^*/2$ and $d/2$ to write
$$\displaylines{
||L_N||_{1,\Zd\setminus E}
\,\leq\,
\sum_{\xx\in
\Zd\setminus X}
\sum_{y\in B(\xx)}
\big(f_N(y)\big)^{2}\,\leq\,
\Big(\sum_{\xx\in
\Zd\setminus X}
\sum_{y\in B(\xx)}
\big(f_N(y)\big)^{2^*}\Big)^{2/2^*}
\big|
\supp f_N
\setminus
E
\big|^{1-2/2^*}\,.}
$$
Moreover, on the event~$\cA$, we have
$$\big|
\supp f_N
\setminus
E
\big|\,\leq\,
\big|
\supp f_N
\big|\,\leq\,cn^d\,.$$
Using also inequality~\eqref{outd},
we conclude that
$||L_N||_{1,\Zd\setminus E}\leq\Gamma_2$.
\end{proof}

\noindent
It follows from inequality~\eqref{prel} and Lemma~\ref{exte} that
\begin{equation*}
P\big(\cB)
\,=\,
P\big(
\cB,
||L_N||_{1,\tD}\,\geq\,N-\Gamma_2
\big)
\,=\,
\sum_{N-\Gamma_2\leq t\leq N}
P\big(
\cB,\,
||L_N||_{1,\tD}=t
\big)\,.
\end{equation*}
Now, if
$||L_N||_{1,\tD}=t$, then $\tau(D,t)<\infty$ and
$L_N(y)=L_t^{\tD}(y)$ for any $y\in \tD$, thus
$$ ||L_N-g^2||_{1,\tD}\,=\,
||L_t^{\tD}-g^2||_{1,\tD}\,,$$
and we obtain the bound
\begin{equation}
P\big(\cB\big)\,\leq\,
\sum_{N-\Gamma_2\leq t\leq N}
P\big(
f_N\in\cF,\,\cA,\,\XX(f_N,\delta)=X,
||L_t^{\tD}-g^2||_{1,\tD}<\Gamma_1
,\,\tau<\infty
\big)\,.
\label{intera}
\end{equation}
Let us fix $t$ such that
\begin{equation}
\label{contg}
N-\Gamma_2\leq t\leq N\,.
\end{equation}
The time has now come to make specific choices for the parameters
$\delta,M,\eta,\lambda$ introduced in the course of the proof (recall that $\lambda$ was introduced in~\eqref{RNtriv}).
We suppose that
\begin{equation}
\delta\,=\,\frac{1}{n^\alpha}\,,\quad
M\,=\,n^\beta\,,\quad
\eta\,=\,\frac{1}{n^\gamma}\,,\quad
\lambda\,=\,{n^\rho}\,,
\label{expos}
\end{equation}
where $\alpha,\,\beta,\,\gamma,\,\rho$ are positive exponents, which
satisfy furthermore
\begin{align}
\label{furco}
	&\beta<1\,,\quad \frac{2^*}{2}\alpha-\gamma<\beta\,,\quad
2^*\alpha<d\beta\,,
\cr
	&1+\beta<\rho\,,\quad
2-4\alpha/d<\rho\,,\quad
\rho<2\,.
\end{align}
These conditions imply that, as $N$ or $n$ goes to $\infty$,
\begin{equation}
\label{dxr}
\Gamma_0\,\lsim\,n^{d/2+\beta}\,,\quad
\Gamma_1\,\lsim\,n^{d+1+\beta}\,,\quad
\Gamma_2\,\lsim\,n^{d+2-4\alpha/d}\,,
\end{equation}
where $\lsim$ means that the logarithms are equivalent.
In particular, we have
\begin{equation}
\label{der}
\frac{\Gamma_0}{\sqrt N}\to0\,,\qquad
\frac{\Gamma_1}{N}\to0\,,\qquad
\frac{\Gamma_2}{N}\to0\,,\qquad
\frac{n^{d+1}}{\Gamma_1}\to0\,,
\end{equation}
whence for $n$ large enough
\begin{equation}
\label{aer}
\Gamma_0<\sqrt{N}\,,\qquad
\Gamma_1<3\Gamma_0\sqrt{N}\,,\qquad
	n^{d+1}<\Gamma_1<n^{d+2}\,,\qquad
\Gamma_2<N/2\,.
\end{equation}
We apply Proposition~\ref{depe} with $r=\Gamma_1/t$,
and $k=\lfloor\Gamma_1\rfloor$: for any $x\in\bD$, we have
\begin{align*}
P\big(
&||L_t^{\tD}-g^2||_{1,\tD}<\Gamma_1
,\,\tau<\infty
\big)\,=\,
P\Big(
\Big|\Big|\frac{1}{t}L_t^{\tD}-\frac{1}{t}g^2\Big|\Big|_{1,\tD}<
\frac{\Gamma_1}{t}
,\,\tau<\infty
\Big)\cr
\,&\leq\,
\frac{\displaystyle 2\,
P_x\Big(
\Big|\Big|\frac{1}{t}L_t^D-
\frac{1}{t}g^2
\Big|\Big|_{1,D}\,\leq\,
\frac{\Gamma_1}{t}
+4\frac{\Gamma_1}{t-\Gamma_1}
,\,\tau<\infty
\Big)}
{\displaystyle P_x(S_{
\lfloor\Gamma_1\rfloor
}=0)+ P_x(S_{
\lfloor\Gamma_1\rfloor
-1}=0)}\,.
\end{align*}
Moreover, using~\eqref{contg} and~\eqref{der},
for $n$ large enough,
$$
\frac{\Gamma_1}{t}
+4\frac{\Gamma_1}{t-\Gamma_1}\,\leq\,
\frac{5\Gamma_1}{t-\Gamma_1}
\,\leq\, \frac{7\Gamma_1}{t}
\,.$$
Since $|R_N|\leq cn^d$, then the range $R_N$ is included in the box $\La(2cn^d)$ and its trace on the renormalized lattice is included in
the box $\La(2cn^{d-1})$.
The set $\bD$ is obtained by enlarging slightly this trace and then coming back to the original lattice, hence it is included in the box
$\La(3cn^{d})$.
We bound the denominator with the help of
Lemma~\ref{afaire} (notice that $n^{d+1}<\Gamma_1<n^{d+2}$) and we take the infimum over $x\in\bD$:
\begin{align}
\label{tagd}
P\big(
||L_t^{\tD}-g^2||_{1,\tD}<\Gamma_1
,\,\tau<\infty
\big)
\,\leq\,
2 \,\exp\Big(c'\frac{n^{2d}}{\lfloor\Gamma_1\rfloor}\Big)
\inf_{x\in \bD}
P_x\Big(
\Big|\Big|\frac{1}{t}L_t^{\tD}-\frac{1}{t}g^2\Big|\Big|_{1,\tD}<
\frac{7\Gamma_1}{t}
,\,\tau<\infty
\Big)\,.
\end{align}
For $k\in\N$ and $r>0$, we define
the closed convex set
$$C(g,k,r)\,=\,\Big\{\,
h\in\ell^1(D):
\Big|\Big|h-\frac{1}{k}g^2\Big|\Big|_{1,\tD}\leq
r
\,\Big\}\,.$$
We apply the deviation inequality of Theorem~\ref{GD} to the set
$C(g,t, {7\Gamma_1}/{t})$:
\begin{equation}
\label{vagd}
\inf_{x\in \bD}
P_x\Big(
\Big|\Big|\frac{1}{t}L_t^{\tD}-\frac{1}{t}g^2\Big|\Big|_{1,\tD}\leq
\frac{7\Gamma_1}{t}
,\,\tau<\infty
\Big)\,\leq\,
\exp\Big(
-t\,\inf_{h\in
C(g,t, {7\Gamma_1}/{t})
}\,\frac{1}{2}\cE(\sqrt{h},D) \Big)
\,.
\end{equation}
The previous inequalities~\eqref{tagd} and~\eqref{vagd}
yield that
\begin{equation}
\label{agd}
P\big(
||L_t^{\tD}-g^2||_{1,\tD}<\Gamma_1
,\,\tau<\infty
\big)
\,\leq\,
2 \, \exp\Big(
c'\frac{n^{2d}}{\lfloor\Gamma_1\rfloor}
-t\,\inf_{h\in
C(g,t, {7\Gamma_1}/{t})
}\,\frac{1}{2}\cE(\sqrt{h},D) \Big)
\,.
\end{equation}
We deal next with the infimum in the exponential.
Our first goal is to obtain a bound which depends on $N$ (and not on $t$).
Let $h\in
C(g,t, {7\Gamma_1}/{t})$.
We have
\begin{align*}
\Big|\Big|h-\frac{1}{N}g^2\Big|\Big|_{1,D}
\,&\leq\,
\Big|\Big|h-\frac{1}{t}g^2\Big|\Big|_{1,D}\,+\,
\Big|\Big| \frac{1}{t}g^2-
\frac{1}{N}g^2\Big|\Big|_{1,D}
\\
\,&\leq\,
\frac{7\Gamma_1}{t}+
\Big| \frac{1}{t}-
\frac{1}{N}\Big|||g^2||_{1,D}\,.
\end{align*}
Now, from~\eqref{contg}, we have
$$\Big| \frac{1}{t}-
\frac{1}{N}\Big|\,=\,
 \frac{N-t}{Nt}
\,\leq\,
\frac{\Gamma_2}{N(N-\Gamma_2)}\,.$$
Therefore, using the inequalities~\eqref{bfg}, \eqref{contg} and~\eqref{aer}, we get
\begin{align}
\label{nbv}
\Big|\Big|h-\frac{1}{N}g^2\Big|\Big|_{1,D}
&\,\leq\,
\frac{7\Gamma_1}{t}+
\frac{\Gamma_2}{N(N-\Gamma_2)}
(\Gamma_0+ \sqrt{N})^2\cr
&\,\leq\,
\frac{7\Gamma_1}{N-\Gamma_2}+
\frac{4\Gamma_2}{N-\Gamma_2}
\,\leq\,
\frac{14\Gamma_1}{N}+
\frac{8\Gamma_2}{N}
\,.
\end{align}
Setting
$$r(N)\,=\,
\frac{14\Gamma_1}{N}+
\frac{8\Gamma_2}{N}\,,$$
we conclude that
$$C\Big(g,t, \frac{7\Gamma_1}{t}\Big)
\subset
C\big(g,N, r(N)\big)
\,.$$
Recalling that $t\geq N-\Gamma_2$, inequality~\eqref{agd}
now implies
\begin{multline*}
P\big(
||L_t^{\tD}-g^2||_{1,\tD}<\Gamma_1,\,\tau<\infty
\big) \,\leq\,\cr
 2 \, \exp\Big(
c'\frac{n^{2d}}{\lfloor\Gamma_1\rfloor}
-(N-\Gamma_2)\,\inf\,\Big\{\,
\,\frac{1}{2}\cE(\sqrt{h},D)
:h\in C\big(g,N, r(N)\big)
\,\Big\} \Big) \,.
\end{multline*}
The good point is that this upper bound does not depend any more on~$t$.
Plugging this upper bound
in equations~\eqref{imc} and~\eqref{intera}, we obtain
\begin{multline}
\label{zbv}
E\big(e^{-|R_N|};\cB \big)
\,\leq\,
	2(\Gamma_2+1)
\exp\Big(-
\big|\,\big\{\,x\in D:g^2(x)\geq\lambda\,\big\}\big|
+\frac{2\Gamma_1}{\lambda}
\qquad
\cr
+
c'\frac{n^{2d}}{\lfloor\Gamma_1\rfloor}
-(N-\Gamma_2)\,\inf\,\Big\{\,\frac{1}{2}\cE(\sqrt{h},D):
h\in C\big(g,N, r(N)\big)
\,\Big\}\Big)
\,.\qquad
\end{multline}
Our next goal is to remove the dependence on~$g$ in the upper bound.
We wish to obtain an upper bound which depends only on the set~$\cF$.
To do so, we shall control the Dirichlet energy
$\cE(\sqrt{h},D)$ restricted to~$D$ with the help of the
Dirichlet energy
$\cE(\sqrt{h})$ in the whole space. Of course, this creates a correcting
factor, which we study in the next section.
\subsection{Truncation}
\label{trun}
We first define a truncation operator associated to the set $X$.
Let $\phi:\Rd\to [0,1]$ be a function such that

\noindent
$\bullet$ $\supp\phi\subset\La(7/4)$,

\noindent
$\bullet$ $\phi$ is piecewise affine on $\La(7/4)\setminus\La(5/4)$,

\noindent
$\bullet$ $\phi$ is equal to $1$ on $\La(5/4)$,

\noindent
$\bullet$ the gradient of $\phi$ has Euclidean norm less than $4$
on $\La(7/4)\setminus\La(5/4)$.

\noindent
To the set $X$, we associate the function $\phi_X$ defined by
$$\forall y\in\Zd\qquad \phi_X(y)\,=\,\max\,\big\{\,
\phi(y/n-x):
{x\in X}\,\big\}
\,.$$
By construction, we have
\begin{equation}
\label{seein}
\supp\phi_X\,\subset\,
\bigcup_{x\in X}\Lambda(nx,2n)\,\subset\, D\,,
\end{equation}
where the domain $D$ was defined in~\eqref{bad}.
Let $f$ be any function in $\ell^1(\Zd)$.
We define the function $\phi_Xf$ by setting
$$
\forall y\in\Zd
\qquad (\phi_Xf)(y)\,=\,\phi_X(y) f(y)\,.$$
If $f$ is a function in $\ell^1(D)$, we extend $f$ outside $D$ by setting
$f(x)=0$ for $x\in\Zd\setminus D$ and the previous definition still makes sense.
We recall that the set $E$ is defined in~\eqref{bod}.
\begin{proposition}
\label{depa}
Let $f$ be any function in $\ell^1(D)$. We have the inequality
$$\cE(f,D)\,\geq\,\Bigg(
	\max\Big(\sqrt{\cE(\phi_Xf)}-
\frac{4}{n}
	||f||_{2,D\setminus E},0\Big)\Bigg)^2\,.$$
\end{proposition}
\begin{proof}
Since the support of $\phi_Xf$ is included in~$D$ (see the inclusion~\eqref{seein}), then we have
\begin{align*}
\cE(\phi_Xf)\,&=\,
\cE(\phi_Xf,D)
\,=\,
\frac{1}{2d}
\!\!
\sum_{\tatop{y,z\in D}{|y-z|=1}}
\Big(\phi_X(y)f(y)-\phi_X(z)f(z)\Big)^2\cr
\,&=\,
\frac{1}{2d}
\!\!
\sum_{\tatop{y,z\in D}{|y-z|=1}}
\Big(\phi_X(y)\big(f(y)-f(z)\big)-\big(\phi_X(z)-\phi_X(y)\big)f(z)\Big)^2\cr
\,&=\,
\frac{1}{2d}
\!\!
\sum_{\tatop{y,z\in D}{|y-z|=1}}
	\Big(
\big(\phi_X(y)\big)^2\big(f(y)-f(z)\big)^2+
\big(\phi_X(z)-\phi_X(y)\big)^2(f(z))^2\cr
&\kern70pt -2
	\phi_X(y)\big(f(y)-f(z)\big)\big(\phi_X(z)-\phi_X(y)\big)f(z)\Big)\,.
\end{align*}
From the definition of $\phi_X$, it follows that if $|y-z|=1$, then
$|\phi_X(z)-\phi_X(y)|\leq 4/n$, and
$\phi_X(z)=\phi_X(y)$
unless both $y,z$ are in $D\setminus E$. Therefore we have the
following bound:
\begin{align*}
\cE(\phi_Xf)
\,&\leq\,\cE(f,D)+\Big(\frac{4}{n}\Big)^2
\kern -3pt
\sum_{z\in D\setminus E} f(z)^2
+\frac{4}{nd}\kern -3pt
\sum_{\tatop{y,z\in D\setminus E}{|y-z|=1}}
\kern -3pt
\big| f(y)-f(z)\big|\big|f(z)\big|\,.
\end{align*}
Using the Cauchy--Schwarz inequality, we get
$$\displaylines{
\sum_{\tatop{y,z\in D\setminus E}{|y-z|=1}}
\big| f(y)-f(z)\big|\big|f(z)\big|
\,\leq\,
\Big(\sum_{\tatop{y,z\in D\setminus E}{|y-z|=1}}
\big( f(y)-f(z)\big)^2
\sum_{\tatop{y,z\in D\setminus E}{|y-z|=1}}
\big( f(z)\big)^2\Big)^{\frac{1}{2}}\hfill\cr
\,\leq\,
\,
\sqrt{2d \cE(f,D\setminus E)}
\sqrt{2d}
\, ||f||_{2,D\setminus E}
\,=
\,{2d}
\,\sqrt{\cE(f,D)}
\, ||f||_{2,D\setminus E}
\,.
}$$
Reporting in the previous inequality, we have
\begin{align*}
\cE(\phi_Xf)
\,&\leq\,\cE(f,D)+
\frac{16}{n^2}
\big(||f||_{2,D\setminus E}\big)^2
+\frac{8}{n}
\sqrt{\cE(f,D)}
||f||_{2,D\setminus E}\cr
\,&=\,
\Big(
\sqrt{\cE(f,D)}+
\frac{4}{n}
||f||_{2,D\setminus E}\Big)^2\,,
\end{align*}
from which we deduce easily the inequality stated in the proposition.
\end{proof}
\subsection{End of proof of Theorem \ref{upbo}}
\label{conc}
We come back to inequality~\eqref{zbv}. We wish to obtain an upper bound which depends
on the set~$\cF$ and not on $g$. We know that $f_N$ belongs to~$\cF$,
so we pick a function $h$ in
$C\big(g,N, r(N)\big)$ and we
try to control the distance between
$\frac{1}{N}L_N$ and $\phi_X^2h$. We write
\begin{align*}
\Big|\Big|\frac{1}{N}L_N-\phi_X^2h\Big|\Big|_{1}
\,\leq\,
\Big|\Big|\frac{1}{N}L_N-
\frac{1}{N}
\phi_X^2
L_N
\Big|\Big|_{1}
+
\Big|\Big|\frac{1}{N}
\phi_X^2
L_N-
\phi_X^2
h\Big|\Big|_{1}
\end{align*}
and we control separately each term.
For the first term, we use the fact that $\phi_X$ is equal to $1$ on $E$ and
lemma~\ref{exte} to get
\begin{equation}
\label{conta}
\Big|\Big|\frac{1}{N}L_N-
\frac{1}{N}
\phi_X^2
L_N
\Big|\Big|_{1}
\,\leq\,
||\frac{1}{N}L_N||_{1,\Zd\setminus E}
\,\leq\,
\frac{\Gamma_2}{N}\,.
\end{equation}
For the second term, we have, recalling that
$\supp\,\phi_X\subset D$ and the definition of $\cB$ in \eqref{E:cB}, and using~\eqref{nbv},
\begin{align}
\Big|\Big|\frac{1}{N}
\phi_X^2
L_N-
\phi_X^2
 h\Big|\Big|_{1}\,&\leq\,
\Big|\Big|\frac{1}{N}L_N- h\Big|\Big|_{1,D}\cr
\,&\leq\,
\Big|\Big|\frac{1}{N}L_N- \frac{1}{N}g^2\Big|\Big|_{1,D}+
\Big|\Big|\frac{1}{N}g^2- h\Big|\Big|_{1,D}\cr
\,&\leq\,
\frac{\Gamma_1}{{N}}\,+\,r(N)
\,.
\label{gui}
\end{align}
Inequalities~\eqref{conta} and~\eqref{gui} together yield
\begin{align}
\label{pouf}
\Big|\Big|\frac{1}{N}L_N-
\phi_X^2
h\Big|\Big|_{1}\,\leq\,
\frac{\Gamma_2}{{N}}\,+\,
\frac{\Gamma_1}{{N}}\,+\,
r(N)
\,.
\end{align}
Furthermore, we have, by inequalities~\eqref{cfg} and~\eqref{pouf},
\begin{align}
\label{aie}
\Big|\Big|\frac{1}{N}g^2-
\phi_X^2
h\Big|\Big|_{1,D}&\,\leq\,
\Big|\Big|\frac{1}{N}g^2-\frac{1}{N}L_N\Big|\Big|_{1,D}\,+\,
\Big|\Big|\frac{1}{N}L_N-
\phi_X^2
h\Big|\Big|_{1,D}\cr
&\,\leq\,
\frac{\Gamma_1}{{N}}\,+\,
\frac{\Gamma_2}{{N}}\,+\,
\frac{\Gamma_1}{{N}}\,+\,
r(N)\,.
\end{align}
Let us set
\begin{equation}
\label{ultg}
\Gamma_3\,=\,
\frac{\Gamma_2}{{N}}\,+\,
2\frac{\Gamma_1}{{N}}\,+\,
r(N)
\,\leq\,
\frac{16\Gamma_1}{{N}}\,+\,
\frac{9\Gamma_2}{{N}}
\,.
\end{equation}
We work next on the first term in the
exponential appearing in~\eqref{zbv}.
For any $f\in\ell^1(D)$, we have
$$\displaylines{
\Big|\,\Big\{\,x\in D:g^2(x)\geq\lambda\,\Big\}\,\Big|\,\geq\,
\Big|\,\Big\{\,x\in D:f(x)\geq\frac{2\lambda}{N},\,
\big|f(x)-\frac{g^2(x)}{N}\big|<\frac{\lambda}{N}\,\Big\}\,\Big|\cr
\,\geq\,
\Big|\,\Big\{\,x\in D:
f(x)\geq\frac{2\lambda}{N}
\,\Big\}\Big|-
\Big|\,\Big\{\,x\in D:
\big|f(x)-\frac{1}{N}g^2(x)\big|\geq\frac{\lambda}{N}
\,\Big\}\,\Big|\,.
}$$
Moreover, by Markov's inequality,
$$\Big|\,\Big\{\,x\in D:
\big|f(x)-\frac{1}{N}g^2(x)\big|\geq\frac{\lambda}{N}
\,\Big\}\,\Big|\,\leq\,
\frac{N}{\lambda}\Big|\Big|
f-\frac{1}{N}g^2
\Big|\Big|_{1,D}
\,,
$$
whence
$$
\Big|\,\Big\{\,x\in D:g^2(x)\geq\lambda\,\Big\}\,\Big|\,\geq\,
\Big|\,\Big\{\,x\in D:
f(x)\geq\frac{2\lambda}{N}
\,\Big\}\Big|\,-\,
\frac{N}{\lambda}\Big|\Big|
f-\frac{1}{N}g^2
\Big|\Big|_{1,D}\,.
$$
Let $h \in C(g,N, r(N))$ and
let us apply this inequality with $f=\phi_X^2h$.
Together with inequality~\eqref{aie}, we obtain
\begin{align*}
\Big|\,\Big\{\,x\in D:g^2(x)\geq\lambda\,\Big\}\,\Big|\,\geq\,
\Big|\,\Big\{\,x\in D:
(\phi_X^2h)(x)\geq\frac{2\lambda}{N}
\,\Big\}\Big|\,-\,
\frac{N}{\lambda}\Gamma_3\,.
\end{align*}
Taking the supremum over $h$
in $C(g,N, r(N))$,
we get
\begin{align*}
\Big|\,\Big\{\,x\in D:g^2(x)\geq\lambda\,\Big\}\,\Big|\,\geq\,
\,\sup\,\Big\{\,
\Big|\,\Big\{\,x\in D:
(\phi_X^2h)(x)\geq\frac{2\lambda}{N}
\,\Big\}\Big|:
h\in C\big(g,N, r(N)\big)\,\Big\}
-\frac{N}{\lambda}\Gamma_3\,.
\end{align*}
Plugging this into inequality~\eqref{zbv},
we arrive at
\begin{multline*}
E\big(e^{-|R_N|};\cB \big)
 \leq\,
	2(\Gamma_2+1)
\exp\Big(
\frac{N}{\lambda}\Gamma_3
+\frac{2\Gamma_1}{\lambda}
+
c'\frac{n^{2d}}{\lfloor\Gamma_1\rfloor}
\\
 -
\sup\,\Big\{\,
\Big|\,\Big\{\,x\in D:
(\phi_X^2h)(x)\geq\frac{2\lambda}{N}
\,\Big\}\Big|
: h\in
C\big(g,N, r(N)\big)
\,\Big\}
\cr
	-(N-\Gamma_2)\,
\inf\,\Big\{\,
	\frac{1}
	{2}\cE(\sqrt{h},D)
: h\in
C\big(g,N, r(N)\big)
\,\Big\}
	\Big)
	\cr
 \leq\,
	2(\Gamma_2+1)
\exp\Big(
\frac{N}{\lambda}\Gamma_3
+\frac{2\Gamma_1}{\lambda}
+
c'\frac{n^{2d}}{\lfloor\Gamma_1\rfloor}
\\
 \quad -
\inf\,\Big\{\,
\Big|\,\Big\{\,x\in D:
(\phi_X^2h)(x)\geq\frac{2\lambda}{N}
\,\Big\}\Big|+
\frac{N-\Gamma_2}{2}\cE(\sqrt{h},D)
: h\in
C\big(g,N, r(N)\big)
\,\Big\}\Big)
\,.
\end{multline*}
Let again $h\in
C\big(g,N, r(N)\big)$.
We apply Proposition~\ref{depa} to the function $f=\sqrt{h}$:
$$\cE(\sqrt{h},D)\,\geq\,
\Bigg(\max
\Big( \sqrt{\cE(\phi_X\sqrt{h})}-
\frac{4}{n}
||\sqrt{h}||_{2,D\setminus E},0\Big)\Bigg)^2\,.$$
Now, using inequalities~\eqref{conta}, \eqref{gui} and~\eqref{ultg}, we get
\begin{align}
\big(||\sqrt{h}||_{2,D\setminus E}\big)^2=
||{h}||_{1,D\setminus E} \leq
\Big|\Big|
h-\frac{1}{N}L_N
\Big|\Big|_{1,D\setminus E}
+
\Big|\Big|
\frac{1}{N}L_N
\Big|\Big|_{1,\Zd\setminus E}
\,\leq\,\frac{\Gamma_1}{N}+r(N)+\frac{\Gamma_2}{N}
\,\leq\,\Gamma_3\,.
\end{align}
We conclude that
\begin{multline}
E\big(e^{-|R_N|};
	{\mathcal B}
	\big)\,\leq\,
	2(\Gamma_2+1)
\exp\Big(
\frac{N}{\lambda}\Gamma_3
+\frac{2\Gamma_1}{\lambda}
+ c'\frac{n^{2d}}{\lfloor\Gamma_1\rfloor}
\hfill\cr
-\inf\,\Big\{\,
\Big|\,\Big\{\,x\in D:
(\phi_X^2h)(x)\geq\frac{2\lambda}{N}
\,\Big\}\Big|+
\hfill\cr
\frac{N-\Gamma_2}{2}
\Big(\max\Big( \sqrt{\cE(\phi_X\sqrt{h})}-
\frac{4\sqrt{\Gamma_3}}{n},0\Big) \Big)^2 :
h\in
C\big(g,N,r(N) \big)
\,\Big\}\Big)
\,.
\label{interm}
\end{multline}
Our next goal is to remove the term $2\lambda/N$ appearing in the infimum.
Let us fix again $h\in C\big(g,N,r(N) \big)$.
We define an auxiliary function $\wh$ by setting
$$\forall x \in\Zd\qquad\wh(x)\,=\,\Bigg(\max\Big(\phi_X\sqrt{h}(x)-
\sqrt{\frac{2\lambda}{N}},0\Big)\Bigg)^2\,.$$
Let also define
$$A\,=\,\Big\{\,x\in\Zd:
\phi_X\sqrt{h}(x)>
\sqrt{\frac{2\lambda}{N}}\,\Big\}\,=\,
\Big\{\,x\in\Zd:
\wh(x)>0
\,\Big\}
\,.$$
Obviously, we have $A\subset D$ (since $\supp \phi_X\subset D$) and
$$
\big|A\big|\,\leq\,
\Big|\big\{\,x\in D:\phi_X^2h(x)\geq
{\frac{2\lambda}{N}} \,\big\}\Big|\,.$$
We have then
\begin{align*}
\cE\big(\sqrt{\wh}\big) & =\,
\frac{1}{2d}
\sum_{\tatop{y,z\in \Zd}{|y-z|=1}}
\Big(\sqrt{{\wh}(y)}-\sqrt{{\wh}(z)}\Big)^2\hfill\\
&=\,
\frac{1}{2d}
\sum_{\tatop{y,z\in A}{|y-z|=1}}
\Big(\phi_X\sqrt{h}(y)-\phi_X\sqrt{h}(z)\Big)^2
\,+\,
\frac{1}{d}
\sum_{\tatop{y\in A,z\not\in A}{|y-z|=1}}
\Big(\phi_X\sqrt{h}(y)-
\sqrt{\frac{2\lambda}{N}}
\Big)^2\\
&\leq\,
\cE\big(\phi_X\sqrt{h}\big)\,.
\end{align*}
We conclude that
\begin{multline}
\Big|\,\Big\{\,x\in D:
(\phi_X^2h)(x)\geq\frac{2\lambda}{N}
\,\Big\}\Big|+
\frac{N-\Gamma_2}{2}
\Big(\max
\big( \sqrt{\cE(\phi_X\sqrt{h})}-
\frac{4\sqrt{\Gamma_3}}{n},0 \big) \Big)^2
\cr
\,\geq\,
\Big|\big\{\,x\in\Zd:\wh(x)>0\,\big\}\Big|
+\frac{N-\Gamma_2}{2}
\Big(\max\big(
\sqrt{\cE\big(\sqrt{\widetilde h}\big)}-
\frac{4\sqrt{\Gamma_3}}{n},0\big) \Big)^2 \,.
\label{qfrt}
\end{multline}
We evaluate next the distance between $\wh$ and $L_N/N$.
To that end, we introduce the function
$$\forall x \in\Zd\qquad
\widetilde{\frac{1}{N}L_N}(x)\,=\,\Bigg(\max\Big(
{\frac{1}{\sqrt N}f_N}(x)-
\sqrt{\frac{2\lambda}{N}},0\Big)\Bigg)^2\,,$$
and we write
\begin{equation}
\Big|\Big|\wh-\frac{1}{N}L_N\Big|\Big|_{1}\,\leq\,
\Big|\Big|\wh-\widetilde{\frac{1}{N}L_N}\Big|\Big|_{1}\,+\,
\Big|\Big|
{\widetilde {\frac{1}{N}L_N}}-
\frac{1}{N}L_N\Big|\Big|_{1}\,.
\label{tgu}
\end{equation}
We shall control the first term with the help of the following lemma.
\begin{lemma}
\label{eaxe}
For any $s,t>0$, any $a>0$, we have
$$\Big|
\Big(\max\big( \sqrt{t}-a,0\big)\Big)^2
-
\Big(\max\big(\sqrt{s}-a,0\big)\Big)^2
\Big|\,\leq\,|t-s|\,.$$
\end{lemma}
\begin{proof}
If 
 $\sqrt{t}\geq\sqrt{s}\geq a$, then
$$\displaylines{
\Big(\max\big( \sqrt{t}-a,0\big)\Big)^2
-
\Big(\max\big(\sqrt{s}-a,0\big)\Big)^2
\,=\,
\big( \sqrt{t}-a\big)^2-
\big( \sqrt{s}-a\big)^2\cr
\,=\,t-s
-2a(\sqrt{t}-\sqrt{s})\,\leq\,t-s\,.}$$
If 
 $\sqrt{t}\geq a\geq\sqrt{s}$, then
$$\displaylines{
\Big(\max\big( \sqrt{t}-a,0\big)\Big)^2
-
\Big(\max\big(\sqrt{s}-a,0\big)\Big)^2
\,=\,
\big( \sqrt{t}-a\big)^2
\,\leq\,
\big( \sqrt{t}-\sqrt{s}\big)^2\cr
\,=\,\Big(\frac{t-s}{\sqrt{t}+\sqrt{s}}\Big)^2
\,\leq\,(t-s)
\frac{t+s}{\big(\sqrt{t}+\sqrt{s}\big)^2}\,\leq\,t-s\,.
}$$
In each case, we obtain the desired inequality.
\end{proof}

\noindent
Now, thanks to Lemma~\ref{eaxe} applied with
$$a=\sqrt{{2\lambda}/{N}}\,,\quad t=\phi_X^2h(x)\,,\quad
s=\frac{1}{N}L_N(x)\,,$$
we have
$$\Big|\Big|\wh-\widetilde{\frac{1}{N}L_N}\Big|\Big|_{1}\,\leq\,
\sum_{x\in\Zd}\Big|\phi_X^2h(x)-\frac{1}{N}L_N(x)\Big|
\,=\,
\,\Big|\Big|\phi_X^2h-{\frac{1}{N}L_N}\Big|\Big|_{1}
\,\leq\,\Gamma_3
\,,
$$
where the last inequality is a consequence of inequalities~\eqref{pouf}
and~\eqref{ultg}.
We deal finally with the second term of formula~\eqref{tgu}.
We set $a=\sqrt{{2\lambda}/{N}}$
and we define
$$B\,=\,\Big\{\,x\in \Zd:
\frac{1}{N}L_N(x)\geq a^2
\,\Big\}\,.$$
We have
\begin{align*}
\Big|\Big|
{\widetilde {\frac{1}{N}L_N}}-
\frac{1}{N}L_N
&\Big|\Big|_{1}
\!\!=
\sum_{x\in B}\Big|
\Big(\frac{1}{\sqrt N}f_N(x)-a\Big)^2
\!-\frac{1}{N}L_N(x)
\Big|
+
\!\!\!
\sum_{x\in \Zd\setminus B}
\!
\frac{1}{N}L_N(x)
\cr
&\,\leq\,
\sum_{x\in B}
\Big|a^2-2a\frac{1}{\sqrt N}f_N(x)\Big|
+a^2\big|\supp f_N\big|
\cr
&\,\leq\,
3a\sum_{x\in B}
\Big|\frac{1}{\sqrt N}f_N(x)\Big|
+a^2\big|\supp f_N\big|
\cr
&\,\leq\,
3a\,
\big|\supp f_N\big|^{1/2}
\Big|\Big|
\frac{1}{N}L_N\Big|\Big|_{1}^{1/2}
+a^2\big|\supp f_N\big|\cr
&\,\leq\,
3\sqrt{\frac{2\lambda}{N}cn^d}
+
{\frac{2\lambda}{N}cn^d}
\,\leq\,
4\sqrt{\frac{2c\lambda}{n^2}}
\,,
\end{align*}
where the last inequalities come from the fact that, on the event~$\cA$,
the range of the random walk has cardinality at most $cn^d$, and
that $\lambda/n^2=n^{\rho-2}$ goes to $0$ as $n$ goes to $\infty$.
Plugging the previous inequalities
in inequality~\eqref{tgu}, we obtain that
$$\Big|\Big|\wh-\frac{1}{N}L_N\Big|\Big|_{1}\,\leq\,
\Gamma_3+4\sqrt{\frac{2c\lambda}{n^2}}\,.
$$
We conclude that,
if $h\in C\big(g,N,r(N)\big)$ and
$f_N\in\cF$, then
$\wh\in
V(\cF,\Gamma_4,N)$,
where
$$\Gamma_4\,=\,\Gamma_3+4\sqrt{\frac{2c\lambda}{n^2}}\,,$$
and
\begin{equation*}
V(\cF,\Gamma,N)\,=\,\Big\{\,
h\in \ell^1(\Zd):
\exists\,f\in\cF\quad
\Big|\Big|h-\frac{1}{N}f^2\Big|\Big|_{1}
\,\leq\,\Gamma
\,\Big\}\,.
\label{prox}
\end{equation*}
Together with
inequalities~\eqref{interm} and~\eqref{qfrt}, we obtain finally that
$$\displaylines{
E\big(e^{-|R_N|};
\cB
\big)\,\leq\,
2(\Gamma_2+1)
\exp\Big(
\frac{N}{\lambda}\Gamma_3
+\frac{2\Gamma_1}{\lambda}
+
c'\frac{n^{2d}}{\lfloor\Gamma_1\rfloor}
\hfill\cr
-\inf\,\Big\{\,
\Big|\big\{\,x\in\Zd:\wh(x)>0\,\big\}\Big|
+\frac{N-\Gamma_2}{2}
\Big(\max\Big(
\sqrt{\cE\big(\sqrt{\wh}\big)}-
\frac{4\sqrt{\Gamma_3}}{n},0\Big) \Big)^2
: \wh\in
V(\cF,\Gamma_4,N)
\,\Big\}\Big)
\,.
}$$
We have reached our goal, indeed, this last upper bound depends only
on the set~$\cF$ and it is uniform over $g\in
\cG(X,M,\eta)$.
We report this upper bound in the successive decompositions in sums
presented in formulas~\eqref{suma} and~\eqref{sumc}.
These two formulas imply that
$$
E\big(e^{-|R_N|};
f_N\in\cF,\cA)\,\leq\,
\big|{\cX(N,\delta)}\big|\times
\big|{\cG(X,M,\eta)}\big|\times
\sup\,
E\big(e^{-|R_N|};
\cB
\big)
\,,
$$
where the supremum is over
$X\in{\cX(N,\delta)}$ and
$g \in{\cG(X,M,\eta)}$.
Using the combinatorial bounds \eqref{cardx},\eqref{siq},\eqref{sib}, and
recalling the definition of~$C_0$ given in \eqref{saq}, we have, for some constant $c_d$
depending on the dimension only,
$$\big|
\cX(N,\delta)
\big|\,\leq\,
\exp\big(c_dC_0\big)\,,\qquad
|\cG|
\,\leq\,
\Big(\frac{\sqrt{N}}{\eta}\Big)^{\displaystyle
\frac{n^d}{M^d}\,c_dC_0}\,.$$
Together with
our
upper bound on the last expectation, this yields
\begin{multline}
\ln\,E\big(e^{-|R_N|};
f_N\in\cF,\cA)\,\leq\,
C_1-
\inf\,\Big\{\,
\Big|\big\{\,x\in\Zd:\wh(x)>0\,\big\}\Big|
+\qquad\qquad\cr
\qquad\qquad
\frac{N-\Gamma_2}{2}
\Big(\max\Big(
\sqrt{\cE\big(\sqrt{\widetilde h}\big)}-
	\frac{4\sqrt{\Gamma_3}}{n},0\Big) \Big)^2
: \wh\in
V(\cF,\Gamma_4,N)
\,\Big\}
\,,
\label{gou}
\end{multline}
where
$$C_1\,=\,
c_dC_0+
\frac{n^d}{M^d}\,c_dC_0\,\ln
\Big(\frac{\sqrt{N}}{\eta}\Big)+
\ln\big(2(\Gamma_2+1)\big)+
\frac{N}{\lambda}\Gamma_3
+\frac{2\Gamma_1}{\lambda}
+
c'\frac{n^{2d}}{\lfloor\Gamma_1\rfloor}
\,.$$
We compute the asymptotic expansion of these different terms in powers of $n$.
Recalling the definitions of $C_0$ and $\Gamma_3$
given in~\eqref{saq} and~\eqref{ultg}, we obtain, using~\eqref{expos},
\begin{equation}
\label{fxr}
C_0\,\lsim\,n^{\alpha 2^*}\,,\qquad
\Gamma_3\,\lsim\,
n^{\beta-1}\,+
\,n^{-4\alpha/d}\,.
\end{equation}
Together with the relations~\eqref{dxr}, this yields
\begin{equation}
\label{vxr}
C_1\,\lsim\,n^{d(1-\beta)+\alpha 2^*}
\,+\,n^{d+1+\beta-\rho }
\,+\,n^{d+2-4\alpha/d-\rho }
\,.
\end{equation}
In addition, we have
$$\Gamma_4\,\lsim\,
n^{\beta-1}\,+
\,n^{-4\alpha/d}\,+
\,n^{\rho/2-1}\,.$$
We make next a specific choice for the values of the exponents, which satisfies
the constraints stated before~\label{expo}. However we do not try to look for
the best possible exponents. So, we choose
\begin{equation}
\label{choi}
\alpha\,=\,\frac{d}{12}\,,\quad
\beta\,=\,\frac{3}{4}\,,\quad
\gamma\,=\,\frac{d^2}{12(d-2)}\,,\quad
\rho\,=\,\frac{15}{8}\,.
\end{equation}
With this choice, we obtain that
\begin{equation*}
C_1\,\lsim\,n^{d-1/8}\,,
\quad
	\Gamma_2\,\lsim\,n^{d+5/3}\,,
\quad
	\Gamma_3\,\lsim\,n^{-1/4}\,,
\quad
\Gamma_4\,\lsim\,n^{-1/16}\,.
\end{equation*}
The equivalents above are logarithmic, however, by perturbing slightly the values of
the exponents $\alpha,\beta,\gamma,\rho$, we can ensure that, for $n$ large enough,
\begin{equation*}
C_1\,\leq\,n^{d-1/8}\,,
\quad {\Gamma_2}\,\leq\,Nn^{-1/4}\,,
\quad 4\sqrt{\Gamma_3}\,\leq\,n^{-1/8}\,,
\quad
\Gamma_4\,\leq\,n^{-1/16}\,.
\end{equation*}
Plugging these inequalities in~\eqref{gou}, we obtain finally the statement of Theorem~\ref{upbo}.

\section{Continuous version of Theorem \ref{upbo}}
\label{S:continuous}

\subsection{Linear interpolation}
\label{lini}

We need a version of Theorem \ref{upbo} in which the discrete Dirichlet energy is replaced by the continuous one, so that we can use quantitative versions of the Faber--Krahn inequality. The first step is to transform the local time into a function defined continuously everywhere on $\R^d$. We will state such a version in Theorem \ref{cupbo}. We start preparing for this result.

Let $f$ be a function defined on $\Z^d$ with values in $\R^+$.
We define a function $\ft$ on $\R^d$ by interpolating linearly $f$
successively in the $d$
directions of the axis.
More precisely, we set $f_0=f$ and for all $(x_1,\dots,x_d)\in\Zd, $
\begin{multline*}
\forall \alpha\in [0,1]\quad
 f_1\big(\alpha x_1+(1-\alpha)(x_1+1),x_2,\dots,x_d\big)\,=
\alpha f_0(x_1,\dots,x_d)+
(1-\alpha) f_0(x_1+1,\dots,x_d)\,.
\end{multline*}
We define iteratively,
for $2\leq k\leq d$, for all $(x_1,\dots,x_{k-1})\in\R^{k-1}$, for all $(x_k,\dots,x_d)\in\Z^{d-k+1}$,
\begin{multline*}
\forall \alpha\in [0,1]\quad f_k\big(x_1,\dots,x_{k-1},\alpha x_k+(1-\alpha)(x_k+1),x_{k+1},
\dots,x_d\big)=\\
\alpha f_{k-1}(x_1,\dots,x_{k-1},\dots,x_d)+
(1-\alpha)
f_{k-1}(x_1,\dots,x_{k-1}+1,\dots,x_d)
\end{multline*}
and finally $\ft=f_d$.
Let $C$ be the unit cube
$$C\,=\,\big\{\,
x=(x_1,\dots,x_d)\in\Rd:0\leq x_i\leq 1,\,1\leq i\leq d\,\big\}\,.$$
Let us denote by $D^d$ the union of all
the lines parallel to the axis which go through the points of $\Zd$.
We make the following observations relating $f$ to $\widetilde f$.
\begin{lemma}
  \label{L:ftildef}
The function $\ft$ is continuous on $\Rd$ and
$C^\infty$ on $\Rd\setminus D^d$.
  We have
  \begin{equation}\label{ftildef1}
	  \int_{\R^d} \widetilde f(x)\, dx = \,\sum_{x\in\Zd} f(x)\,,
  \end{equation}
  \begin{equation}\label{ftildef3}
	  \int_{\R^d} \widetilde f(x)^2\, dx \leq \,\sum_{x\in\Zd} f(x)^2\,,
  \end{equation}
  \begin{equation}\label{ftildef2}
    \int_{\R^d} |\nabla \widetilde f(x)|^2 dx \le d \cE(f).
  \end{equation}
\end{lemma}

\begin{proof}
With the help of a standard induction,
we get the following formula for $\ft$:
$$\displaylines{
\forall (y_1,\dots,y_{d})\in \Zd\quad
\forall (x_1,\dots,x_{d})\in
(y_1,\dots,y_{d})+C\hfill\cr
\qquad \ft(x_1,\dots,x_d)\,=\,\hfill\cr
\sum_{\varepsilon_1,\dots,\varepsilon_{d} \in\{0,1\}}
\kern-7pt
f\big(y_1+\varepsilon_1,\dots,y_d+\varepsilon_d\big)
\prod_{1\leq k\leq d}\Big(
(1-\varepsilon_k)\big(1-(x_k-y_k)\big)
+\varepsilon_k(x_k-y_k)\Big)
\,.
}$$
Let us
compute the integral of $\ft$ over $\R^d$. We have
$$\int_{C}\ft(x)\,dx\,=\,
\int_{0\leq x_1\leq 1}\Bigg(\cdots
\Bigg(\int_{0\leq x_d\leq 1}
\ft(x)\,dx_d\Bigg)\cdots\Bigg) dx_1\,.$$
Let us fix
$(x_1,\dots,x_{d-1})\in [0,1]^{d-1}$ and let us compute
\begin{align*}
\int_{0\leq x_d\leq 1}
\ft (x_1,\dots,x_{d})
\,dx_d&\,=
\int_{0\leq x_d\leq 1}
\Big((1-x_d)f_{d-1} (x_1,\dots,x_{d-1},0)
+x_df_{d-1} (x_1,\dots,x_{d-1},1)
\Big)
\,dx_d\\
& \,=\,
\frac{1}{2}\Big(
f_{d-1} (x_1,\dots,x_{d-1},0)
+f_{d-1} (x_1,\dots,x_{d-1},1)
\Big)\,.
\end{align*}
Iterating this computation, we obtain
$$\int_{C}\ft(x)\,dx\,=\,
\frac{1}{2^d}\sum_{x_1,\dots,x_d\in\{0,1\}}
f (x_1,\dots,x_{d})\,.$$
Each point of $\Zd$ belongs to $2^d$ integer translates of $C$, therefore
$$\int_{\Rd}\ft(x)\,dx\,=\,\sum_{y\in\Zd}
\int_{y+C}\ft(x)\,dx\,=\,
\sum_{x\in\Zd}f(x)\,$$
which proves \eqref{ftildef1}.
Inequality~\eqref{ftildef3} is proved in the same way, by using $d$ times the convexity of the
function $t\to t^2$ in the linear interpolations.
Our next goal is to compute the integral of $(\nabla\ft)^2$ over $\R^d$.
Let us first compute its partial derivatives. For simplicity, we deal with
the derivative with respect to the last variable $x_d$, and we consider only
the points $x$ in the unit cube $C$:
$$\displaylines{
\frac{\partial\ft}{\partial x_d}(x)\,=\,
\kern-7pt
\sum_{\varepsilon_1,\dots,\varepsilon_{d-1} \in\{0,1\}}
\kern-7pt
D_df\big(\varepsilon_1,\dots,\varepsilon_{d-1}\big)
\prod_{k=1}^{d-1}\Big(
(1-\varepsilon_k)\big(1-x_k\big)
+\varepsilon_kx_k\Big)
\,,
}$$
where we define, for
${\varepsilon_1,\dots,\varepsilon_{d-1} \in\{0,1\}}$,
$$D_df\big(\varepsilon_1,\dots,\varepsilon_{d-1}\big)\,=\,
f\big(\varepsilon_1,\dots,\varepsilon_{d-1},1\big)\,-\,
f\big(\varepsilon_1,\dots,\varepsilon_{d-1},0\big)\,.$$
We apply Fubini's theorem to write
\begin{multline*}
\int_{C}
\Big(\frac{\partial\ft}{\partial x_d}(x)\Big)^2\,dx\,=\,
\kern-7pt
\sum_{\tatop{\varepsilon_1,\dots,\varepsilon_{d-1} \in\{0,1\}}
{\varepsilon'_1,\dots,\varepsilon'_{d-1} \in\{0,1\}}}
D_df\big(\varepsilon_1,\dots,\varepsilon_{d-1}\big)
D_df\big(\varepsilon'_1,\dots,\varepsilon'_{d-1}\big)\cr
\prod_{k=1}^{d-1}
\Bigg(\int_{0}^1
\Big( (1-\varepsilon_k)\big(1-x_k\big) +\varepsilon_kx_k\Big)
\Big( (1-\varepsilon'_k)\big(1-x_k\big) +\varepsilon'_kx_k\Big)\,dx_k\Bigg)
\,.
\end{multline*}
We compute the value of the integrals.
For $\varepsilon,\varepsilon'\in\{\,0,1\,\}$, we have
$$\int_{0}^1
\Big( (1-\varepsilon)\big(1-x\big) +\varepsilon x\Big)
\Big( (1-\varepsilon')\big(1-x\big) +\varepsilon'x\Big)\,dx\,=\,
\frac{1}{6}\big(1+\delta(
\varepsilon,\varepsilon')\big)\,,$$
where $\delta$ is the Kronecker symbol.
Next, we have the bound
\begin{align*}
\int_{C}
\Big(\frac{\partial\ft}{\partial x_d}(x)\Big)^2\,dx\,& \leq
\sum_{\tatop{\varepsilon_1,\dots,\varepsilon_{d-1} \in\{0,1\}}
{\varepsilon'_1,\dots,\varepsilon'_{d-1} \in\{0,1\}}}
\Big(D_df\big(\varepsilon_1,\dots,\varepsilon_{d-1}\big)\Big)^2
\prod_{k=1}^{d-1}
\frac{1}{6}\big(1+\delta(
 \varepsilon_k,\varepsilon'_k)
\big)\\
& =
\frac{1}{2^{d-1}}
\sum_{{\varepsilon_1,\dots,\varepsilon_{d-1} \in\{0,1\}}}
\Big(D_df\big(\varepsilon_1,\dots,\varepsilon_{d-1}\big)\Big)^2
\,.
\end{align*}
Next, we sum the previous inequality over all integer translates of the unit cube~$C$.
We obtain
\begin{align*}
& \int_{\Rd}
\Big(\frac{\partial\ft}{\partial x_d}(x)\Big)^2\,dx\,
 \,=\,\sum_{y\in\Zd}
\int_{y+C}
\Big(\frac{\partial\ft}{\partial x_d}(x)\Big)^2\,dx\,
 \\
 &  \leq\,
\sum_{y\in\Zd}
\frac{1}{2^{d-1}}
\sum_{{\varepsilon_1,\dots,\varepsilon_{d-1} \in\{0,1\}}}
\Big(f\big(y_1+\varepsilon_1,\dots,y_{d-1}+\varepsilon_{d-1},y_d+1\big)
-f\big(y_1+\varepsilon_1,\dots,y_{d-1}+\varepsilon_{d-1},y_d\big)
\Big)^2\\
& =\,
\sum_{y\in\Zd}
\Big(f\big(y_1,\dots,y_{d-1},y_d+1\big)
-f\big(y_1,\dots,y_{d-1},y_d\big)
\Big)^2
\,.
\end{align*}
Summing finally the integrals associated to each partial derivatives, we conclude that
$$
\int_{\Rd}\big|\nabla \ft\big|^2\,dx\,\leq\,d\,\cE(f)\,.$$
Notice that, in the definition~\eqref{defE} of $\cE$, each edge of the lattice appears twice in the summation.
This proves \eqref{ftildef2}, which finishes the proof of the lemma.
\end{proof}

\subsection{Rescaling}
We denote by $\Znd$ the lattice $\Zd$ rescaled by a factor $n$:
$\Znd\,=\,\frac{1}{n}\Z^d\,.$
We shall simultaneously rescale the space by a factor $n$ and the values of $L_N$
by a factor $n^d/N = 1/ n^2$.
Starting with the local time $L_N$, which is a function defined on $\Zd$,
we define the function $\ell_N$ on $\Rd$ by setting
\begin{equation}\label{D:rescaledLN}
\forall x\in\Rd\qquad \ell_N(x)\,=\,
\frac{n^d}{N}{L_N\big(\lfloor nx\rfloor\big)}\,.
\end{equation}
For $g$ a continuous function from $\Rd$ to $\R$, we define its support as
	$$\supp g\,=\,
\Big\{\,x\in \Rd:
g(x)\neq 0
\,\Big\}\,,$$
and we denote by $|\supp g|$ its Lebesgue measure.

Here is a continuous analogue of Theorem \ref{upbo}.
\begin{theorem}
\label{cupbo}
Let us denote by $D^d$ the union of all
the lines parallel to the axis which go through the points of $\Zd$.
Let $\cL$ be a collection of functions from~$\Rd$ to $\R^+$ such that
$\big\| \ell\big\|_2=1$ for $\ell\in\cL$.
For any $\kappa\geq 1$,
we have, for $n$ large enough,
\begin{multline*}
E\big(e^{-|R_N|};
\ell_N\in\cL)\,\leq\,
3\exp(-\kappa n^d)+
\exp\Bigg(2n^{d-1/8}-
\hfill
\cr
\qquad
n^d\,
	\inf\,\Big\{\,
	\big|\,\supp\,g\,\big|
+
\frac{1}
{2d}
(1-n^{-1/4})
\bigg(\max\Big(
\sqrt{
	\int_{\Rd}\big|\nabla g\big|^2\,dx
}-
	\frac{\sqrt{d}}{n^{1/8}} ,0\Big) \bigg)^2
\hfill\cr
\hfill
:	g\text{ is continuous } \Rd\to\R^+
	\text{ and }C^\infty\text{ on }\Rd\setminus \frac{1}{n}D^d,\,
\hfill\cr
\exists\,\ell\in\cL\quad
	\Big\|g^2- \ell\Big\|_{L^1(\Rd)}
\,\leq\,
	\frac{3}{n^{1/16}}
\,\Big\}\Bigg)\,.
\end{multline*}
\end{theorem}
\begin{proof}
Let us define formally the operator which transforms the function~$f_N$ into $\ell_N$ as in \eqref{D:rescaledLN}.
To a function $f:\Zd\to\R^+$, we associate a function $\Phi_n(f):\Rd\to\R^+$ by setting
	\begin{equation}\label{E:rescale_op}
\forall x\in\Rd\qquad \Phi_n(f)(x)\,=\,
	\frac{n^d}{N}{\Big(f\big(\lfloor nx\rfloor\big)\Big)^2}\,.
\end{equation}
	With this definition, we check that $\ell_n=\Phi_n(f_N)$.
	Therefore
$$
E\Big(e^{-|R_N|};
\ell_N\in\cL\Big)\,=\,
E\Big(e^{-|R_N|};
f_N\in\cF\Big)\,,
$$
where $\cF$ is the collection of functions defined by
	$$\cF\,=\,
	\Phi_n^{-1}(\cL)
	\,=\,
	\big\{\,f:\Zd\to\R^+, \,\Phi_n(f)\in\cL\,\big\}\,.$$
	We apply the upper bound of theorem~\ref{upbo}.
	For $h$ a function from $\Zd$ to $\R^+$, we denote its support
	by
	$$\supp h\,=\,
\Big\{\,x\in \Zd:
h(x)\neq 0
\,\Big\}\,.$$
Let $\kappa\geq 1$.
For $n$ large enough, we have
$$\displaylines{
E\big(e^{-|R_N|};
\ell_N\in\cL)\,\leq\,
\exp(-\kappa n^d)+
\hfill
\cr
\exp\Bigg(n^{d-1/8}-
\inf\,\Big\{\,
\big|\,\supp h
\big|+
\frac{N}
{2}
(1-n^{-1/4})
\bigg(\max\Big(
\sqrt{\cE\big(\sqrt{h}\big)}-
\frac{1}{n^{9/8}} ,0\Big) \bigg)^2
\cr
\hfill
:h\in \ell^1(\Zd),\,
\exists\,f\in\cF\quad
\Big|\Big|h-\frac{1}{N}f^2\Big|\Big|_{1}
\,\leq\,\frac{1}{n^{1/16}}
\,\Big\}\Bigg)\,.
}$$
In the infimum appearing above, we can limit ourselves to functions $h$
for which the infimum is of order $n^d$, otherwise the exponential becomes
negligible compared to the first term
$\exp(-\kappa n^d)$. The relevant functions $h$ should be such that
\begin{equation}
\label{aly}
\big|\,\supp\, h\,\big|\,\leq\,2\kappa n^d\,,
\quad
\cE\big(\sqrt{h}\big)\,\leq\,
\frac{3\kappa}{n^2}\,.
\end{equation}
We make also the change of function $h\to h^2$.
We obtain that, for $n$ large enough,
\begin{multline}
\label{naly}
E\big(e^{-|R_N|};
\ell_N\in\cL)\,\leq\,
3\exp(-\kappa n^d)+
\hfill
\cr
\exp\Bigg(n^{d-1/8}-
\inf\,\Big\{\,
\big|\,\supp h
\big|+
\frac{N}
{2}
(1-n^{-1/4})
\bigg(\max\Big(
\sqrt{\cE\big({h}\big)}-
\frac{1}{n^{9/8}} ,0\Big) \bigg)^2
\cr
\hfill
:h\in \ell^2(\Zd),\,h\geq 0\,,
\big|\,\supp h\,\big|\,\leq\,2\kappa n^d\,,
\cE\big({h}\big)\,\leq\,
\frac{3\kappa}{n^2},\,
\exists\,f\in\cF\quad
\Big|\Big|h^2-\frac{1}{N}f^2\Big|\Big|_{1}
\,\leq\,\frac{1}{n^{1/16}}
\,\Big\}\Bigg)\,.
\end{multline}
We apply now the linear interpolation procedure described in section~\ref{lini} to the function $h$
appearing in the above infimum.
Since $||f^2||_1=N$ for any $f\in\cF$, the constraint on the function $h$
implies that
\begin{equation}
\label{vbg}
\Big|\|h^2\|_1-{1}\Big|
\,\leq\,\frac{1}{n^{1/16}}\,.
\end{equation}
Starting from a function $h$ in $\ell^2(\Z^d)$,
we apply the interpolation procedure and
we obtain a function $\smash{\hht}$
in $L^2(\R^d)$
which satisfies, according to
inequalities~\eqref{ftildef3}
and~\eqref{ftildef2},
\begin{equation}
\label{hjy}
||\hht||_{L^2(\Rd)}\,\leq\,||h||_{\ell^2(\Zd)}\,,\quad
\int_{\Rd}\big|\nabla \hht\big|^2\,dx\,\leq\,d\,\cE(h)\,.
\end{equation}
We wish to obtain an infimum involving only continuous functions, so we have to get completely
rid of the discrete function~$h$. Therefore we should control
$\big|\,\supp h\,\big|$ and the distance between $\hht$ and the set $\cL$.
Let us start with
$\big|\,\supp h\,\big|$. The linear interpolation $\hht$ of $h$ is non--zero only in
the unit cubes having at least one vertex in the support of~$h$, therefore
\begin{multline}
\label{peni}
\big|\,\{\,x\in \Rd:\hht(x)>0\,\}\,\big|\,\leq\,
\big|\,\supp h\,\big|+\,
\big|\,\{\,x\in \Zd\setminus\supp h:\exists y\in
\supp h
\quad |x-y|_\infty\leq 1
\,\}\,\big|
\,.
\end{multline}
Notice that we use $|\cdot|$ to denote the Lebesgue measure for a continuous set
and the cardinality for a discrete set.
Each point $x$ of $\Zd$ admits $3^d$ points $y\in\Zd$ such that $|x-y|_\infty=1$,
therefore
\begin{equation}
\label{meni}
\big|\,\{\,x\in \Rd:\hht(x)>0\,\}\,\big|\,\leq\,
	(3^d+1)\big|\,\supp h\,\big|\,.
\end{equation}
This inequality will be useful, however
we need a better lower bound for
$\big|\,\supp h\,\big|$, therefore we need a better upper bound on the last term in inequality~\eqref{peni}.
This is a delicate matter. Our strategy is to use the bound on the discrete Dirichlet energy to control
this boundary term, and to do so, we truncate the function at a fixed level $\lambda>0$. Since $\hht(x)$ is
a linear interpolation between the $2^d$ values of $h$ at the vertices of the unit cube containing $x$,
we have
\begin{multline}
\label{geni}
\big|\,\{\,x\in \Rd:\hht(x)>\lambda\,\}\,\big|\,\leq\,
\big|\,\supp h\,\big|\cr
	+\,
\Big|\,\Big\{\,x\in \Zd\setminus\supp h:\exists y\in \supp h
	\quad |x-y|_\infty\leq 1\,,\quad h(y)\geq{\lambda}
\,\Big\}\,\Big|
\,.
\end{multline}
Let $C$ be a unit cube with vertices in $\Zd$ such that one vertex $x$ of $C$ is not in $\supp h$
and another vertex $y$ of $C$ satisfies
	$h(y)\geq{\lambda}$.
	Then there exist two vertices $x',y'$ of $C$ which are nearest neighbours and which
	satisfy
	$$h(x')-h(y')\geq\frac{\lambda}{2^{d}}\,,$$
and the contribution to the discrete Dirichlet energy $\cE(h)$ of the edges belonging to the
boundary of $C$ is larger or equal than
	$$\frac{1}{d}\big(h(x')-h(y')\big)^2\geq\frac{1}{d}\frac{\lambda^2}{2^{2d}}\,.$$
An edge belongs to at most $2^{d-1}$ unit cubes. The number $M$ of cubes making such a contribution
satisfies therefore
	$$M\frac{1}{d}\frac{\lambda^2}{2^{2d}}\frac{1}{2^{d-1}}\,\leq\,\cE(h)\,.$$
Moreover the last term in inequality~\eqref{geni} is bounded from above by $M2^{2d}$.
Taking into account the bound on $\cE(h)$ given in~\eqref{aly}, we conclude that
\begin{multline}
\label{weni}
\big|\,\supp h\,\big|\,\geq\,
\big|\,\{\,x\in \Rd:\hht(x)>\lambda\,\}\,\big|\,-\,
	M2^{2d} 
	\,\geq\,
\big|\,\{\,x\in \Rd:\hht(x)>\lambda\,\}\,\big|\,-\,
	\frac{d2^{5d}}{\lambda^2}
\frac{3\kappa}{n^2}\,.
\end{multline}
In order to delay the rescaling of the space, we define an intermediate operator $\Phi$ acting
on the functions as follows.
To a function $f:\Zd\to\R^+$, we associate a function $\Phi(f):\Rd\to\R^+$ by setting
	$$\forall x\in\Rd\qquad \Phi(f)(x)\,=\,
	{\Big(f\big(\lfloor x\rfloor\big)\Big)^2}\,.$$
We rewrite
\eqref{E:rescale_op} with the help of $\Phi$ as follows:
\begin{equation}
\label{ggi}
\forall x\in\Rd\qquad \Phi_n(f)(x)\,=\,
	\frac{n^d}{N}
	\Phi(f)(nx)\,.
\end{equation}
Let now $f$ be an element of
$\cF$ such that
$$\Big|\Big|h^2-\frac{1}{N}f^2\Big|\Big|_{1}
\,\leq\,\frac{1}{n^{1/16}}\,.$$
By definition of $\cF$, we have that $\Phi_n(f)\in\cL$.
On the one hand,
we have
\begin{equation}
\label{zeni}
	\Big\|\Phi(h)-
\frac{1}{N} \Phi(f)
	\Big\|_{L^1(\Rd)}\,=\,
	\Big|\Big|h^2-\frac{1}{N}f^2\Big|\Big|_{\ell^1(\Zd)}
\,\leq\,\frac{1}{n^{1/16}}\,.
\end{equation}
On the other hand,
we have
\begin{equation}
\label{beni}
	\Big\|\hht^2-
 \Phi(h)
	\Big\|_{L^1(\Rd)}\,=\,
	\int_{\Rd}\big|
	\hht(x)^2-
	h\big(\lfloor x\rfloor\big)^2
	\big|\,dx
	\,.
\end{equation}
Let us define
$$\forall x\in\Zd\qquad D(x)\,=\,\big\{\,y\in \Rd:
	\lfloor y\rfloor=x\,\big\}\,.$$
We can rewrite the previous integral as
\begin{align}
\label{bvni}
	\Big\|\hht^2- \Phi(h)
	\Big\|_{L^1(\Rd)}
	&\,=\,
	\sum_{x\in\Zd}
	\int_{D(x)}\big|
	\hht(y)^2-
	h\big(\lfloor y\rfloor\big)^2
	\big|\,dy
	\cr
	&\,\leq\,
	\sum_{x\in\Zd}
	\sup_{y\in D(x)}\big|
	\hht(y)^2-
	h(x)^2
	\big|
	\cr
	\,\leq\,
	\sum_{x\in\Zd}
	&2^d\max\big\{\,
	\big|
	h(y)^2-
	h(z)^2
	\big|:
	y,z\in \overline{D(x)}\cap\Zd, |y-z|=1\,\big\}
	\cr
	&\,\leq\,
	\sum_{x\in\Zd}
	2^{d}
	\sum_{\tatop{y,z\in \overline{D(x)}\cap\Zd }{|y-z|=1}}
	\big|
	h(y)^2-
	h(z)^2
	\big|
	\cr
	&\,\leq\,
	2^{2d}
	\sum_{\tatop{y,z\in \Zd}{|y-z|=1}}
	\big| h(y)- h(z) \big|
	\times
	\big| h(y)+ h(z) \big|
	\cr
	&\,\leq\,
	2^{2d}
	\Big(\sum_{\tatop{y,z\in \Zd}{|y-z|=1}}
	\big| h(y)- h(z) \big|^2
	\Big)^{1/2}
	\Big(\sum_{\tatop{y,z\in \Zd}{|y-z|=1}}
	\big| h(y)+ h(z) \big|^2\Big)^{1/2}
	\cr
	&\,\leq\,
	2^{2d}
	\big({2d\cE(h)}
	\big)^{1/2}
	8d\|h\|_{\ell^2(\Zd)}
	\,.
\end{align}
Using inequalities~\eqref{aly} and~\eqref{vbg} (remember that we changed $h$ into $h^2$ just
afterwards), we obtain that
\begin{equation}
\label{cni}
	\Big\|\hht^2- \Phi(h)
	\Big\|_{L^1(\Rd)}
	\,\leq\,
	32d^22^{2d}
	\frac{\sqrt{3\kappa}}{n}
	\,.
\end{equation}
Combining this inequality and inequality~\eqref{zeni}, we conclude that, for $n$ large enough,
\begin{equation}
\label{cri}
	\Big\|\hht^2- \frac{1}{N} \Phi(f) \Big\|_{L^1(\Rd)}
\,\leq\, 32d^22^{2d} \frac{\sqrt{3\kappa}}{n} +
	\frac{1}{n^{1/16}}
\,\leq\,
	\frac{2}{n^{1/16}}\,.
\end{equation}
In addition, the function $\hht$ has bounded support and it is continuous on $\Rd$.
Let us denote by $D^d$ the union of all
the lines parallel to the axis which go through the points of $\Zd$.
The function $\hht$ is also
$C^\infty$ on $\Rd\setminus D^d$.
Thus we can take the infimum over the set of functions $\hht$ having these properties.
We are now ready to substitute $\hht$ to $h$ in inequality~\eqref{naly}.
We bound from below the infimum in the exponential with the help of
inequalities~\eqref{hjy},
\eqref{meni},
\eqref{weni}
and~\eqref{cri}:
\begin{multline}
\label{snaly}
E\big(e^{-|R_N|};
\ell_N\in\cL)\,\leq\,
3\exp(-\kappa n^d)+
\exp\Bigg(n^{d-1/8}-
\hfill
\cr
\qquad\inf\,\Big\{\,
\big|\,\{\,x\in \Rd:\hht(x)>\lambda\,\}\,\big|\,-\,
	\frac{d2^{5d}}{\lambda^2}
\frac{3\kappa}{n^2}
+\hfill
\cr\hfill
\frac{N}
{2}
(1-n^{-1/4})
\bigg(\max\Big(
\sqrt{
\frac{1}{d}\int_{\Rd}\big|\nabla \hht\big|^2\,dx
}-
\frac{1}{n^{9/8}} ,0\Big) \bigg)^2
\hfill\cr
\hfill
:
	\hht\text{ is continuous } \Rd\to\R^+
	\text{ and }C^\infty\text{ on }\Rd\setminus D^d,\,
	\quad
\big|\,\supp\hht\,\big|\,\leq\, (3^d+1)2\kappa n^d\,,
\cr\hfill
\exists\,f\in\cF\quad
	\Big\|\hht^2- \frac{1}{N} \Phi(f) \Big\|_{L^1(\Rd)}
\,\leq\, 
	\frac{2}{n^{1/16}}
\,\Big\}\Bigg)\,.
\end{multline}
We are almost done. We make a change of scale in order to replace $\cF$ by $\cL$
in the infimum. To the function
$\hht: \Rd\to\R^+$, we associate the function
$\hht_n: \Rd\to\R^+$ obtained by rescaling the space by a factor~$n$ and the values by $n^{d/2}$:
	$$\forall x\in\Rd\qquad \hht_n(x)\,=\,
	{n^{d/2}}
	{\widetilde{h}(nx)}
	\,.$$
We have then
\begin{align}
\label{zntly}
	\phantom{\int_{\Rd}}
	\big|\,\supp\hht_n\,\big|&\,=\,
n^d\big|\,\supp\hht\,\big|\,,\cr
\big|\,\{\,x\in \Rd:\hht(x)>\lambda\,\}\,\big|
	&\,=\,
	{n^d}
	\big|\,\{\,x\in \Rd:\hht_n(x)>n^{d/2}{\lambda}\,\}\,\big|
	\,,\cr
\int_{\Rd}\big|\nabla \hht_n(x)\big|^2\,dx
	&\,=\,
	n^d
\int_{\Rd}\big|
n
\nabla
	{\widetilde{h}}(nx)\big|^2\,dx
	\,=\,
n^{2}
\int_{\Rd}\big|
\nabla
	{\widetilde{h}}(x)\big|^2\,dx
	\,.
\end{align}
Moreover, using the identity~\eqref{ggi}, we see that
\begin{equation}
\label{znaly}
	\Big\|\hht_n^2- \Phi_n(f) \Big\|_{L^1(\Rd)}\,=\,
\int_{\Rd}
	n^d\Big|\hht(nx)^2-
	\frac{1}{N} \Phi(f)(nx)
	\Big|\,dx\,=\,
	\Big\|\hht^2- \frac{1}{N} \Phi(f) \Big\|_{L^1(\Rd)}\,.
\end{equation}
	Remember that $\cF=\Phi_n^{-1}(\cL)$, thus
	$\Phi_n(f)\in\cL$ whenever $f\in\cF$.
	The identities~\eqref{zntly} and~\eqref{znaly}
	allow to rewrite~\eqref{snaly} as follows:
\begin{multline}
\label{tnaly}
E\big(e^{-|R_N|};
\ell_N\in\cL)\,\leq\,
3\exp(-\kappa n^d)+
\exp\Bigg(n^{d-1/8}-
\inf\,\Big\{\,
n^d
	\big|\,\{\,x\in \Rd:\hht_n(x)>n^{d/2}{\lambda}\,\}\,\big|
	\cr
	\,-\,
	\frac{d2^{5d}}{\lambda^2}
\frac{3\kappa}{n^2}
+
\frac{N}
{2}
(1-n^{-1/4})
\bigg(\max\Big(
\sqrt{
\frac{1}{dn^2}\int_{\Rd}\big|\nabla \hht_n\big|^2\,dx
}-
\frac{1}{n^{9/8}} ,0\Big) \bigg)^2
\cr
\hfill
	:\hht_n\text{ is continuous } \Rd\to\R^+
	\text{ and }C^\infty\text{ on }\Rd\setminus \frac{1}{n}D^d,\,
	\quad
\big|\,\supp\hht_n\,\big|\,\leq\, (3^d+1)2\kappa,\,
\cr\hfill
\exists\,\ell\in\cL\quad
	\Big\|\hht_n^2- \ell\Big\|_{L^1(\Rd)}
\,\leq\,
	\frac{2}{n^{1/16}}
\,\Big\}\Bigg)\,.
\end{multline}
We choose $\lambda=n^{-\frac{d+1}{2}}$ and  we set
$$\forall x\in \Rd\qquad
g(x)\,=\,\max\Big(
\hht_n(x)- \frac{1}{\sqrt{n}},0 \Big)\,.$$
This function $g$ satisfies
\begin{align}
\label{unaly}
	\big|\,\{\,x\in \Rd:\hht_n(x)>n^{d/2}{\lambda}\,\}\,\big|
	&\,=\,
	\big|\,\{\,x\in \Rd:g(x)>0\,\}\,\big|\,,\cr
	1+\frac{2}{n^{1/16}}
\,\geq\,\int_{\Rd}\big|\hht_n\big|^2\,dx
	&\,\geq\,
\int_{\Rd}\big| g\big|^2\,dx\,,\cr
\int_{\Rd}\big|\nabla \hht_n\big|^2\,dx
	&\,\geq\,
\int_{\Rd}\big|\nabla g\big|^2\,dx\,.
\hfil
\end{align}
Moreover we have
\begin{align}
	\label{glou}
	\big\|g^2-\hht_n^2\big\|_{L^1(\Rd)}&\,=\,
	\int_{\Rd}\big|
	g(x)^2-
	\hht_n(x)^2
	\big|\,dx\cr
	&\,\leq\,
\frac{1}{\sqrt{n}}
	\int_{\Rd}\big|
	g(x)+
	\hht_n(x)
	\big|\,dx\,
	\cr
	&\,\leq\,
\frac{1}{\sqrt{n}}
	\Big(
	\big\|g\big\|_{L^2(\Rd)}
	\big(\big|\,\supp\,g\,\big|\big)^{1/2}+
\big\|\hht_n\big\|_{L^2(\Rd)}
	\big(\big|\,\supp\,\hht_n\,\big|\big)^{1/2}\Big)\cr
	&\,\leq\,
\frac{2}{\sqrt{n}}
\big\|\hht_n\big\|_{L^2(\Rd)}
	\big(\big|\,\supp\,\hht_n\,\big|\big)^{1/2}\cr
	&\,\leq\,
\frac{4}{\sqrt{n}}
	\big((3^d+1)2\kappa\big)^{1/2}\,.
\end{align}
Plugging the inequalities~\eqref{unaly} and~\eqref{glou}
in the infimum of~\eqref{tnaly}, we get, for $n$ large enough,
\begin{multline*}
E\big(e^{-|R_N|};
\ell_N\in\cL)\,\leq\,
3\exp(-\kappa n^d)+
\exp\Bigg(n^{d-1/8}-
\hfill
\cr
\qquad\inf\,\Big\{\,
n^d
	\big|\,\supp\,g\,\big|
	\,-\,
	{d2^{5d}}
	{3\kappa}{n^{d-1}}
+
\frac{N}
{2}
(1-n^{-1/4})
\bigg(\max\Big(
\sqrt{
\frac{1}{dn^2}\int_{\Rd}\big|\nabla g\big|^2\,dx
}-
\frac{1}{n^{9/8}} ,0\Big) \bigg)^2
\hfill\cr
\hfill
:	g\text{ is continuous } \Rd\to\R^+
	\text{ and }C^\infty\text{ on }\Rd\setminus \frac{1}{n}D^d,\,
\hfill\cr
\quad
\exists\,\ell\in\cL\quad
	\Big\|g^2- \ell\Big\|_{L^1(\Rd)}
\,\leq\,
	\frac{3}{n^{1/16}}
\,\Big\}\Bigg)\,.
\end{multline*}
For $n$ large enough, this inequality can be rewritten as in the statement of
theorem~\eqref{cupbo}.
\end{proof}

\section{Application of the quantitative Faber--Krahn inequality}
\label{S:FK}

\subsection{Statement}
A crucial ingredient to extend Bolthausen's result to dimensions $d\geq 3$ is the
quantitative Faber-Krahn inequality. Currently, the best version of this
inequality is due to
Brasco, De Philippis and Velichkov \cite{BLGV}.
A weaker version
was proved before by Fusco, Maggi and Pratelli \cite{FMP}, and we could very well
rely on this weaker version to achieve our goal.
Let $G$ be an open subset of $\Rd$ having finite Lebesgue measure.
Let $\lambda(G)$ be the first eigenvalue of the Dirichlet--Laplacian of $G$, defined by
$$\lambda(G)\,=\,\inf\Big\{\,\int_G\big|\nabla u\big|^2\,dx:||u||_{2,G}=1\,\Big\}\,.$$
Recall that our notation differs from Bolthausen's by a factor $1/2$: the notation $\lambda(G)$ in \cite{BO} is half of the quantity defined above.

\medskip

\noindent
To control the distance of $G$ to a ball, we define the Fraenkel asymmetry
$$\cA(G)\,=\,\inf\Big\{\,\frac{|G\Delta B|}{|B|}:
B\text{ ball such that }|B|=|G|\,\Big\}\,,$$
where $G\Delta B$ is the symmetric difference between the sets $G$ and $B$.
\begin{theorem}
\label{fmp}
There exists a positive constant $\sigma$, which depends on the dimension $d$ only,
such that, for any
an open subset $G$ of $\Rd$ having finite Lebesgue measure, we have, for any $d$ dimensional
ball $B$,
$$
|G|^{2/d}\lambda(G)\,-\,
|B|^{2/d}\lambda(B)\,\geq\,\sigma\cA(G)^2\,.$$
\end{theorem}
This result is
proved by
Brasco, De Philippis and Velichkov \cite{BLGV}.
With the help of this
quantitative Faber-Krahn inequality, we shall prove the crucial Lemma~\ref{A1},
which is the $d$ dimensional counterpart of lemma~$A.1$ in Bolthausen's paper \cite{BO}.
Let $\lambda_d$ be the principal eigenvalue of $-\Delta$ in the unit ball $B(0,1)$ of $\Rd$,
let $\omega_d$ be the volume of $B(0,1)$, and let us define
$$\rho_d\,=\,\Big(\frac{\lambda_d}{d^2\omega_d}\Big)^{\textstyle\frac{1}{d+2}}\,.$$
Let $G$ be an open subset of $\Rd$ having finite Lebesgue measure and
let $r>0$ be such that $|G|=r^d\omega_d$.
The classical Faber--Krahn inequality states that
\begin{equation}
\label{cfkr}
\lambda(G)
\,\geq\,
\frac{\lambda_d}{r^{2}}
\,,
\end{equation}
while the inequality of theorem~\ref{fmp} can be rewritten as
\begin{equation}
\label{fkr}
\lambda(G)
\,\geq\,
\frac{1}{r^{2}}\Big(\lambda_d+
\frac{ \sigma\cA(G)^2 }{(\omega_d)^{2/d}}
\Big)\,.
\end{equation}
By the classical Faber--Krahn inequality,
we have
$$|G|+\frac{1}{2d}\lambda(G)\,\geq\,
r^d\omega_d+ \frac{\lambda_d}{2dr^{2}}\,.$$
Let us define
\begin{equation}
\label{psir}
\forall r>0\qquad\psi(r)\,=\,
r^d\omega_d+ \frac{\lambda_d}{2dr^{2}}\,.
\end{equation}
The function $\psi$ admits a unique minimum on $\R^+$ at $r=\rho_d$.
The constant $\chi_d$ is defined through the variational formula
$$\chi_d\,=\,\inf\,\Big\{\,|G|+\frac{1}{2d}\lambda(G):
G\text{ open subset of $\Rd$}\,\Big\}\,.$$
We conclude from the previous inequalities that
$$\chi_d\,=\,\frac{d+2}{2}\Big(\frac{\lambda_d}{d^2}\Big)^{\textstyle\frac{d}{d+2}}
(\omega_d)^{\textstyle\frac{2}{d+2}}\,.$$
Let $\phi$ be the eigenfunction of $-\Delta$ in $B(0,\rho_d)$ with Dirichlet boundary
conditions associated to $\lambda_d$ and normalized so that $||\phi||_2=1$, $\phi\geq 0$.
We extend $\phi$ to $\Rd$ by setting it equal to $0$ outside $B(0,\rho_d)$.
For $x\in\Rd$, we denote by $\phi_x$ the translate of $\phi$ defined by
$$\forall y\in\Rd\qquad\phi_x(y)\,=\,\phi(y-x)\,.$$
We state next the counterpart
of Lemma~$A.1$
of Bolthausen's paper \cite{BO}.
The difference is that we work in dimensions $d\geq 3$ and in the full
space rather than in the torus. The spirit of the proof is exactly the same
as in Bolthausen's case, the major new input is the quantitative Faber--Krahn inequality.
Equipped with this powerful inequality,
the proof becomes more transparent than Bolthausen's proof,
which contains somehow a two--dimensional version of a quantitative isoperimetric inequality.
\begin{lemma}
\label{A1}
If $g:\Rd\to\R^+$ is $C^\infty$ and such that
$||g||_2=1$, $g\geq 0$, and if
$$\varepsilon\,=\,\inf_{x\in\Rd}\,||g-\phi_x||_2\,>\,0$$
is small enough, then
$$\big|\{\,g>0\,\}\big|+\frac{1}{2d}
\int_{\Rd}\big|\nabla g\big|^2\,dx\,\geq\,\chi_d+\varepsilon^{48}\,.
$$
\end{lemma}
\begin{proof}
There are many constants involved throughout this proof. So we denote by $c$ a generic
positive constant which depends only on the dimension~$d$, and we warn that of course the value
of $c$ changes from one formula to another!
Let $g$ be a function as in the statement of the lemma and let us set
$$G\,=\,\big\{\,x\in \Rd:g(x)>0\,\big\}\,.$$
We need only to consider the case where $|G|<\infty$ and the function $g$ is such that
$$\big|G\big|+\frac{1}{2d}
\int_{\Rd}\big|\nabla g\big|^2\,dx\,\leq\,2\chi_d\,.$$
The Sobolev inequality implies then that there exists a constant $c_S(d)$ depending on the
dimension only such that
\begin{equation}
\label{aob}
||g||_4^2\,\leq\,c_S(d)
\int_{\Rd}\big(g^2+\big|\nabla g\big|^2\big)\,dx\,\leq\,c_S(d)(1+4d\chi_d)\,.
\end{equation}
Let $r>0$ be such that $|G|=r^d\omega_d$.
Applying inequality~\eqref{fkr} to the set $G$, we obtain
\begin{multline}
\big|\{\,g>0\,\}\big|+\frac{1}{2d}
\int_{\Rd}\big|\nabla g\big|^2\,dx\,\geq\,
r^d\omega_d+
\frac{1}{2d}\lambda(G)\cr
\,\geq\,
r^d\omega_d+
\frac{1}{2dr^{2}}\Big(\lambda_d+
\frac{ \sigma\cA(G)^2 }{(\omega_d)^{2/d}}
\Big)
\,=\,
\psi(r)+
\frac{ \sigma\cA(G)^2 }{ {2dr^{2} (\omega_d)^{2/d}}
}
\,,
\label{sati}
\end{multline}
where $\psi(r)$ is the function defined in~\eqref{psir}.
We recall that the function $\psi$ admits a unique minimum at $r=\rho_d$
and moreover $\psi''(\rho_d)>0$.
Therefore there exist two positive constants
$\eta,c$ such that $\eta<\rho_d/2$ and
\begin{equation}
\label{mini}
\forall r\in ]\rho_d-\eta,\rho_d+\eta[\qquad
\psi(r)\geq \psi(\rho_d)+c(r-\rho_d)^2\,.
\end{equation}
For
$\varepsilon_0>0$ small enough, we can take $\eta$ sufficiently small to ensure
that we have in addition
$$
\forall r\in \R^+\setminus ]\rho_d-\eta/2,\rho_d+\eta/2[\qquad
\psi(r)\geq
 \psi(\rho_d)+
(\varepsilon_0)^{48}\,.$$
In particular, the inequality stated in the lemma holds if $\varepsilon<\varepsilon_0$
and if $r$ does not belong
to the interval
$]\rho_d-\eta/2,\rho_d+\eta/2[$.
From now onwards, we suppose that $\varepsilon<\varepsilon_0$ and
we consider only the cases of sets $G$
such that
$\rho_d-\eta/2< r<\rho_d+\eta/2$.
From inequality~\eqref{mini}, we see that we need only to consider
the case where
\begin{equation}
\label{fco}
c(r-\rho_d)^2\,\,\leq\,\varepsilon^{48}\,.
\end{equation}
Recalling that $\psi(r)\geq\chi_d$, it follows from~\eqref{sati} that if
$$\frac{\sigma\cA(G)^2}{8d(\rho_d)^{2}(\omega_d)^{2/d}}
\,\geq\,\varepsilon^{48}\,,$$
then the inequality of the lemma is satisfied.
From now onwards, we consider only the cases of sets $G$
satisfying
\begin{equation}
\label{sata}
{\cA(G)}
\,<\,
\sqrt{\frac{1}{\sigma}{8d(\rho_d)^{2}(\omega_d)^{2/d}}\varepsilon^{48}}\,=\,
c\varepsilon^{24}\,.
\end{equation}
Therefore there exists $x_0\in\Rd$ and $r_0>0$
such that
\begin{equation}
\label{diff}
\big|G\Delta B(x_0,r_0)\big|\,\leq\,
c
(2\rho_d)^d\omega_d
\varepsilon^{24}
\,\leq\,
c\varepsilon^{24}
\,.
\end{equation}
This inequality, together with inequality~\eqref{fco}, imply that
\begin{equation}
\label{adiff}
\big|r_0-\rho_d\big|\,\leq\,
c\varepsilon^{24}
\,.
\end{equation}
Let $\delta>0$. We shall compare
$\int_{\Rd}\big|\nabla g\big|^2\,dx$
with $\lambda\big({B(x_0,r_0+\delta)}\big)$.
To this end, we truncate smoothly the function $g$ as follows.
Let $h$ be a $C^\infty$ function from $\R^d$ to $[0,1]$ satisfying
$$\forall x\in\Rd\qquad
h(x)\,=\,
\begin{cases}
1 & \text{ if }x\in B(x_0,r_0)\cr
0 & \text{ if }x\not\in B(x_0,r_0+\delta)\cr
\end{cases}\,,
$$
as well as the following bound on its gradient:
$$\forall x\in\Rd\qquad
\big|\nabla h(x)\big|\,\leq\frac{c}{\delta}\,.$$
We set $\gt=hg$ and we estimate the Dirichlet energy of $\gt$ as follows.
Let
$$A\,=\,B(x_0,r_0+\delta)\setminus B(x_0,r_0)\,,$$
we have
\begin{multline}
\label{kold}
\int_{\Rd}\big|\nabla \gt\big|^2\,dx\,=\,
\int_{B(x_0,r_0)}\big|\nabla g\big|^2\,dx\,+\,
\int_{A}\big|\nabla (hg)\big|^2\,dx\,\cr
\,\leq\,
\int_{\Rd}\big|\nabla g\big|^2\,dx\,+\,
\int_{A}\big|\nabla h\big|^2g^2\,dx\,+\,
\int_{A}2hg \,
\big|\nabla h\big|\,
\big|\nabla g\big|\,dx\,.
\end{multline}
Moreover
$$
\int_{A}\big|\nabla h\big|^2g^2\,dx
\,\leq\,
\Big(\frac{c}{\delta}\Big)^2
\int_{A\cap G}g^2\,dx\,.$$
Yet, using~\eqref{diff}, we have
\begin{equation}
\label{vold}
\big|A\cap G\big|
\,\leq\,
\big|G\setminus B(x_0,r_0)\big|
\,\leq\,
c\varepsilon^{24}
\,,
\end{equation}
and, by H\"older's inequality,
\begin{multline}
\label{hold}
\int_{A\cap G}g^2\,dx\,\leq\,
\int_{G\setminus B(x_0,r_0)}g^2\,dx
\,\leq\,
\Big(
\int_{\Rd}g^4\,dx
\Big)^{\frac{1}{2}}
\big|G\setminus B(x_0,r_0)\big|
^{\frac{1}{2}}
\,\leq\,
||g||_4^2
\sqrt{c\varepsilon^{24}}\,.
\end{multline}
We control the last integral of~\eqref{kold} as follows:
\begin{equation}
\label{last}
\int_{A}2 |hg| \,
\big|\nabla h\big|\,
\big|\nabla g\big|\,dx
\,\leq\,
2\frac{c}{\delta}
\int_{A\cap G} \,
\big|g\nabla g\big|\,dx
\,\leq\,
2\frac{c}{\delta}
\Big(\int_{\Rd}\big|\nabla g\big|^2\,dx\Big)^{\frac{1}{2}}
\Big(\int_{A\cap G}g^2\,dx\Big)^{\frac{1}{2}}
\,.
\end{equation}
We plug these
inequalities in~\eqref{kold}. Together with inequality~\eqref{aob}, we obtain
\begin{equation}
\int_{\Rd}\big|\nabla \gt\big|^2\,dx\,\leq\,
\int_{\Rd}\big|\nabla g\big|^2\,dx\,+\,
c
\frac{\varepsilon^{12}}{\delta^2}+
c
\frac{\varepsilon^{6}}{\delta}\,.
\end{equation}
The end of the argument is the same as in lemma~$A.1$ of \cite{BO}.
We denote by $\hp_0$ the normalized eigenfunction in $B(x_0,r_0+\delta)$, that is,
$$\forall x\in
B(x_0,r_0+\delta)\qquad
\hp_0(x)\,=\,
\Big(\frac{\rho_d}{r_0+\delta}\Big)^{d/2}
\phi\Big(
\frac{\rho_d}{r_0+\delta}(x-x_0)\Big)\,.$$
With a change of variables, we obtain that
$$\int_{\Rd}(\hp_0)^2\,dx\,=\,1$$
and
$$\int_{\Rd}\big|\nabla \hp_0\big|^2\,dx\,=\,
\Big(\frac{\rho_d}{r_0+\delta}\Big)^{2}
\int_{\Rd}\big|\nabla \phi\big|^2\,dx\,.$$
We have then, using~\eqref{vold},
\begin{multline}
\label{glob}
\big|\{\,g>0\,\}\big|
+
\frac{1}{2d}
\int_{\Rd}\big|\nabla g\big|^2\,dx\,\geq\,
\big|B(x_0,r_0)\big|-
c\varepsilon^{24}
+
\frac{1}{2d}
\int_{\Rd}\big|\nabla \gt\big|^2\,dx
-c
\frac{\varepsilon^{12}}{\delta^2}-
c
\frac{\varepsilon^{6}}{\delta}
\cr
\,\geq\,
\big|B(x_0,r_0+\delta)\big|
+\frac{1}{2d}
\int_{\Rd}\big|\nabla \hp_0\big|^2\,dx
-c\delta
-c\varepsilon^{24}
+
\frac{1}{2d}
\int_{\Rd}\Big(\big|\nabla \gt\big|^2
-\big|\nabla \hp_0\big|^2\Big)\,dx
-c
\frac{\varepsilon^{12}}{\delta^2}-
c
\frac{\varepsilon^{6}}{\delta}
\,.
\end{multline}
By inequality~\eqref{adiff}, we have
\begin{equation}
\label{ulto}
\big|r_0+\delta-\rho_d\big|\,\leq\,
c\varepsilon^{24}+\delta\,.
\end{equation}
Thus, for $\varepsilon$ and $\delta$ small enough, the value
$r_0+\delta$ belongs to the interval
$]\rho_d-\eta,\rho_d+\eta[$ and we can apply \eqref{mini} to get
\begin{equation}
\label{intert}
\big|B(x_0,r_0+\delta)\big|
+\frac{1}{2d}
\int_{\Rd}\big|\nabla \hp_0\big|^2\,dx
\,\geq\,\psi
(r_0+\delta)
\,\geq\,
\psi(\rho_d)+c\big(r_0+\delta-\rho_d\big)^2\,.
\end{equation}
To estimate the second integral in~\eqref{glob}, we proceed as in lemma~$A.1$ of \cite{BO}.
We denote by $\delta_{12}$ the difference between the first and the second eigenvalues
of the Laplacian in
$B(x_0,r_0+\delta)$. We have
$$\displaylines{
\frac{1}{2d}
\int_{\Rd}\Big(\big|\nabla \gt\big|^2
-\big|\nabla \hp_0\big|^2\Big)\,dx
\,\geq\,
\delta_{12}||\gt-\hp_0||_2^2\,.
}
$$
Again, since $r_0+\delta$ belongs to the interval
$]\rho_d-\eta,\rho_d+\eta[$, which is included in
$]\rho_d/2,3\rho_d/2[$, there exists a constant $c>0$ depending on the dimension~$d$
only such that
$\delta_{12}\geq c$. Moreover we have
$$
\varepsilon\,\leq\,||g-\phi_{x_0}||_2\,\leq\,
||g-\gt||_2\,+\,
||\gt-\hp_0||_2\,+\,
||\hp_0-\phi_{x_0}||_2\,.
$$
Now, thanks to inequalities~\eqref{aob} and~\eqref{hold},
$$||g-\gt||_2^2\,=\,
\int_{\Rd}\big((1-h)g\big)^2\,dx\,\leq\,
\int_{G\setminus B(x_0,r_0)}g^2\,dx\,\leq\,
c {\varepsilon^{12}}\,.
$$
Using the fact that the eigenfunction $\phi$ has bounded support
and is Lipschitz, we have
$$\displaylines{
||\hp_0-\phi_{x_0}||_2^2
\,=\,
\int_{\Rd}
\Bigg(
\Big(\frac{\rho_d}{r_0+\delta}\Big)^{d/2}
\phi\Big(
\frac{\rho_d}{r_0+\delta}(x-x_0)\Big)
-\phi(x-x_0)\Bigg)^2\,dx\hfill\cr
\,\leq\,
c\Bigg(\Big(\frac{\rho_d}{r_0+\delta}\Big)^{d/2}-1\Bigg)^2
+c\Bigg(\Big(\frac{\rho_d}{r_0+\delta}\Big)-1\Bigg)^2\cr
\,\leq\,
c\big({r_0+\delta}-{\rho_d}\big)^2
\,.
}
$$
The previous inequalities yield
$$\displaylines{
\frac{1}{2d}
\int_{\Rd}\Big(\big|\nabla \gt\big|^2
-\big|\nabla \hp_0\big|^2\Big)\,dx
\,\geq\,
c\Big(
\varepsilon\,-\,
c {\varepsilon^{12}}
\,-\,
c\big({r_0+\delta}-{\rho_d}\big)^2
\Big)^2
\,.}$$
Reporting in the inequality~\eqref{glob}, and using inequalities~\eqref{ulto}
and~\eqref{intert},
we get
$$\displaylines{
\big|\{\,g>0\,\}\big|
+
\frac{1}{2d}
\int_{\Rd}\big|\nabla g\big|^2\,dx\,\geq\,
\psi(\rho_d)
+
c\Big(
c\varepsilon\,-\,
c {\varepsilon^{12}}
\,-\,
c\big(c\varepsilon^{24}+\delta\big)^2
\Big)^2
-c\delta
-c\varepsilon^{24}
-c
\frac{\varepsilon^{12}}{\delta^2}-
c
\frac{\varepsilon^{6}}{\delta}\,.
}$$
Choosing $\delta=\varepsilon^3$, we conclude that,
for $\varepsilon$ small enough,
$$\displaylines{
\big|\{\,g>0\,\}\big|
+
\frac{1}{2d}
\int_{\Rd}\big|\nabla g\big|^2\,dx\,\geq\,
\chi_d+
c {\varepsilon^{2}}
\,.
}$$
Therefore the inequality stated in the lemma is satisfied for $\varepsilon$
small enough whenever the conditions~\eqref{fco} and~\eqref{sata}
are fulfilled.
\end{proof}

We shall extend slightly Lemma~\ref{A1}, in two ways. First, we will relax the
condition $||g||_2=1$, second, we will consider functions which are not $C^\infty$
on the whole space, but on the complement of a countable union of lines.
We could probably reach the Sobolev space $W^{1,2}(\Rd)$, however this won't be necessary
for our purpose.
\begin{corollary}
\label{vA1}
Let $D$ be a subset of $\Rd$ which is a countable union of lines.
Let $\varepsilon>0$.
If $g:\Rd\to\R^+$ is continuous and
	$C^\infty$ on $\Rd\setminus D$,
	and if
	$$\displaylines{
		1-\varepsilon^{49}\,\leq\,||g||_2\,\leq\,1
		+\varepsilon^{49}\,,\cr
	\forall{x\in\Rd}\quad||g-\phi_x||_2\,\geq\,2\varepsilon\,,}$$
	then, for $\varepsilon$ small enough,
$$\big|\{\,g>0\,\}\big|+\frac{1}{2d}
\int_{\Rd}\big|\nabla g\big|^2\,dx\,\geq\,\chi_d+\varepsilon^{49}\,.
$$
\end{corollary}
\begin{proof}
	The condition on the regularity is not problematic, in fact the proof of Lemma~\ref{A1}
	can be used to deal directly with these functions. The only thing we need to do is to
	rescale the function $g$ in order to have a function of $L^2$ norm one, for which we
	can use lemma~\ref{A1}.
	So let $g$ be a function satisfying the hypothesis of corollary~\ref{vA1} and let us set
	$$h\,=\,\frac{1}{\|g\|_2}g\,.$$
	We have obviously $\|h\|_2=1$. Moreover, for any $x\in \Rd$,
\begin{multline*}
	2\varepsilon\,\leq\,
	||g-\phi_x||_2\,\leq\,
	||g-h||_2\,+\,
	||h-\phi_x||_2
	\cr\,\leq\,
	\Big|1-
	\frac{1}{\|g\|_2}\Big|
	||g||_2
	\,+\,
	||h-\phi_x||_2
	\,\leq\,
	\varepsilon^{49}
	\,+\,
	||h-\phi_x||_2\,.
\end{multline*}
	We can thus apply the inequality of Lemma~\ref{A1} to $h$. Moreover, we have
\begin{multline*}
\big|\{\,h>0\,\}\big|+\frac{1}{2d}
\int_{\Rd}\big|\nabla h\big|^2\,dx\,=\,
\big|\{\,g>0\,\}\big|+\frac{1}{2d
	(||g||_2)^2
	}
\int_{\Rd}\big|\nabla g\big|^2\,dx\,
\cr
\,\leq\,
	\frac{1}{(1-\varepsilon^{49})^2}\Big(
\big|\{\,g>0\,\}\big|+\frac{1}{2d
	}
	\int_{\Rd}\big|\nabla g\big|^2\,dx\Big)\,.
\end{multline*}
In the end, we get that, for $\varepsilon$ small enough,
$$
\big|\{\,g>0\,\}\big|+\frac{1}{2d
	}
	\int_{\Rd}\big|\nabla g\big|^2\,dx\,\geq\,
	(1-\varepsilon^{49})^2(\chi_d+\varepsilon^{48})\,,
	$$
	and the last term is larger than
$\chi_d+\varepsilon^{49}$ for $\varepsilon$ sufficiently small.
\end{proof}

\subsection{Application of Faber--Krahn: proof of Theorem \ref{T:Prop31_intro}}
The time has come to apply the quantitative Faber--Krahn inequality to the random walk in order to prove
Theorem \ref{T:Prop31_intro}.
Note that
Theorem \ref{T:Prop31_intro}
is the analogue of Proposition 3.1 in \cite{BO}.


We shall apply theorem~\ref{cupbo} to the set
$\cL_n$.
To this end, we consider a function $g$ satisfying the constraints of the infimum
appearing in theorem~\ref{cupbo}.
So,
let $g$ be a continuous function
$\Rd\to\R^+$
such that
$g$ is $C^\infty$ on $\Rd\setminus \frac{1}{n}D^d$\,,
and there exists
$\ell\in\cL_n$ satisfying
$$	\big\|g^2- \ell\big\|_{L^1(\Rd)}
\,\leq\,
	\frac{3}{n^{1/16}}\,.$$
	This implies in particular that
	 $$1-\frac{3}{n^{1/16}}\,\leq\,
\big\| g\big\|_2\,\leq\,
	 1+\frac{3}{n^{1/16}}
	\,.$$
	Since $\ell$ belongs to
$\cL_n$,
then we have, for any $x\in\Rd$, for $n$ large enough,
\begin{align*}
	\frac{1}{n^{1/800}}
	&\,\leq\,
||\ell-(\phi_x)^2||_1\,\leq\,
||\ell-g^2||_1\,+\,
||g^2-(\phi_x)^2||_1\,\cr
	 &\,\leq\,
	\frac{3}{n^{1/16}}\,+\,
\int_{\Rd}
	\big|g -\phi_x\big|\times
	\big|g +\phi_x\big|
	\,dx\cr
	 &\,\leq\, \frac{3}{n^{1/16}}\,+\,
||g-\phi_x||_2\times
||g+\phi_x||_2\cr
	 &\,\leq\, \frac{3}{n^{1/16}}\,+\,
3||g-\phi_x||_2
	\,.
\end{align*}
It follows that, for $n$ large enough,
$$\forall x\in\Rd\qquad ||g-\phi_x||_2\,\geq\,
	\frac{1}{4n^{1/800}}\,.$$
	We apply the inequality of corollary~\ref{vA1} to $g$ with $\varepsilon^{49}=3/{n^{1/16}}$.
	For $n$ large enough, we have
\begin{equation}
\label{zntert}
\big|\{\,g>0\,\}\big|+\frac{1}{2d}
\int_{\Rd}\big|\nabla g\big|^2\,dx\,\geq\,\chi_d+
	\frac{1}{n^{1/16}}\,.
\end{equation}
In order to exploit the inequality of
theorem~\ref{cupbo}, we will first restrict the set of the functions $g$ which are relevant in the infimum.
Let $\kappa\geq 1$ and let $\delta>0$ be such that
$$\omega_d
\Big(\frac{\lambda_d}{{2}\delta}\Big)^{d/2}\,\geq\,2\kappa\,.$$
Let us set
$$G\,=\,\big\{\,x\in \Rd:g(x)>0\,\big\}\,.$$
Let $r>0$ be such that $|G|=r^d\omega_d$.
By the classical Faber--Krahn inequality~\ref{cfkr},
we have
$$\int_{\Rd}\big|\nabla g\big|^2\,dx\,\geq\,
\lambda(G)\,\big( \|g\|_2\big)^2\,\geq\,
\frac{\lambda_d}{r^{2}}
	 \Big(1-\frac{3}{n^{1/16}}\Big)^2
\,.$$
Suppose that
$$\int_{\Rd}\big|\nabla g\big|^2\,dx\,\leq\,\delta\,.$$
We would then have, for $n$ large enough,
$$r^2\,\geq\,
\frac{\lambda_d}{{2}\delta}\,,\qquad
|G|\,=\,
\omega_dr^d\,\geq\,
\omega_d
\Big(\frac{\lambda_d}{{2}\delta}\Big)^{d/2}\,\geq\,2\kappa
\,,$$
and for such a function $g$, the functional in the infimum is larger or equal than $2\kappa$.
This will also be the case if the Dirichlet energy of $g$ is too large, say larger than $4d\kappa$.
So, up to terms which are negligible compared to
$\exp(-\kappa n^d)$, we can restrict ourselves to functions $g$ such that
$$\delta\,\leq\,\int_{\Rd}\big|\nabla g\big|^2\,dx\,\leq\,4d\kappa\,.$$
We have then, for $n$ large enough,
\begin{multline}
	\label{simo}
\bigg(\max\Big(
\sqrt{
\int_{\Rd}\big|\nabla g\big|^2\,dx
}-
	\frac{\sqrt{d}}{n^{1/8}} ,0\Big) \bigg)^2
	\,\geq\,
\int_{\Rd}\big|\nabla g\big|^2\,dx-
	\frac{2\sqrt{d}}{n^{1/8}}\sqrt{4d\kappa}
	\,\geq\,
\int_{\Rd}\big|\nabla g\big|^2\,dx-
\frac{1}{n^{1/9}}
	\,.
\end{multline}
We apply theorem~\ref{cupbo} and we use inequality~\eqref{simo} to simplify
the infimum.
We have, for $n$ large enough,
\begin{multline*}
E\big(e^{-|R_N|};
\ell_N\in\cL_n)\,\leq\,
4\exp(-\kappa n^d)+
\hfill
\cr
\qquad
\exp\Bigg(2n^{d-1/8}-
n^d\,
(1-n^{-1/4})
	\inf\,\Big\{\,
	\big|\,\supp\,g\,\big|
+
\frac{1}
{2d}
\int_{\Rd}\big|\nabla g\big|^2\,dx-
\frac{1}{2dn^{1/9}}
\hfill\cr
\hfill
:
	g\text{ is continuous } \Rd\to\R^+
	\text{ and }C^\infty\text{ on }\Rd\setminus \frac{1}{n}D^d,\,
\quad
\exists\,\ell\in\cL\quad
	\Big\|g^2- \ell\Big\|_{L^1(\Rd)}
\,\leq\,
	\frac{3}{n^{1/16}}
\,\Big\}\Bigg)\,.
\end{multline*}
We can finally use inequality~\eqref{zntert}!
We obtain
\begin{multline*}
E\big(e^{-|R_N|};
\ell_N\in\cL_n)\,\leq\,
4\exp(-\kappa n^d)+
\exp\Bigg(2n^{d-1/8}-
n^d\,
(1-n^{-1/4})
	\Big(
\chi_d+ \frac{1}{n^{1/16}}
-\frac{1}{2dn^{1/9}}
	\Big)\Bigg)\,.
\end{multline*}
By choosing $\kappa> \chi_d$, for $n$ large enough, we obtain
the statement of theorem \ref{T:Prop31_intro}.

\section{Filling the ball and proof of Theorem \ref{T:filledball_intro} }

\label{S:filling}

We now come to the proof of Theorem \ref{T:filledball_intro}. We will show that with high probability under $\wP_N$ a ball of approximately the right radius is entirely filled. The analogue of this result in \cite{BO} (Proposition 4.1), is however only valid for $d=2$, whereas the proof below holds for any $d \ge 2$. Fix $x \in \R^d$, and let
$$
\cG_{loc,x} = \left\{ \|\ell_N - (\phi_x)^2 \|_1 \le 1/ n^{s}\right \}
$$
with $s = 1/800$ is as in Theorem \ref{T:Prop31_intro} and we recall that the rescaled local time profile $\ell_N$ is defined in \eqref{D:rescaledLN}. For $\kappa>0$ and $x \in \R^d$,
$$
\cR_{\kappa,x} = \big\{\forall  z \in B(x, \rho_d  (1 - n^{- \kappa}) ) \cap \Z_n^d,\quad \ell_N(z)>0 \big \}.
$$
This is the event that the Euclidean ball of radius $\rho_d n (1 - n^{-\kappa})$ around $nx$, restricted to the unscaled lattice $\mathbb{Z}^d$, is filled by the walk. The main result of this section, which is the analogue of Proposition 4.1 in \cite{BO}, is the following.

  \begin{theorem}\label{T:filledball} There exists $\kappa>0$ such that for all $n$ (or equivalently $N$) large enough,
  $$
  \frac1{Z_N}\E( e^{ - |R_N|} ; \cG_{loc,x} ; (\cR_{\kappa,x})^c) \le \exp( - n^{\kappa}).
  $$
\end{theorem}

\subsection{Time spent in mesoscopic balls}
\label{S:time_meso}

We recall that constants $C, c>0$ denote constants depending only on the dimension, whose precise numerical value is allowed to change from line to line. We will use Landau's notations $O(\cdot), o(\cdot)$ where the implicit constants are allowed to depend on the dimension only. The notation $\lesssim, \gtrsim, \asymp$ denote inequalities and equality up to constant respectively: thus $a_n \lesssim b_n$ means $a_n \le C b_n$ for some constant $C$ (depending only on the dimension). $\kappa$ will be a small parameter eventually chosen in a way that depends only on the dimension, so at the end of the proof we will be able to absorb its value in a generic constant $c>0$, but we will refrain from doing so during the course of the proof.

Without loss of generality we take $x = 0$ in the rest of the proof of Theorem \ref{T:filledball}; write $\cG_{loc}$ and $\cR_\kappa$ for $\cG_{loc, x}$ and $\cR_{\kappa, x}$. Fix $z \in B( 0, \rho_d n(1- n^{ - \kappa}))$. We let $m = \rho_d n^{1- 2\kappa}$ where $\kappa$ is sufficiently small; $m$ is a mesoscopic scale (quite close to $n$) and our first task will be to control the amount of time spent in a ball of that scale around the point $z$. Let
$$
B = B(z, m) \cap \Z^d
$$
be the discretised ball of radius $m$ around $z$.
Note that $B \subset B(0, \rho_d n(1- n^{-\kappa}/2))$.
The idea will be to condition on some information outside $B$ including the local time of the random walk on every site outside of $B$. Let $B^\circ$ denote the ball of radius $m/2$ around $x$, i.e., $B^\circ = B(x, m/2) \cap \Z^d$.

We make the following simple deterministic observations (recall that we expect $L_N(x)$ to be typically of order $n^2$ at any point in the bulk of $R_N$).
\begin{lemma}\label{L:det_time}
If $\cG_{loc}$ holds, then necessarily, for some sufficiently small but fixed $\kappa$ and $\delta>0$ depending only on the dimension $d$, we have
\begin{equation}\label{timeBcirc}
L_N(B^\circ) \ge \delta m^dn^{2- 2\kappa}.
\end{equation}
Furthermore,
\begin{equation}\label{range_det}
|R_N| \ge \omega_d( \rho_d n)^{d} - c n^{d-\eps}
\end{equation}
 where $\eps$ depends only $d$.
\end{lemma}

\def\ph{\varphi}
\begin{proof}
Let $\ph = \phi^2$. Note that as $z \to \partial B(0, \rho_d)$, for some constant $C>0$ depending only on the dimension,
\begin{equation}\label{radialderivative}
\ph(z) \sim C d(z)^2, \text{ where } d(z) = \dist_{\R^d} (z, \partial B(0, \rho_d)).
\end{equation}
 Indeed recall that $\phi$, the first eigenfunction in $B(0, \rho_d)$, has nonzero normal derivative on $\partial B(0, \rho_d)$.

Let $y \in B$. Then
$\dist (y, \partial B(0, \rho_d n)) \ge \rho_d n^{1- \kappa}/2$ and so $d ( y / n)  \ge \rho_d /(2 n^{\kappa})$.
Now suppose for contradiction that $L_N(B^\circ) \le \delta  m^dn^{2- 2\kappa}$. Then necessarily, by definition of $\ell_N$ as a function on $\R^d$ in \eqref{D:rescaledLN},
$$
\int_{(B^\circ/n)} \ell_N (y) dy \le \sum_{y \in \Z^d \cap B^\circ} \frac1{n^{d+2}}L_N(y) = \frac{L_N(B^\circ) }{ n^{d+2}} \le \delta (m/n)^d n^{-2\kappa} =\delta \rho_d^d n^{ - 2(d+1) \kappa}
$$
whereas
$$
\int_{(B^\circ/n)} \ph(y) dy \ge \omega_d (\tfrac{m/2}{n})^d C \rho_d^2 n^{-2 \kappa}/4 = (\tfrac{C \rho_d^{d+2}\omega_d }{2^{2+d}}) n^{ - 2(d+1) \kappa}.
$$
Thus if we take $\delta = C \rho_d^{d+2}\omega_d / 2^{3+d}$,
$$
\| \ell_N - \ph\|_1 \ge \Big| \int_{(B^\circ/n)} \ell_N(y) dy - \int_{(B^\circ/n)} \ph(y)dy \Big| \ge ( \tfrac{C \rho_d^{d+2}\omega_d}{2^{3+d}}) n^{- 2(d+1) \kappa}.
$$
Hence if $\kappa$ is small enough that $2(d+1) \kappa \le s$ we see that this cannot hold at the same time as $\cG_{loc}$. This shows \eqref{timeBcirc} with this choice of parameters.

The proof of \eqref{range_det} follows a similar argument.
\end{proof}


\subsection{Bridges in mesoscopic balls}

\label{S:bridge_meso}

Let us call an interval of time $[t_1, t_2]$ a \textbf{bridge} if $X[t_1, t_2] \subset B$ and $X_{t_1}, X_{t_2} \in B^\circ$. (We warn the reader that this differs from the more standard notion of bridge). We call $a = X_{t_1}$ and $b = X_{t_2}$ its endpoints and $t_2 - t_1$ its length; note that we require that $a, b \in  B$ be far away from $\partial B$; in fact we require them to be in $B^\circ$.
We will consider bridges of length at least $m^2$ and our first goal will be to show that there are sufficiently many such disjoint bridges.

   Let $\cG_{br} = \cG_{br, \delta} $ be the event that there are at least $a\delta n^{2- 2\kappa} m^{d-2}$ bridges of length $m^2$, where $\delta $ is as in Lemma \ref{L:det_time}, and where $a $ is a suitable constant depending also only on the dimension, which will be chosen below.

\begin{lemma}
  \label{L:numexc}
We can choose $\kappa$ sufficiently small and $a>0$, depending only on the dimension, so that
 $$
 \E ( e^{ - |R_N|}; \cG_{loc}; (\cG_{br})^c  ) \le Z_N \exp( - c n^{d - 6\kappa}) ,
 $$
 for some constant $c>0$ depending only on $d$.
\end{lemma}

\begin{proof}
We already know that the walk spends at least $\delta  n^{2-2\kappa} m^d$ units of time in $B^\circ$. Roughly speaking, every time the walk is in $B^\circ$ there is a positive probability that during the next $m^2$ units of time, the walk stays in $B$ and its position at the end of this interval is again in $B^\circ$, thereby completing a bridge; independently of the past. So we wish to use standard Chernoff bounds for deviations of binomial random variables. In order to implement this strategy, we must however take care that we are working under the weighted probability measure $\wP_N$ and not $P$. We will deal with this complication by performing a suitable change of measure (which as it turns out is essentially the same as the one used by Bolthausen in his proof of Proposition 4.1 in \cite{BO}). The first step will be to work in continuous time rather than discrete time. By Lemma \ref{L:det_time} (and more precisely \eqref{range_det}) we write
$$
  \E ( e^{ - |R_N|} ;\cG_{loc} ;(\cG_{br})^c  )  \le \exp \big(- \omega_d(\rho_d n)^d + c n^{d-\eps}\big) \P( \cG_{loc}; \cG_{br}^c)
$$
and we interpret the event in the right hand side of the above inequality as an event for the jump chain of a continuous Markov chain whose jump rates from $x $ to $y$ is $1/(2d)$ if $x$ and $y$ are neighbours and zero else. Let $(\widetilde X_t, t \ge 0)$ be this process, let $\widetilde \P$ denote its law and let $J(\widetilde X)$ be the jump chain of $\widetilde X$. Set $\widetilde \cG_{loc} = \{J( \widetilde X) \in \cG_{loc}\}$ and, similarly, $\widetilde \cG_{br} = \{ J( \widetilde X) \in \cG_{br}\}$. Then $\P( \cG_{loc} ; ( \cG_{br})^c) ) = \widetilde \P( \widetilde \cG_{loc}; (\widetilde \cG_{br})^c)$.

\begin{figure}
  \includegraphics[width=\textwidth]{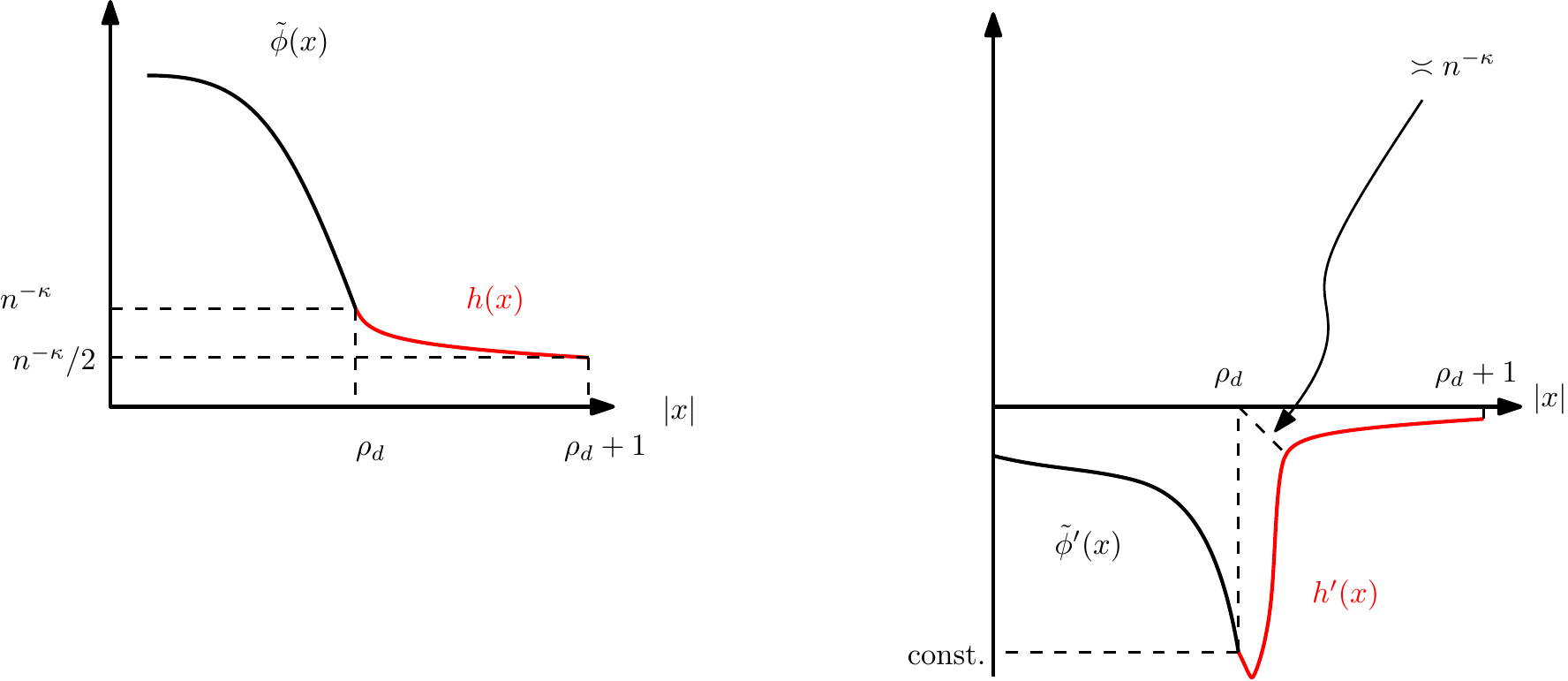}
  \caption{Choice of function $h$: as desired, $h$ decreases from $n^{- \kappa}$ to $n^{ - \kappa}/2$ over $[\rho, \rho+1]$. This requires the area above the red function on the right hand side (its derivative) to be less than $n^{-\kappa}/2$. The value of the constant in the right hand side is the normal (radial) derivative of $\phi$. To achieve this, the second derivative will be at most $\lesssim n^{\kappa}$, as claimed in \eqref{hsec}.}
  \label{F:eigenmodif}
\end{figure}

Define a modification $\widetilde \phi$ of $\phi$ as follows:
$$
\widetilde \phi (x)  = \phi(x) + n^{-\kappa} \quad  ; \quad x \in B( 0, \rho_d) 
$$
We will also define $\widetilde \phi$ outside of $B( 0, \rho_d)$
so that it is positive everywhere and also in such a way that it is reasonably smooth near the boundary of that ball (indeed, if we do not take a positive function we cannot use it to change the underlying probability measure). Of course we could set $\widetilde \phi$ to be constant outside of that ball, equal to its value on the boundary, but this turns out to not be sufficiently smooth; in particular the Laplacian on this sphere would be too large.

Instead we define $$\widetilde \phi (x) = h (|x|)$$ where
$$h : [ \rho_d, \rho_d +1] \to [ n^{-\kappa}/2 , n^{- \kappa}]$$ is a smooth monotone decreasing convex function such that its derivative at $1+ \rho_d$ is 0 while its derivative at $\rho_d$ is the radial derivative of $\phi$ on $\partial B(0, \rho_d)$, 
and such that
\begin{equation}\label{hsec}
\sup_{ x \in [ \rho_d, \rho_d+1]} \frac{d^2}{dx^2} h \lesssim n^{\kappa}.
\end{equation}
It is elementary to check that such a function $h$ exists, so that $\widetilde \phi$ is well defined (see Figure \ref{F:eigenmodif} for an illustration of the function $\tilde \phi$ and its derivative).

With the help of $\widetilde \phi$ we can define a new probability measure $\widetilde \Q$ to be the law of the Markov chain whose transition rates are given by
$$
q(x,y) = \frac{\widetilde \phi (y/n)}{\widetilde \phi (x/n)}
$$
whenever $x\neq y$ are neighbours in $\mathbb{Z}^d$, and $q(x,y) = 0$ otherwise. Note that since $\widetilde \phi$ is positive, $\tilde \Q$ is indeed equivalent to $\tilde \P$ and furthermore, letting $\tilde t_N$ be the first time $\tilde X$ has jumped $N$ times,
\begin{equation}
\left.\frac{d\widetilde \P}{d\widetilde \Q}\right|_{\cF_{\tilde t_N}} = \frac{\widetilde \phi(\widetilde X_0 /n)}{\widetilde \phi (\widetilde X_{\tilde t_N}/n)} \exp\left(\int_0^{\tilde t_N} \frac{\Delta_1 \widetilde \phi}{\widetilde \phi}(\widetilde X_s/n) ds \right).
\label{RN}
\end{equation}
(See, e.g., e.g. \cite{RogersWilliams}, IV, (22.8)). Here recall that $\Delta_1$ is the discrete Laplacian, i.e. $2d$ times the generator of the Markov chain $\widetilde X$ under $\widetilde \P$.

\medskip \noindent \textbf{Step 1.} We will argue that
\begin{equation}\label{step1}
\E( e^{- |R_N|}; \cG_{loc}; (\cG_{br})^c) \le Z_N  e^{- N^{1- 5 \kappa}} + Z_N e^{3 n^{d- \eps}} \tilde \Q( \tilde \cG_{loc}; (\tilde \cG_{br})^c).
\end{equation}

Essentially the proof consists in analysing carefully the integral in \eqref{RN}.
 We follow roughly the arguments in \cite{BO} (see equations (4.3) and (4.4)), with additional details.
 We start by observing that if $x \in B(0, \rho_d n) \cap \Z^d$ is such that all its $2d$ neighbours are also in $B(0, \rho_d n) \cap \Z^d$, then by a Taylor expansion,
\begin{align*}
\Delta_1 \tilde \phi (x/n) &=  \Delta_1 \phi(x/n) = \frac12 \Delta \phi(x/n) + O( n^{-3}) \\
&  = - \frac{\lambda}{2n^2} \phi (x/n) + O (n^{ - 3})\\
& = - \frac{\lambda}{2n^2} \tilde \phi (x/n)  + O ( n^{- 3}) + O( n^{ - 2 - \kappa}).
\end{align*}
Assume that
$\kappa <1$ without loss of generality so that $O( n^{-3}) = O ( n^{ - 2 - \kappa})$. Hence if furthermore $x \in B(0, \rho_d n( 1- n^{ - 0.9\kappa}))$,
\begin{equation}\label{chgms_eigen1}
\big | \frac{\Delta_1 \tilde \phi}{\tilde \phi} (x/n) + \frac{\lambda}{2n^2} \big | \lesssim n^{-2 - \kappa + 0.9\kappa} \le n^{ - 2 - 0.1 \kappa}
\end{equation}
using \eqref{radialderivative}.

Furthermore, for any other $x$,
$$
|\Delta_1 \tilde \phi (x/n)| \lesssim n^{ - 2 + \kappa}
$$
by \eqref{hsec} (where the implicit constant depends only on $\tilde \phi$ and so on the dimension), so that
\begin{equation}
  \label{chgms_eigen2}
  \big | \frac{\Delta_1 \tilde \phi}{\tilde \phi} (x/n) + \frac{\lambda}{2n^2} \big | \lesssim n^{-2+ 2\kappa}
\end{equation}
for such $x$, and the same remark holds about the implicit constant.

Combining \eqref{chgms_eigen1} and \eqref{chgms_eigen2}, it follows that
\begin{equation}\label{chgmeas_int}
\int_0^{\tilde t_N} \frac{\Delta_1 \tilde \phi}{\tilde \phi}(\tilde X_s/n)ds  = ( - \lambda/2 )\frac{{\tilde t_N}}{n^2} + O ( \tilde t_N n^{-2 -0.1 \kappa}) + O \Big( n^{-2 + 2\kappa}\tilde L_{\tilde t_N} \big( B(0, \rho_d n( 1- n^{ - 0.9 \kappa}))^c\big)\Big),
\end{equation}
where for a set $A \subset \Z^d$ and a time $t>0$, $\tilde L_t( A) = \int_0^t 1_{\{ \tilde X_s \in A\}} ds$ is the local time of $\tilde X$ (in continuous time) in the set $A$ up to time $t$.

Set $\tilde \cA = \tilde \cG_{loc} \cap (\tilde \cG_{br})^c$. Set $\tilde \cA' $ to be the event
$$
\tilde \cA\quad ; \quad |\tilde t_{N} - N | \le N^{1-\kappa}  \quad ; \quad \tilde L_{\tilde t_N} (B(0,\rho_d n(1- n^{ - 0.9 \kappa})^c) \le N n^{-2.1 \kappa}.
$$
Then, still writing $R_N$ for the range of the (continuous time) walk at the time $\tilde t_N$ of its $N$th jump,
\begin{equation}\label{A'}
\E_{\tilde \P}( e^{- |R_N|}; \tilde \cA) \le \E_{\tilde \P}( e^{-|R_N|}; \tilde \cA') + \tilde \P( \tilde \cA \setminus \tilde \cA') .
\end{equation}
We bound separately each of those terms. We start with the first term. On $\cA'$, we see that
$$
\int_0^{\tilde t_N} \frac{\Delta_1 \tilde \phi}{\tilde \phi}(\tilde X_s/n)ds  = ( - \lambda /2) n^d  + O (n^{d- (d+2) \kappa}) + O ( n^{d- 0.1\kappa}) + O ( n^{d- 0.1 \kappa})
$$
so that if $\kappa < \eps/(d+2)$ (which implies $\kappa < 10 \eps$), all error terms are $O (n^{d- (d+2)\kappa}) \le n^{d- \eps}$ and thus, plugging into \eqref{RN}, and using the fact that $\max_{B(0, \rho_d)} \tilde \phi \lesssim 1  $, $\min_{B(0, \rho_d)} \tilde \phi \gtrsim n^{- \kappa}$, as well as the already established lower bound of Proposition \ref{P:lbZ} on $Z_N $,
\begin{align}
\E_{\tilde \P} (e^{- |R_N|}; \tilde \cA') &\lesssim n^\kappa\exp\big( - \omega_d (\rho_d n)^d - (\lambda/2) n^d + 2n^{d- \eps} \big) \tilde \Q( \tilde \cA') \nonumber \\
& \le Z_N e^{2n^{d- \eps} + O( n^{d-1})} \tilde \Q( \tilde \cA')\nonumber \\
& \le Z_N e^{3n^{d- \eps}} \tilde \Q( \cG_{loc} ; (\cG_{br})^c).\label{chgmeasA'}
\end{align}
Let us now deal with the second term in \eqref{A'}. We have,
\begin{equation}\label{decA'}
\tilde \P( \tilde \cA\setminus \tilde \cA' ) \le \tilde \P( |\tilde t_N - N| \ge N^{1- \kappa}) + \tilde \P ( \tilde L_{\tilde t_N} (B(0, \rho_d n(1 - n^{- 0.9 \kappa}))^c) \ge N n^{- 2.1\kappa}; \tilde \cG_{loc}).
\end{equation}
Now, standard Chernoff estimates for exponential random variables show that
\begin{equation}\label{devtN}
\P( |\tilde t_N - N| \ge N^{1- \kappa})  \le \exp( - c N^{1- 3 \kappa})
\end{equation}
for some constant $c>0$. Furthermore, on $\cG_{loc}$, deterministically we have
$$
L_N (B(0, \rho_d n ( 1- n^{ - 0.9 \kappa})^c) \le N n^{-2.2 \kappa},
$$
since otherwise, reasoning as in Lemma \ref{L:det_time},
$$
\int_{B(0, \rho_d ( 1- n^{ - 0.9 \kappa}))^c} \ell_N (y) dy \ge n^{-2.2 \kappa} \text{ while }\int_{B(0,\rho_d ( 1- n^{ - 0.9 \kappa}))^c} \phi^2(y) dy \le C \omega_d n^{- 2.7 \kappa}
$$
and we would deduce that for $n$ large enough, $\| \ell_N - \phi^2\| >n^{-2.3\kappa}$ which by definition cannot take place on $\cG_{loc}$ if $\kappa \le s/2.3$ (with $s = 1/800 $ as in
Theorem \ref{T:Prop31_intro}. Hence
\begin{equation}\label{devLgamma}
\tilde \P ( \tilde L_{\tilde t_N} (B(0, \rho_d n(1- n^{ - 0.9 \kappa}))^c) \ge Nn^{- 2.1 \kappa}; \tilde \cG_{loc}) \le \P( \sum_{i=1}^{ \lfloor Nn^{-2.2 \kappa}\rfloor } X_i  \ge Nn^{- 2.1\kappa})
\end{equation}
where $X_i$ are independent unit exponential random variables. Hence for all $n$ (or equivalently $N$) large enough,
$$
\P( \sum_{i=1}^{ \lfloor Nn^{-2.2 \kappa}\rfloor } X_i  \ge Nn^{-2.1 \kappa}) \le
\P( \sum_{i=1}^{ \lfloor Nn^{-2.2 \kappa}\rfloor } X_i  \ge 2 Nn^{-2.2 \kappa}) \le \exp ( - c N n^{-2.2 \kappa}) \le \exp ( - N^{1 - 3\kappa})\,,
$$
for some constant $c>0$, by elementary Chernoff estimates for exponential random variables. Hence by \eqref{devLgamma},
$$
\tilde \P ( \tilde L_{\tilde t_N} (B(0, \rho_d n(1- n^{ - 0.9 \kappa}))^c) \ge Nn^{-2.1 \kappa}; \tilde \cG_{loc})  \le \exp( - N^{1- 3 \kappa}).
$$
Plugging into \eqref{decA'}, we deduce that if $\kappa$ is sufficiently small, for all $N$ large enough,
$$
\tilde \P( \cA \setminus \cA') \le \exp ( - N^{1 - 4 \kappa}).
$$
Since $Z_N = \exp (- O( n^d))$, we deduce that
\begin{equation}\label{controlA'}
\tilde \P( \cA \setminus \cA') \le Z_N \exp( - N^{1- 5 \kappa}) .
\end{equation}
Combining \eqref{controlA'}, \eqref{chgmeasA'} and \eqref{A'}, we obtain
$$
\E_{\tilde \P} ( e^{- |R_N|}; \cG_{loc}; (\cG_{br})^c) \le Z_N  e^{- N^{1- 5 \kappa}} + Z_N e^{3 n^{d- \eps}} \tilde \Q( \tilde \cG_{loc}; (\tilde \cG_{br})^c),
$$
as desired in \eqref{step1}.

%

\medskip \noindent \textbf{Step 2.} Now it remains to show that $\widetilde \cG_{loc} \cap \tilde \cG_{br}$ is overwhelmingly likely under $\widetilde \Q$. More precisely, we will argue that
\begin{equation}\label{step2}
\tilde \Q (  (\tilde \cG_{br})^c; \tilde \cG_{loc}) \le \exp ( - c n^{2- 2\kappa} m^{d-2}).
\end{equation}

We claim that every time the walk is in $B^\circ$, there is a probability bounded below by a constant $p$, say, depending only on the dimension, such that under $\tilde \Q$, the walk will perform a bridge of duration $m^2$ (recall that this means the walk remains in $B$ for the next $m^2$ jumps and ends up in $B^\circ$ again after this time), uniformly over the initial position in $B^\circ$ of the walk. To see this, note that
$\min_{B/n} \widetilde \phi = (1- o(1)) \max_{B/n} \widetilde \phi$, so that the Radon--Nikodym derivative $d\widetilde \Q / d\widetilde \P$ during an interval of time consisting of the first $m^2$ jumps of the chain (call this time $\tilde t_{m^2}$) is at least
\begin{align}
\frac{d\tilde \Q }{ d\tilde \P} &\ge
\frac{\min_{B/n} \widetilde \phi }{\max_{B/n} \widetilde \phi} \exp ( \tilde t_{m^2}\inf_{x \in B} \frac{\Delta_1 \widetilde \phi }{\widetilde \phi }(x/n) )\nonumber \\
& \ge (1+o(1)) \exp ( - (\lambda/2 + o(1)) \frac{\tilde t_{m^2}}{n^2}) . \label{controltime}
\end{align}
Now, under $\widetilde \P$, $\tilde t_{m^2} \le 2 m^2$ with probability at least $3/4$, if $m$ (equivalently $n$ or $N$) is large enough. Moreover, the probability of making a bridge of duration $m^2$, under $\widetilde \P,$ is clearly at least $p$ for some constant $p>0$ depending only on $d$. Since the latter event depends only on the jump chain and the former event depends only on the time parametrisation, which are independent processes under $\widetilde \P$, we conclude from \eqref{controltime} that for any $a \in B^\circ$,
\begin{align*}
\tilde \Q_a(X[0, \tilde t_{m^2}] \text{ is a bridge}) &\ge \E_{\tilde \P_a} \big( \frac{d\tilde \Q_a}{d\tilde \P_a} ;X[0, \tilde t_{m^2}] \text{ is a bridge}\big)\\
& \ge (1+ o(1))\E_{\tilde \P_a} \big(\exp( - (\lambda/2 + o(1)) \frac{\tilde t_{m^2}}{n^2}) ; \tilde t_{m^2} \le 2m^2 ; X[0, \tilde t_{m^2}] \text{ is a bridge} )\big)\\
& \ge (3/4) (1+ o(1)) \tilde \P_a(X[0, \tilde t_{m^2}] \text{ is a bridge}) \ge  (3/4) p (1+ o(1)).
\end{align*}
Hence $\tilde \Q_a( X[0, \tilde t_{m^2}]\text{ is a bridge}) \ge p/2$ for $n$ large enough; thereby proving what we desired (with $p$ replaced by $p/2$, a distinction which is of no consequence in the rest of the argument).

Since we know by \eqref{timeBcirc} that on $ \tilde \cG_{loc}$, the total amount of discrete steps in $B^\circ$ (and hence in $B$) is deterministically at least $\delta m^d n^{2- 2\kappa}$, and each time the walk is in $B^\circ$ there is a probability $p/2$ (under $\tilde \Q$) to make a bridge over the next $m^2$ jumps, we deduce that
the number of bridges stochastically dominates a binomial random variable of parameters $(\delta m^{d-2} n^{2-2\kappa}, p/2 )$, under $\tilde \Q$. More precisely, let
$$\sharp (s,t): = \text{ number of jumps by $\tilde X $ during } (s, t],
$$
and define the sequence of stopping times:
$$
\sigma_1 = \inf \{ t>0 : \tilde X_t  \in B^\circ\}; \quad  \tau_1 = \inf \{t>0: \sharp (\sigma_1, t) \ge m^2\} \wedge \inf \{ t > \sigma_1: \tilde X_t \notin B^\circ\};
$$
then inductively, for $i \ge 2$:
$$
\sigma_i = \inf \{ t > \tau_{i-1} : \tilde X_t \in B^\circ\} ; \quad \tau_i = \inf \{ t>0 : \sharp (\sigma_i, t) \ge m^2\} \wedge \inf \{ t> \sigma_i: \tilde X_t \notin B^\circ\}.
$$
Let $$
\xi_i = 1_{\{\sharp(  \sigma_i, \tau_i) \ge  m^2\}}; \quad i \ge 1
$$
be the Bernoulli variable which is the indicator of the event that the $i$th trial results in a bridge. Then note that by Lemma \ref{L:det_time}, if $j = \lfloor \delta m^{d-2} n^{2- 2\kappa}\rfloor $, then $\tau_j \le \tilde t_N$ no matter what on $\tilde \cG_{loc}$. Hence if $S = \sum_{i=1}^j \xi_i$, then
$$
\{S \ge (\delta p /4) m^{d-2} n^{2- 2\kappa}\} \subset \tilde \cG_{br}
$$
where the constant $a>0$ defining $\cG_{br}$ is taken to be $a = p /4$.

Moreover by the Markov property and the above,
$$
\tilde  \Q( \xi_i = 1 | (X_t, t \le \sigma_i)) \ge p/2.
$$
Hence $S$ dominates stochastically a Binomial random variables with parameters $j$ and $p/2$. By standard Chernoff bounds for binomials, for some constant $c>0$, we deduce that
$$
\tilde \Q (  (\tilde \cG_{br})^c; \tilde \cG_{loc}) \le \exp ( - c n^{2- 2\kappa} m^{d-2}),
$$
which shows \eqref{step2}.

Plugging \eqref{step2} into \eqref{step1}, we deduce
$$
 \E ( e^{ - |R_N|} ; {\cG_{loc}}; {(\cG_{br})^c}  )  \le Z_N (e^{- N^{1- 5\kappa}} + e^{ n^{d-\eps}} e^{- c n^{2- 2\kappa}   m^{d-2}}).
$$
Since $m = \rho_d n^{1- 2\kappa}$,
$\eps$ is fixed (by Lemma \ref{L:det_time}) in a way that depends only on the dimension, and we are free to choose $\kappa$ as small as we want, we can choose it so that $n^{2- 2 \kappa} m^{d-2} = n^{d - 6\kappa}$ is much greater than $ n^{d- \eps}$ (i.e., we assume $6\kappa < \eps$) and then
for all $n$ large enough we have
$$
 \E ( e^{ - |R_N|} ; {\cG_{loc}} ; {(\cG_{br})^c}  )  \le Z_N (e^{ - N^{1- 5 \kappa}} + e^{ - c n^{2- 2\kappa} m^{d-2}}) \le  Z_N e^{ - c n^{d- 6 \kappa}}.
$$
This completes the proof of Lemma \ref{L:numexc}.
\end{proof}

\subsection{Conditioning and summation}

\label{S:sum}

Call $\P^*$ the conditional probability given the local time at every site in $\mathbb{Z}^d \setminus B$. Let $k\ge 1$, $\cX = \{x_1, \ldots, x_k \}\subset  B $ be a subset of $k$ distinct points in $B$ and let $\cB_{\cX}$ be the bad event that the range in $B$ avoids exactly those $k$ points, i.e., $R_N \cap B = B \setminus \cX$.
Note that
\begin{equation}\label{eq:decomp}
\{ B \not \subset R_N \} \subset  (\cG_{br})^c \cup \bigcup_{k =1}^{|B|} \bigcup_{\cX \subset  B; |\cX| = k} (\cB_{\cX} \cap \cG_{br}).
\end{equation}
So
we are led to try and analyse expectations of the form
$$
\frac1{Z_N} \E\left[ e^{ - |R_N\cap B^c|} \E^*[ e^{- | R_N \cap B| }1_{\cB_{\cX} \cap \cG_{br}} ] \right]
$$
where $\cX \subset B$ and $| \cX | = k \ge 1$ is arbitrary between $1$ and $|B|$. The key will be the following estimate:
\begin{lemma}For constants $C,c$ depending only on the dimension $d$,
  \label{L:badmiss}
  $$
  \E^*\left[ e^{- | R_N\cap B| }1_{\cB_{\cX} ; \cG_{br}} ] \right] \le Ce^{- |B| +k } \exp(- c \frac{n^{2 - 2\kappa} k}{ m^2(\log m)^{2+ 10d}} ).
  $$
\end{lemma}
We defer the proof of Lemma \ref{L:badmiss} and instead show how it implies Theorem \ref{T:filledball}.

\begin{proof}[Proof of Theorem \ref{T:filledball}, assuming Lemma \ref{L:badmiss}] Using \eqref{eq:decomp} and Lemma \ref{L:numexc}, we see that since $|R_N\cap B| \le |B|$,
\begin{align*}
\frac1{Z_N}\E( e^{ - R_N} 1_{ \cG_{loc}} 1_{B \not \subset R_N}) & \le
e^{- cn^{d- 6\kappa}}+  \sum_{k=1}^{|B|} \sum_{\cX \subset B, |\cX| = k} \frac1{Z_N} \E\left[ e^{ - |R_N\cap B^c|} \E^*[ e^{- | R_N \cap B| }1_{\cB_{\cX} ; \cG_{br}} ] \right]
\\
& \le e^{ - cn^{d-6 \kappa }} + C\sum_{k=1}^{|B|} \sum_{\cX \subset B, |\cX| = k}  \frac{1}{Z_N} \E\left[ e^{ - |R_N\cap B^c|} e^{- |B| +k}  e^{- c \frac{n^{2- 2\kappa} k}{ m^2(\log m)^{2+ 10d}} }\right]\\
& \le  e^{ - cn^{d-6 \kappa}}  + C \sum_{k=1}^{|B|} {|B| \choose k} e^k e^{- c \frac{n^{2- 2\kappa} k}{ m^2(\log m)^{2+ 10d}} } \frac{1}{Z_N} \E\left[ e^{ - |R_N|}\right]\\
& \le e^{ - cn^{d-6 \kappa}}  + C \sum_{k=1}^{m^d}{|B|  \choose k} \exp\left( k (1 - c\frac{n^{2- 2\kappa} k}{ m^2(\log m)^{2+ 10d}})\right).
\end{align*}
Now note that on the one hand, the entropic factor satisfies ${|B| \choose k}  \le (\omega_d m^{d})^k = e^{ C k d\log m}$. On the other hand since $m = \rho_d n^{1-2\kappa}$ we have that  $n^{2 - 2\kappa} / (m^2 (\log m)^{2+ 10d}) \gtrsim n^{2\kappa}/(\log n)^{2+ 10d}$. This is of course much greater than the exponential factor in the entropic term of $C \log m$, and so altogether the above series is exponentially decaying. Hence we can conclude that
$$
\frac1{Z_N}\E( e^{ - R_N} 1_{ \cG_{loc}} 1_{B \not \subset R_N}) \le C \exp( - n^{\kappa} )
$$
where $\kappa>0$ and $C$ depend only on the dimension. Summing over all $O(n^d)$ possible centres of the ball $B$ and using a union bound we immediately deduce the statement of Theorem \ref{T:filledball}.
\end{proof}

\subsection{Transition probabilities for bridges}

\label{S:bridge_trans}

We now start the proof of Lemma \ref{L:badmiss}. The idea is to show that for each bridge of duration at least $m^2$ there is a good chance of hitting our $k$ points. For this we will need the following bounds on the heat kernel of bridge; we will now further condition on the endpoints of the bridge. Let $\P^{a \to b; \tau}$ denote the law of a bridge of duration $\tau$ starting from $a \in B$ and ending in $b$, i.e., simple random walk starting from $a$, conditioned to be in $b$ at time $\tau$ and to remain in $B$ throughout $[0, \tau]$. Implicit in this notation is the fact that $\P^a (X_\tau = b) >0$, i.e., the parity of $b-a$ is the same as $\tau$.

We start with the following lemma.

\begin{lemma}
  \label{L:stayinball}Suppose $\dist (x, \partial B)  =r$ with $1\le r \le m$. Then if $s \asymp m^2$,
      \begin{equation}\label{eq:Exc}
  \P_x( X[0, s] \subset B) \asymp \frac{r}{m}.
  \end{equation}
\end{lemma}
The right hand side is essentially the familiar gambler's ruin probability in one dimension. Here it is important that we use a curved ball $B$ and not a box (otherwise if $x$ is near a corner the probability would be much smaller). This lemma could be deduced from Proposition 6.9.4 in \cite{LL} but we choose to include a proof in order to make the presentation as self-contained as possible.

\begin{proof}
The lower bound follows easily from an eigenvalue estimate and optional stopping: let $\lambda^1_B$ denote the principal eigenvalue of the (discrete) Laplacian $-\Delta_1$ in the ball $B$ with Dirichlet boundary conditions.
Let $\phi^1_B$ is the corresponding eigenfunction, normalised so that $m^{-d} \sum_{x \in B} \phi(x) = 1$.
Then note that $$
M_s = \frac{1}{(1-\lambda_B^1)^s} \phi_B^1 (X_s),\quad  s = 0, 1, \ldots
$$
is a nonnegative martingale. Apply the optional stopping theorem at the bounded stopping time $\tau_B \wedge s$ (where $\tau_B$ is the first time the walk leaves $B$) to see that
$$
\phi_B^1 (x) = \frac1{(1- \lambda_B^1)^s} \E(\phi_B^1(X_s) 1_{\tau_B > s})\le \frac{\| \phi_B^1\|_\infty \P_x( \tau_B > s)}{ (1- \lambda_B^1)^s}.
$$
Now, we have already mentioned that $\lambda_B^1 \asymp 1/m^2$ (with implied constants depending as usual only on the dimension) so that when $s \asymp m^2$,
$$
\phi_B^1 (x) \lesssim \| \phi_B^1\|_\infty \P_x( \tau_B > s).
$$
Moreover, using Lemma 2.1 of \cite{BO}, and using known properties of the principal eigenfunction in the continuum (namely that there is a radial derivative on the boundary of the ball), we see that
$$
\phi_B^1 (x) \lesssim \frac{r }{ m} + O(1/m) \lesssim \frac{r }{ m}.
$$
Since furthermore the principal eigenfunction on the unit ball in the continuum is bounded, using again Lemma 2.1 in \cite{BO}, we deduce the lower bound
\begin{equation}\label{Tball}
\P_x ( \tau_B > s) \gtrsim \frac{r }{ m} ,
\end{equation}
which gives the desired lower bound.

In the other direction, let $\pi$ be the hyperplane tangent to the ball $B$ closest to the point $x$. Let $H$ denote the half space in the complement of $\pi$ that contains $x$, and let $\vec{n}$ be the normal vector to $\pi$. Let $\tau_{H} = \inf\{t \ge 0: X_t \notin H\}$, and note that since $B \subset H$,
$$
\P( \tau_B>s) \le \P( \tau_H >s).
$$
We will show that
\begin{equation}\label{Tplane}
\P( \tau_H > s) \lesssim r/ m.
\end{equation} Consider first the case where $r = \dist (x, \partial B) \ge \log m$. Then, using a KMT approximation (strong coupling with Brownian motion), see Theorem 7.1.1 in \cite{LL}, there is a $d$-dimensional Brownian motion $(W^d_t, t \ge 0)$ such that if $s \asymp m^2$ then for some constant $C>0$,
$$\sup_{t \le s} |W_t - X_t | \le C \log m
$$
on an event $E_1$ of probability at least $ 1- m^{-1}$.
Let $$\text{osc} = \sup_{0 \le t \le s} |W^d_t - W^d_{\lfloor t \rfloor}|$$
and note that on an event $E_2$ of probability at least $ 1- O(m) e^{ - (\log m)^2/2} \ge 1- 1/m$,
$$
\text{osc} \le \log m.
$$
Let $\pi^-$ denote a hyperplane parallel to $\pi$ at distance $C\log m$ from $\pi$ such that $\pi^-$ in the half space which does not contain $B$. Let $T_{\pi^-}$ denote the first (continuous) time when $W^d$ hits $\pi^-$. Then on $E_1 \cap E_2$, $\tau_H > s$ implies $T_{\pi^-} > s$ and hence
$$
\P_x( \tau_H>s ) \le \P_x( T_{\pi^-} >s) + \P_x( E_1^c) + \P_x( E_2^c) \le 2/m + \P_x( T^- >s).
$$
Using rotational invariance of Brownian motion and projecting onto $\vec{n}$, letting $W= W^1$ be a standard one-dimensional Brownian motion and using the reflection principle,
$$
\P_x(T_{\pi^-} >s) = \P_{r + C \log m} (W_t \ge 0 \text{ for all } 0\le t \le s)\asymp \frac{r+ C \log m}{m}.
$$
Since we assumed initially that $ r \ge \log m$ we see that the right hand side above is $\asymp r/m$ and so this proves the upper bound in this case.

Now suppose that $r \le \log m$. Let $\pi^+$ denote another hyperplane parallel to $\pi$, also at distance $\log m$ from $\pi$ but such that $\pi^+$ is contained in the half space containing $B$ (and in particular, $\pi^+$ intersects $B$). Let $S$ be the slab comprised between $\pi $ and $\pi^+$, and let $\tau_S = \inf \{ t \ge 0: X_t \notin S\}$ denote the first time the walk leaves the slab.
Then note that if the walk has left $S$ by time $s/2$, it must do so by hitting $\pi^+$ before $\pi$ and must then avoid $\pi$ for at least $s/2$ units of time. Of course, $s/2 \asymp m^2$ hence using the result in the case already proved that $r \ge \log m$, since $\pi^+$ is at distance $\log m$ from $\pi$,
\begin{equation}\label{ballstrip}
\P_x( \tau_H >s) \lesssim \P_x( \tau_S > s/2) + \P_x( \tau_{\pi^+} < \tau_H) \frac{\log m}{m}.
\end{equation}
Let us bound the first term in the right hand side above. Note that the slab $S$ has a width equal to $\log m$ by definition. Hence every $(\log m)^2$ units of time, the walk has a probability bounded below by $p>0$ to exit $S$. We deduce that
$$
\P_x( \tau_S > s/2) \le (1-p)^{\lfloor (s/2)/(\log m)^2\rfloor }
$$
and since $s \asymp m^2$ we see that this decays faster than any polynomial in $m$ and hence in particular is $O( 1/m)$. Moreover, we claim that
\begin{equation}\label{Tstrip}
  \P_x( \tau_{\pi^+} < \tau_H)  \asymp \frac{r}{\log m}
\end{equation}
so that, combining with \eqref{ballstrip}, we get $\P( \tau_H >s) \lesssim  r/m$ as desired. To see \eqref{Tstrip}, let $x'$ be the point at which $\pi$ is tangent to $B$ and consider the martingale
$$
M'_t = (X_t - x') \cdot \vec{n},
$$
that is to say, the (signed) distance to the plane $\pi$ of the walk $X_t$. Note that this is indeed a martingale since $X$ is a martingale and the projection onto $\vec{n}$ is a linear operation. Apply the optional stopping theorem at the time $\tau_S = \tau_H \wedge \tau_{\pi^+}$. If $\tau_S = \tau_H$ then $|M'_{\tau_S} | \le 1 $, whereas if $\tau_S = \tau_{\pi^+}$ then $|M_{\tau_S}-  \log m |\le 1$. Consequently,
$$
r = O(1) (1- \P_x (\tau_{\pi^+} < \tau_H)) + (\log m  + O(1) ) \P_x( \tau_{\pi^+} < \tau_H)
$$
from which it immediately follows that
$$
\P_x( \tau_{\pi^+} < \tau_H) \le \frac{r + O(1)}{\log m + O(1)}
$$
which proves \eqref{Tstrip}. As explained, the lemma follows.
\end{proof}

We will also need a slight improvement of this estimate where the end point is specified. (This would also follow from Proposition 6.9.4. in \cite{LL} but as above we prefer to provide our own proof).

\begin{lemma}\label{L:stayinball_bridge} In the same setting as Lemma \ref{L:stayinball},
  We have
  $$\P_x ( X[0, t] \subset B; X_t = z) \lesssim t^{ - d/2} \frac{\dist(x, \partial B)}{\sqrt{t}} \frac{\dist(z, \partial B)}{\sqrt{t}}.$$
  Furthermore if $z \in B^\circ$ and $t \asymp m^2$ is such that $\P_x( X_t = z )>0$, the same inequality holds with $\lesssim$ replaced by $\gtrsim$.
\end{lemma}

\begin{proof}
  This uses a simple time reversal argument as well as the Markov property. Split the interval $[0,t]$ into three intervals of length $t /3$ each (for this argument we can assume without loss of generality that $t/3$ is an integer). Observe that the process $(X_{t-s} : 0 \le s \le t )$ is also by a reversibility a random walk which given $X_t =z$ will be starting from $z$. Hence, using the standard fact that $\P_{x'}(X_t = z') \lesssim t^{-d/2}$ for all $t \ge 1$ and $x', z' \in \mathbb{Z}^d$,
  \begin{align*}
    \P_x ( X[0, t] \subset B; X_t = z)  & \le \sum_{x',z'\in B} \P_x( X[ 0 , t/3] \subset B; X_{t/3} = x') \P_{x'} (X[0, t /3] \subset B; X_{t/3} = z')\times \\
    & \quad \quad \quad \quad \quad \quad \times \P_{z'} (X_{t/3} = z; X[ 0, t/3] \subset B)\\
    & \le   \sum_{x',z'\in B} \P_x( X[ 0 , t/3] \subset B; X_{t/3} = x') \P_{x'} (X[0, t /3] \subset B; X_{t/3} = z')\times \\
    & \quad \quad \quad \quad \quad \quad \times \P_{z} (X_{t/3} = z'; X[ 0, t/3] \subset B)\\
    & \lesssim \sum_{x',z'\in B} \P_x( X[ 0 , t/3] \subset B; X_{t/3} = x')  \times \frac1{t^{d/2}} \times \P_{z} (X_{t/3} = z'; X[ 0, t/3] \subset B)\\
    & \le t^{ - d/2} \P_x (X[0, t/3] \subset B) \P_z (X[0, t / 3] \subset B)\\
    & \lesssim t^{ - d/2} \frac{\dist(x, \partial B)}{\sqrt{t}} \frac{\dist(z, \partial B)}{\sqrt{t}}
  \end{align*}
  by \eqref{Tplane}, as desired. When $z \in B^\circ$ and $t \asymp m^2$, the opposite inequality holds using the lower bound in Lemma \ref{L:stayinball}, and the fact that a Bronwian bridge from a point in $B(0,1/2)$ to another point in $B(0,1/2)$ has a probability bounded below to stay in $B(0,1)$ throughout, uniformly in the endpoints of the trajectory within $B(0,1/2)$.
  \end{proof}

The estimate we will rely on is the following:

\begin{lemma}
  \label{L:HKexc}
  Let $t \ge 2$ and $x \in B$ with $\dist (x, \partial B) =r$ with $1 \le r \le m$. Then if $\tau = m^2$, uniformly in $a, b \in B^\circ$, and $s \in [\tau/4 ; \tau/2]$,
  \begin{equation}
    \label{HKlb}
    \P^{a \to b; \tau} (X_s = x) \gtrsim \frac1{m^d} \left( \frac{r}{m}\right)^2
  \end{equation}
if the parity of $x - a$ is the same as $s$.  Moreover, in that case, uniformly in $a, b \in B^\circ$, and $s \in [\tau/4 ; \tau/2]$, $1 \le t \le \tau/4$, $ 1\le r \le m$, the following holds: let $A = \{x \in B : r \le \dist(x, \partial B) < 2r \}$ be the annulus at distance $r$. Then for any $x \in A$,

  \begin{equation}\label{HKub}
   \P^{a \to b; \tau} (X_{s+ t} \in A | X_s = x)  \lesssim r^2  (\log t )^{10d} / t.
  \end{equation}
\end{lemma}

\begin{proof}
Observe that, uniformly in $a, b \in B^\circ$,
$$\P_a(X_\tau = b, X[ 0 ,\tau] \subset B) \asymp \P_a( X_\tau = b ) \asymp \frac1{m^d}$$
since a Brownian bridge from one point in a ball to another point in a ball has positive probability to stay in that ball throughout.
Thus using Lemma \ref{L:stayinball_bridge} and reversibility (and noting that the parity of $b - x$ is the same as $\tau -s$ under our assumptions),
\begin{align*}
\P^{a \to b; \tau} (X_s = x) &\asymp \frac{\P_a (X_s = x; X[0, s] \subset B) \times \P_x (X_{\tau - s} = b ; X[ 0, \tau - s] \subset B)}{1/m^d}\\
& \gtrsim \frac{\P_x (X_s = a; X[0,s] \in B) \times \tfrac{1}{m^d} \tfrac{r}{m}}{1/m^d}\\
& \gtrsim \frac{\tfrac{1}{m^d}{ \tfrac{r}{m} \times  \tfrac{1}{m^d} \tfrac{r}{m}}}{1/m^d} = \tfrac{1}{m^d} (\tfrac{r}{m})^2
\end{align*}
which proves \eqref{HKlb}.

  For the upper bound \eqref{HKub}, observe that by the Markov property, if $\tau' = \tau -s $,
  $$
  \P^{a \to b; \tau}(X_{s+t} \in A  |X_s=x) = \P^{x \to b; \tau'} (X_t \in A ) = \frac{\P_x( X_t \in A ; X[0, \tau'] \subset B ; X_{\tau'} = b)}{\P_x ( X[0, \tau'] \subset B; X_{\tau'} = b)}.
  $$
  Consider first the denominator. Observe that by Lemma \ref{L:stayinball_bridge}, since $b \in B^\circ$ and $\tau'\asymp m^2$,
  $$
  \P_x ( X[0, \tau'] \subset B; X_{\tau'} = b) \asymp \frac{r}{m} \frac1{m^d}.
  $$

  Now consider the numerator and suppose we condition on $X_t = y \in A$.
  Since $s \in [\tau/4, \tau/2]$ we have $\tau' = \tau - s \in [\tau/2, 3\tau/4]$. Hence since we also assume that $t \le \tau/4$ we deduce that $\tau' - t \in [\tau/4, \tau/2]$, and hence $\tau' - t \asymp \tau' \asymp \tau$. From this, it is not hard to see that
 $$
 \P_y( X([0, \tau' - t] \subset B; X_{\tau' - t} = b) \asymp \frac{r}{m}\frac1{m^d}.
 $$
 Consequently the numerator satisfies
 \begin{align*}
 \P_x( X_t \in A ; X[0, \tau'] \subset B ; X_{\tau'} = b) &\lesssim \frac{r}{m}\frac1{m^d} \P_x(X_t \in A ; X[0,t] \in B)\lesssim  \frac{r}{m} \frac1{m^d} \sum_{ y \in A} \P_x( X_t = y, X[0, t] \in B).
 \end{align*}
 Now, if $| y - x | \le t^{1/2} ( \log t )^{10}$ then we use Lemma \ref{L:stayinball} to obtain that
 $$
 \P_x( X_t = y, X[0, t] \in B) \lesssim \frac{r^2}{t} t^{ - d/2},
 $$
and since $\#\{ y : | y - x | \le t^{1/2} ( \log t )^{10} \} \lesssim t^{d/2} ( \log t )^{10d}$, the contribution to the sum from such points is at most
$$
 \sum_{ y \in A: | y - x| \le ( \log t )^{10}} \P_x( X_t = y, X[0, t] \in B) \lesssim  \frac{r^2}{t} ( \log t )^{10}.
$$
Now if $ | y - x | \le t^{1/2} ( \log t )^{10}$, then
$$
\P_x( X_t = y, X[0, t] \in B) \le \P_x (X_t = y )
$$
and so summing over all such $y$, the contribution is at most
$$
\P_x (|X_t - x | \ge t^{1/2} (\log t )^{10}) \le \exp( - ( \log t)^{20}/2) = o(1/t) \lesssim r^2/t
$$
using elementary Chernoff large deviation bounds for sums of i.i.d. random variables. Hence
 $$
 \sum_{ y \in A} \P_x( X_t = y, X[0, t] \in B) \lesssim \frac{r^2( \log t )^{10d}}{t}
 $$
 and so the numerator satisfies
 $$
 \P_x( X_t \in A ; X[0, \tau'] \subset B ; X_{\tau'} = b) \lesssim \frac{r}{m} \frac1{m^d} \frac{r^2( \log t )^{10d}}{t}.
 $$
 Combining with the bound on the denominator, we deduce
 $$
   \P^{a \to b; \tau}(X_{s+t} \in A  |X_s=x) \le \frac{r^2( \log t )^{10d}}{t}
 $$
 as desired.
\end{proof}

\subsection{End of proof of Theorem \ref{T:filledball}}
\label{S:moment}

We now start the proof of Lemma \ref{L:badmiss}. Decompose the ball $B = B(0,m)$ into dyadic annuli $$
A_j = \{ y \in B :{\dist (y, \partial B)} \in [2^j; 2^{j+1})\}
; \quad 0 \le j \le \log_2(m) $$
Given our point configuration $\cX$ of $k$ disjoint points in $B$, let $\cX_j = \cX \cap A_j$; let $j$ be such that $|\cX_j|$ is maximal. Then note that
$$
|\cX_j| \gtrsim \frac{k}{\log m}.
$$
We will show that any bridge of duration $\tau = m^2$ has a probability bounded below uniformly in its endpoints $a,b$ to visit $\cX_j$. Given a bridge $X$ of duration $\tau$, let $$L (\cX_j) = \int_{\tau/4}^{\tau/2} 1_{X_s \in \cX_j} ds $$
denote its local time spent in $\cX_j$ during its second quarter.

\begin{lemma}
  \label{L:exchit}
  Uniformly over $a, b \in B^\circ$,
  $$
  \P^{a \to b; \tau} ( L ( \cX_j) >0) \gtrsim \frac{k}{ (\log m)^{2 + 10d} m^d}.
  $$
\end{lemma}

\begin{proof}
  We use the trivial identity for nonnegative random variables:
  $$
  \P( X>0) = \frac{ \E(X)}{\E(X|X>0)}
  $$
  Hence we need a lower bound on $\E^{a \to b; \tau} ( L ( \cX_j) )$ and an upper bound on $\E^{a \to b; \tau} ( L ( \cX_j) | L(\cX_j) >0 )$.

  We start by the lower bound. Using \eqref{HKlb},
  \begin{align*}
  \E^{a \to b; \tau} ( L ( \cX_j) ) & \gtrsim |\cX_j| (\tau/4) \frac1{m^d} ( \frac{r}{m})^2\\
  & \gtrsim \frac{k}{\log m} \frac{1}{m^{d-2}} ( \frac{r}{m})^2
  \end{align*}
  where $r = 2^j$.
  On the other hand, for the upper bound we can apply the Markov property at the first hitting of $\cX_j$ and \eqref{HKub} to deduce that
  \begin{align*}
   \E^{a \to b; \tau} ( L ( \cX_j) | L(\cX_j) >0 ) & \le \sup_{s \in [\tau/4,\tau/2]}\sup_{x \in \cX_{j}} \sum_{y \in \cX_j}\int_0^{ \tau/4}\P^{a \to b ; \tau}( X_{s+t} = y | X_s=x) dt\\
   & \lesssim \sup_{s \in [\tau/4,\tau/2]} \sup_{x \in \cX_{j}} \int_0^{ \tau/4} \P^{a \to b; \tau} (X_{s+t} \in A_j | X_s = x) dt\\
   & \lesssim 1+ \int_2^{ \tau / 4} \frac{r^2 ( \log  t)^{10 d}}{t} dt \lesssim  r^2 ( \log m)^{1+ 10 d}.
  \end{align*}
 Taking the quotient of these two terms, we deduce
  $$
   \P^{a \to b; \tau} ( L ( \cX_j) >0) \gtrsim \frac{k}{ (\log m)^{2+ 10 d} m^d}
  $$
  as desired.
\end{proof}

We are now able to complete the proof of Lemma \ref{L:badmiss}.

\begin{proof}{Proof of Lemma \ref{L:badmiss}.}
On $\cG_{br}$ there are at least $c n^{2- 2 \kappa} m^{d-2}$ bridges of length $m^2$, by definition of this event, where $c$ depends only on the dimension. When we condition on the endpoints of the bridges they are independent of each other and of anything else under $\P^*$. Hence, using Lemma \ref{L:exchit}, and the inequality $1- x \le e^{-x}$ valid for all $x \in \R$,
\begin{align*}
\P^* ( \cG_{br}; \cB_{\cX}) &\le (1-\tfrac{k}{ (\log m)^{2+ 10d} m^d})^{c n^{2- 2\kappa} m^{d-2}}\\
& \le \exp ( - c \tfrac{k}{m^d (\log m)^{2+ 10d}} n^{2 - 2\kappa} m^{d-2} )\\
& \le \exp ( - c \tfrac{n^{2 - 2 \kappa}}{m^2 (\log m)^{2+ 10d}} k).
\end{align*}
Since on this event $|R_N \cap B| = |B|- k$, Lemma \ref{L:badmiss} follows.
\end{proof}

As explained this also finishes the proof of Theorem \ref{T:filledball}.

\appendix
\section{Functional inequalities on the rescaled lattice}
Let $x\in\Rd$ and $r>0$. The cubic box
centered at~$x$ of side length~$r$ is
$$\Lambda(x,r) = \big\{\,y=(y_1,\cdots,y_d)\in\R^d: -r/2<y_i-x_i\leq r/2,
\,1\leq i\leq d\,\big\}\,.$$
Let $n\geq 1$.
We denote by $\Znd$ the lattice $\Zd$ rescaled by a factor $n$:
$$\Znd\,=\,\frac{1}{n}\Z^d\,.$$
Let $x\in\Znd$ and $r$ a rational number
such that $nr$ is an odd integer.
We define the discrete box $\La_n(x,r)$ as
$$\La_n(x,r)\,=\,\La(x,r)\cap\Znd\,.$$
We rewrite first the general discrete Poincar\'e--Sobolev inequality proved by
Bes\-se\-mou\-lin--Chatard, Chainais--Hillairet and Filbet
(see theorem~$3$ of \cite{BCHF})
in the particular case
of a box and a cubic mesh and for the exponent
$$2^*\,=\,\frac{2d}{d-2}\,.$$
\begin{theorem}
\label{PS}
[Discrete Poincar\'e--Sobolev inequality]
Let $f$ be a function from $\La_n(x,r)$ to $\R$.
There exists a constant $c_{PS}=c_{PS}(x,r,d)$ which depends on $x,r,d$ such that
$$\displaylines{
\Big(\sum_{y\in\Lambda_n(x,r)}\frac{1}{n^d}|f(y)|^{2^*}\Big)^{\frac{1}{2^*}}
\,\leq\,\hfill\cr
c_{PS}\Big(\sum_{y\in\Lambda_n(x,r)}\frac{1}{n^d}|f(y)|^{2}\Big)^{\frac{1}{2}}
+c_{PS}\Big(
\sum_{\tatop{y,z\in\Lambda_n(x,r)}{|y-z|=1/n}}
\frac{1}{n^{d-2}}\big(f(y)-f(z)\big)^{2}\Big)^{\frac{1}{2}}\,.
}$$
\end{theorem}
We rewrite now the discrete Poincar\'e--Wirtinger inequality proved by
Bes\-se\-mou\-lin--Chatard, Chainais--Hillairet and Filbet
(see theorem~$5$ of \cite{BCHF})
in the particular case
of a box and a cubic mesh and for the exponent~$2$.
\begin{theorem}
\label{PW}
[Discrete Poincar\'e--Wirtinger inequality]
Let $f$ be a function from $\La_n(x,r)$ to $\R$.
There exists a constant $c_{PW}=c_{PW}(x,r,d)$ which depends on $x,r,d$ such that
$$\displaylines{
\Big(\sum_{y\in\Lambda_n(x,r)}\frac{1}{n^d}\big(f(y)-\baf\big)^{2}\Big)^{\frac{1}{2}}
\,\leq\,
c_{PW}\Big(
\sum_{\tatop{y,z\in\Lambda_n(x,r)}{|y-z|=1/n}}
\frac{1}{n^{d-2}}\big(f(y)-f(z)\big)^{2}\Big)^{\frac{1}{2}}\,,
}$$
where $$\baf\,=\,
\frac{1}{(rn)^d}
\sum_{y\in\Lambda_n(x,r)}
f(y)\,.$$
\end{theorem}
In general, the constants
$c_{PS}$,
$c_{PW}$
depend on both $x\in\Znd$ and $r$. However, the lattices $\Znd$ being invariant
under a translation by an element of $\Znd$, these constants are the same for all points
$x\in\Znd$. From now onwards, we suppose that $x=0$. Let us examine the dependence
of the constants
$c_{PS}$,
$c_{PW}$ with respect to~$r$.
Let $f$ be a function defined on $\La_n(x,r)$ with values in $\R$ and let us set
$$\forall x\in\La_{nr}(0,1)\qquad g(x)\,=\,f(rx)\,.$$
We apply the inequality stated in theorem~\ref{PW} to $g$:
$$\displaylines{
\Big(\sum_{y\in\Lambda_{nr}(0,1)}\frac{1}{(nr)^d}\big(g(y)-\bag\big)^{2}\Big)^{\frac{1}{2}}
\,\leq\,
c_{PW}(0,1,d)\Big(
\sum_{\tatop{y,z\in\Lambda_{nr}(0,1)}{|y-z|=1/(nr)}}
\frac{1}{(nr)^{d-2}}\big(g(y)-g(z)\big)^{2}\Big)^{\frac{1}{2}}\,,
}$$
and we rewrite everything in terms of the function~$f$:
$$\bag\,=\,
\frac{1}{(nr)^d}
\sum_{y\in\Lambda_{nr}(0,1)}
f(ry)\,=\,
\frac{1}{(nr)^d}
\sum_{y\in\Lambda_{n}(0,r)}
f(y)\,=\,\baf\,,
$$
$$\sum_{y\in\Lambda_{nr}(0,1)}\big(g(y)-\bag\big)^{2}\,=\,
\sum_{y\in\Lambda_{n}(0,r)}\big(f(y)-\baf\big)^{2}\,,$$
$$\sum_{\tatop{y,z\in\Lambda_{nr}(0,1)}{|y-z|=1/(nr)}}
\big(g(y)-g(z)\big)^{2}\,=\,
\sum_{\tatop{y,z\in\Lambda_{n}(0,r)}{|y-z|=1/n}}
\big(f(y)-f(z)\big)^{2}\,.
$$
We obtain the following inequality for the function~$f$:
$$\displaylines{
\Big(\sum_{y\in\Lambda_{n}(0,r)}\frac{1}{n^d}\big(f(y)-\baf\big)^{2}\Big)^{\frac{1}{2}}
\,\leq\,
c_{PW}(0,1,d)\,r\,\Big(
\sum_{\tatop{y,z\in\Lambda_{n}(0,r)}{|y-z|=1/n}}
\frac{1}{n^{d-2}}\big(f(y)-f(z)\big)^{2}\Big)^{\frac{1}{2}}\,.
}$$
We conclude that
$c_{PW}(0,r,d)=
c_{PW}(0,1,d)r$.
\section{Inequalities on the original lattice}
We shall adopt a slightly different viewpoint to apply these inequalities. Instead
of rescaling the lattice by a factor~$n$, we will consider functions defined on the
lattice $\Zd$ but on boxes of side~$n$.
We shall deduce the relevant inequalities
from the previous ones by a simple change of variables $y\to ny$.
More precisely, let $x\in\Zd$, $n\geq 1$ and
let $f$ be a function from $\La(x,n)\cap\Zd$ to $\R$.
Let
$\La_n\big({x}/{n},1\big)$ be the box
$$\La_n\big({x}/{n},1\big)\,=\,\La(x/n,1)\cap\Znd\,.$$
We define a function
$f_n$ on
$\La_n\big({x}/{n},1\big)$
by setting
$$\forall y
\in\La_n\big({x}/{n},1\big)
\qquad f_n(y)\,=\,f(ny)\,.$$
We apply then the inequalities stated in theorems~\ref{PS}, \ref{PW} to the function~$f_n$
and we rewrite everything in terms of~$f$.
We first introduce some notation before stating the inequalities.
Let $f$ be a function defined on a subset~$D$
of $\Zd$ with values in~$\R$.
For $p\geq 1$, we define its $p$--norm over $D$
$$||f||_{p,D}\,=\,
\Big(\sum_{y\in D} |f(y)|^p\Big)^{\frac{1}{p}}$$
and its Dirichlet energy over $D$
$$\cE(f,D)\,=\,
\frac{1}{2d}
\sum_{\tatop{y,z\in D}{|y-z|=1}}
\big(f(y)-f(z)\big)^2\,.$$
We recall that the exponent $2^*$ is equal to $2^*=2d/(d-2)$.
The Poincar\'e--Sobolev inequality stated in theorem~\ref{PS}
yields the following inequality in a box of side~$n$.
\begin{corollary}
\label{UPS}
Let $x\in\Zd$, $n\geq 1$ and
let $f$ be a function from $\La(x,n)\cap\Zd$ to $\R$.
There exists a constant $c_{PS}$ which depends on the dimension $d$ only such that
$$||f||_{2^*,\La(x,n)}
\,\leq\,c_{PS}\Big(\frac{1}{n}
||f||_{2,\La(x,n)}+\sqrt{2d\,\cE(f,\La(x,n))}\Big)\,.$$
\end{corollary}
The Poincar\'e--Wirtinger inequality stated in theorem~\ref{PW}
yields the following inequality in a box of side~$n$.
\begin{corollary}
\label{UPW}
Let $x\in\Zd$, $n\geq 1$ and
let $f$ be a function from $\La(x,n)\cap\Zd$ to $\R$.
There exists a constant $c_{PW}$ which depends on the dimension $d$ only such that
$$||f-\baf||_{2,\La(x,n)}
\,\leq\,n\,c_{PW}
\sqrt{2d\,\cE(f,\La(x,n))}\,,$$
where $$\baf\,=\,
\frac{1}{n^d}
\sum_{y\in\Lambda(x,n)}
f(y)\,.$$
\end{corollary}
Finally, if we send $n$ to $\infty$ in the inequality of corollary~\ref{UPS},
we get the following result.
\begin{corollary}
\label{TPS}
Let $f$ be a function defined on $\Zd$ with values in $\R$ having finite support.
There exists a constant $c_{PS}$ which depends on the dimension $d$ only such that
$$||f||_{2^*}
\,\leq\,c_{PS}
\sqrt{2d\,\cE(f)}\,.$$
\end{corollary}

\noindent
{\bf Acknowledgements.} We warmly thank the Referees for their precise reading and their numerous comments which helped
to improve the paper.
\bibliographystyle{plain}
\bibliography{bobe}
 \thispagestyle{empty}

\end{document}